\documentclass[11pt]{article}
\usepackage[margin=1in]{geometry}
\usepackage{tikz}
\usetikzlibrary{fit,calc,positioning,decorations.pathreplacing,matrix}
\usepackage{lastpage}
\usetikzlibrary{arrows,positioning,shapes,matrix,backgrounds}
\usepackage{float}
\usepackage{bm}
\usepackage{multirow}
\usepackage{tabularx}
\usepackage{pdfpages}
\usepackage{xypic}
\usepackage{amssymb}
\usepackage{amsthm}
\usepackage{enumerate}
\usepackage{url}
\usepackage{amsmath}
\usepackage[tiny,compact]{titlesec}
\usepackage{graphicx}
\usepackage{marvosym}
\usepackage{bbm}
\usepackage{fancybox}
\usepackage{fancyheadings}
\usepackage[section]{placeins}
\usepackage{setspace}
\usepackage{pdflscape}
\usepackage{mathrsfs,euscript}
\usepackage[toc]{multitoc}
\usepackage[backgroundcolor=green!40, linecolor=green!40]{todonotes}
\usepackage{booktabs}
\usepackage{lipsum}
\usepackage[utf8]{inputenc}
\usepackage{mathscinet}
\usepackage{upgreek}
\usepackage{authblk}

\newcolumntype{Q}{>{$\displaystyle}c<{$}} 
\newcolumntype{A}{>{$}c<{$}}

\numberwithin{equation}{section} 
\newtheorem{theorem}[equation]{Theorem}
\newtheorem{proposition}[equation]{Proposition}

\newtheorem{lemma}[equation]{Lemma}
\newtheorem{corollary}[equation]{Corollary}
\theoremstyle{definition}

\newtheorem{definition}[equation]{Definition}
\newtheorem{remark}[equation]{Remark}

\def\AA{\mathbf{A}}
\def\CC{\mathbf{C}}

\def\QQ{\mathbf{Q}}

\def\ZZ{\mathbf{Z}} 
\def\FF{\mathbf{F}}

\def\A{{\mathrm A}}
\def\B{{\mathrm B}}
\def\C{{\mathrm C}}

\def\E{{\mathrm E}}
\def\F{{\mathrm F}}
\def\G{{\mathrm G}}
\def\H{{\mathrm H}}
\def\I{{\mathrm I}}
\def\J{{\mathrm J}}
\def\K{{\mathrm K}}
\def\L{{\mathrm L}}
\def\M{{\mathrm M}}
\def\N{{\mathrm N}}
\def\P{{\mathrm P}}
\def\Q{{\mathrm Q}}
\def\R{{\mathrm R}}
\def\SS{{\mathrm S}}
\def\T{{\mathrm T}}
\def\U{{\mathrm U}}
\def\V{{\mathrm V}}
\def\W{{\mathrm W}}
\def\X{{\mathrm X}}
\def\Y{{\mathrm Y}}
\def\Z{{\mathrm Z}}

\def\Aa{\mathscr{A}}

\def\Cc{\EuScript{C}}

\def\Ee{\EuScript{E}}

\def\Hh{\mathscr{H}}

\def\Jj{\mathcal{J}}

\def\Ll{\mathscr{L}}
\def\Mm{\mathcal{M}}
\def\Nn{\mathcal{N}}

\def\Ss{\mathcal{S}}

\def\Uu{\mathcal{U}}
\def\Vv{\mathscr{V}}
\def\Ww{\mathcal{W}}

\def\Ga{\Gamma}
\def\La{\Lambda}

\def\a{\alpha} 
\def\b{\beta}
\def\d{\delta}
\def\e{\epsilon}

\def\g{\gamma}
\def\h{\varphi}

\def\l{\lambda}

\def\o{\EuScript{O}}
\def\p{\mathfrak{p}}
\def\s{\sigma}
\def\t{\theta}
\def\w{\varpi}
\def\om{\omega}

\def\>{\geqslant}
\def\<{\leqslant}

\def\mult#1{{#1}^{\times}}

\def\Gal{\operatorname{Gal}}
\def\GL{\operatorname{GL}}

\def\supp{\operatorname{supp}}
\def\diag{\operatorname{diag}}

\def\Hom{\operatorname{Hom}}
\def\Mat{\operatorname{M}}

\def\I{\operatorname{I}}

\def\tr{\operatorname{tr}}
\def\Nm{\operatorname{N}}

\def\Res{\operatorname{Res}}
\def\ind{\operatorname{ind}}
\def\Ind{\operatorname{Ind}}
\def\cind{\operatorname{ind}}
\def\Aut{\operatorname{Aut}}
\def\End{\operatorname{End}}

\def\As{{\rm As}}

\def\kk{\boldsymbol{k}}
\def\ee{\boldsymbol{k}_\E}

\def\be{{\mathbf{e}}}

\def\bl{{\boldsymbol{\l}}}

\def\bJ{{\boldsymbol{\J}}}
\def\TT{\boldsymbol{\Theta}}
\def\0{\boldsymbol{0}}

\def\aa{\mathfrak{a}}

\def\bb{\mathfrak{b}}

\def\pp{\mathfrak{p}}

\def\WW{\W}
\def\LL{\mathscr{L}}
\def\GG{\EuScript{G}}

\def\CCC{{\sf c}}

\def\CH{\textsf{H}}
\def\FC{{\mathrm R}}
\def\WW{\EuScript{W}}

\def\FK{{}}
\def\LS{\mathrm{LS}}
\def\RS{\mathrm{RS}}

\def\GGG{{\mathsf G}}
\def\HHH{{\mathsf H}}
\def\KKK{{\mathsf K}}
\def\one{\mathbf{1}}

\def\etavec{\uptau}

\def\ge{\geqslant}
\def\le{\leqslant}
\def\({\left(}
\def\){\right)}

\def\so{{\mathsf{o}}} 
\def\ef{\omega_{\F/\F_\so}}     

\def\pp{{\rm p}}
\def\ttt{\rho}
\def\RRR{{\mathsf R}}

\def\psiu{\psi_{\U}}

\def\ignore#1{\relax}

\def\presuper#1#2%
  {\mathop{}%
   \mathopen{\vphantom{#2}}^{#1}%
   \kern-\scriptspace%
   #2}
   
\tikzset{
  every node/.style={scale=1}
}
\tikzset{
    >=latex,
    punkt/.style={
           rectangle,
           rounded corners,
           draw=black, 
           minimum height=1em,
           },
    pil/.style={
           ->,
           shorten <=10pt,
           shorten >=10pt,}
}
\tikzset{
  every node/.style={scale=1}
}

\makeatletter
\newcommand{\resetHeadWidth}{\fancy@setoffs}
\makeatother
\titleformat{\subsubsection}[runin]{\normalsize\bfseries}{\thesubsubsection}{5pt}{}
\titleformat{\subsection}{\normalsize\bfseries}{\thesubsection}{5pt}{}
\titleformat{\section}{\normalsize\bfseries\filcenter}{\thesection}{5pt}{}
\titlespacing{\section}{0pt}{*0}{*0}
\titlespacing{\subsection}{0pt}{*0}{*0}

\title{Galois self-dual cuspidal types and Asai local factors}

\author{U.~K.~Anandavardhanan, R.~Kurinczuk, N.~Matringe, 
V.~S{\'e}cherre and S.~Stevens}

\begin{document}
\maketitle
\thispagestyle{fancy}
\pagestyle{fancy}






\begin{abstract}
Let~$\F/\F_{\so}$ be a quadratic extension of non-archimedean locally compact
fields of odd residual characteristic and $\s$ be its non-trivial automorphism.
We show that any~$\s$-self-dual cuspidal~re\-pre\-sentation of~$\GL_n(\F)$ 
contains a~$\s$-self-dual Bushnell--Kutzko type. 
Using such a~type, we~cons\-truct {an explicit test vector} for Flicker's local Asai 
$\L$-function of a $\GL_n(\F_\so)$-distingui\-shed cus\-pidal representation and 
compute the associated Asai root number. 
Finally, by using global~me\-thods, we compare this root number to 
Langlands--Shahidi's local Asai root number,
{and more generally}
we compare the corresponding epsilon factors 
{for any cuspidal representation}.
\end{abstract}

\section{Introduction}

\subsection{}
\label{P11}

Let~$\F/\F_\so$ be a quadratic extension of locally compact non-archimedean 
fields of odd residual characteristic~$p$ and let $\sigma$ denote the
non-trivial element of~$\Gal(\F/\F_\so)$. 
Let~$\G$ denote the general linear group $\GL_n(\F)$,
make $\s$ act on $\G$ componentwise
and let $\G^\s$ be the $\s$-fixed points subgroup~$\GL_n(\F_\so)$. 
In \cite{Flicker}, using the Rankin--Selberg method, Flicker has associated 
Asai local factors to any generic ir\-reducible (smooth, complex)
representation of $\G$.
Let $\N$ denote the subgroup of upper triangu\-lar uni\-po\-tent matrices in~$\G$, 
and~$\psi$ be a non-degenerate character of~$\N$ trivial on 
$\N^\s=\N\cap\G^\s$. 
Given a {generic} irreducible representation of $\G$,
let~$\Ww(\pi,\psi)$ denote its Whittaker model with respect to the 
Whittaker datum
$(\N,\psi)$, that is the unique subrepresentation of 
the smooth induced represen\-tation $\Ind_\N^\G(\psi)$ which 
is isomorphic to $\pi$.
Given a function $\W\in \Ww(\pi,\psi)$, 
a smooth compactly~sup\-por\-ted com\-plex function~$\Phi$ on $\F_\so^n$
and a complex number $s\in\CC$,
the associated local Asai integral~is 
\begin{equation*}
\label{LAI}
\I_\As(s,\Phi,\W) = \int_{\N^\sigma\backslash \G^\sigma} \W(g)
\Phi((0\ \dots\ 0\ 1) g)|\det(g)|_\so^s\ dg,
\end{equation*}
where~$|\cdot|_\so$ is the normalized absolute value of $\F_\so$ and $dg$ 
is a right invariant measure on $\N^\s\backslash\G^\s$.~This 
inte\-gral~is convergent when the real part of $s$ is large enough,
and it is a rational function in $q_\so^{-s}$,
where $q_\so$~is the cardinality of the residual field of $\F_\so$. 
When one varies the functions $\W$ and $\Phi$,~these 
integrals generate a fractional ideal 
of $\CC[q_\so^s,q_\so^{-s}]$. 
The Asai $\L$-function~$\L_{\As}(s,\pi)$~of $\pi$ is defined as a generator, 
suitably normalized, of this fractional ideal. 
It does not depend on the choice of~the non-degenerate character $\psi$.

\subsection{}
\label{ourson}

Now consider a cuspidal (irreducible, smooth, complex) representation $\pi$ of 
$\G$, and suppose that~its Asai $\L$-function $\L_{\As}(s,\pi)$ is 
non-trivial. 
By~\cite{MatringeManuscripta}, 
this happens if and only if $\pi$ has a \textit{dis\-tin\-guished} unramified 
twist, that is, an unramified twist carrying a non-zero $\G^\s$-invariant 
linear form.
In this case, the Asai $\L$-function
can be described~ex\-pli\-citly (see Proposition~\ref{MatMan}).
We prove that it can~be realized as a single Asai integral:

\begin{theorem}[Theorem \ref{Testvectortheorem} and 
Corollary \ref{Testvectorcorollary}]
\label{Testvectorcorollaryintro}
Let $\pi$ be a cuspidal representation of $\G$ having a 
distinguished~un\-ra\-mified twist.
Then there is an explicit function $\W_0\in\Ww(\pi,\psi)$ such that 
\begin{equation*}
\I_\As(s,\Phi_0,\W_0) = \L_{\As}(s,\pi)
\end{equation*}
where $\Phi_0$ is the characteristic function of the lattice $\o_\so^n$ in 
$\F_\so^n$ and $\o_\so$ is the ring of integers of $\F_\so$.
\end{theorem}

This thus provides an integral formula for the Asai $\L$-function of $\pi$. 
As an application of this~theo\-rem, we com\-pu\-te the associated root 
number: 
the Asai $\L$-functions of $\pi$ and its contragredient $\pi^\vee$ are related 
by a functional equation \eqref{FE}, in which appears a local Asai epsilon factor 
$\e_{\As}(s,\pi,\psi_\so,\delta)$ 
depending on a non-trivial character $\psi_\so$ of $\F_\so$
and a non-zero scalar $\delta\in\F^\times$ such that $\tr_{\F/\F_\so}(\delta)=0$.
We prove the following theorem conjectured in
  \cite[Remark 4.4]{AnandRoot}.
  Such a theorem can be seen as~the twisted tensor analogue of \cite[Theorem 
  2]{BH99} in the cuspidal case and also as the Rankin--Selberg~coun\-terpart 
  of \cite[Theorem 1.1]{AnandRoot} in the cuspidal distinguished case.

\begin{theorem}[Theorem \ref{distinguished cuspidal Asairootnumber}]
\label{rootnumberintro}
Let $\pi$ be a distinguished cuspidal representation of $\G$.
Then:
\begin{equation*}
\e_{\As}\left(\frac{1}{2},\pi,\psi_\so,\delta\right) = 1.
\end{equation*}
\end{theorem}

Our proof of this theorem
is purely local and relies on Theorem \ref{Testvectorcorollaryintro}.

Independently, using a global argument, we compare the 
Asai epsilon factor $\e_{\As}(s,\pi,\psi_\so,\delta)$ 
with~the local Asai epsilon factor $\e_{\As}^{{\rm LS}}(s,\pi,\psi_\so)$ 
defined via the Langlands--Shahidi method. 

\begin{theorem}[Theorem \ref{thm cuspidal epsilon equality}]
\label{epsilonFKLSintro}
Let $\pi$ be a cuspidal representation of $\G$ with central character 
$\omega_{\pi}$. 
For any non-trivial character $\psi_\so$ of $\F_\so$
and non-zero scalar $\delta\in\F$ such that $\tr_{\F/\F_\so}(\delta)=0$, we have:
\begin{equation*}
\e_{\As}(s,\pi,\psi_\so,\delta) = 
\om_\pi(\d)^{n-1}\cdot
|\d|^{\frac{n(n-1)}{2} \left(s-\frac{1}{2}\right)}\cdot
\l(\F/\F_\so,\psi_\so)^{-\frac{n(n-1)}{2}}\cdot
\e_{\As}^{{\rm LS}}(s,\pi,\psi_\so) 
\end{equation*}
where $|\delta|$ is the normalized absolute value of $\delta$
and $\l(\F/\F_\so,\psi_\so)$ is the local Langlands constant. 
\end{theorem}

When in addition $\pi$ is distinguished,
Theorem \ref{epsilonFKLSintro}
and \cite[Theorem 1.1]{AnandRoot} together imply
Theorem~\ref{rootnumberintro},
which thus gives us another proof, using a global argument, 
of this theorem. 

\subsection{}

Let us now explain our strategy to prove Theorem 
\ref{Testvectorcorollaryintro}. 
The basic idea is to use Bushnell-Kutzko's theory of types \cite{BK}, 
which provides an explicit model for a cuspidal representation $\pi$ of $\G$ 
as~a~com\-pactly induced representation from an extended maximal simple type: 
a pair $(\bJ,\bl)$ consisting of~an irreducible representation $\bl$ of a 
compact mod centre, open subgroup $\bJ$ of $\G$, 
constructed via a~pre\-cise recipe, such that 
\begin{equation}
\label{CI}
\pi \simeq \ind^\G_{\bJ}(\bl).
\end{equation}
Such an extended maximal simple type $(\bJ,\bl)$ is unique up to 
$\G$-conjugacy, 
and we abbreviate by saying that $\pi$ contains the type $(\bJ,\bl)$ when 
\eqref{CI} holds.

Suppose now $\pi$ is distinguished.
By a result of Prasad and Flicker (\cite{PrasadDUKE,Flicker}), 
it is then \textit{$\sigma$-self-dual},
that is, its contragredient representation~$\pi^\vee$ is isomorphic to 
$\pi^{\sigma}=\pi\circ \sigma$.
In order to compute a test vector for $\pi$, 
{that is, the explicit function $\W_0$ of 
Theorem \ref{Testvectorcorollaryintro}},
our first task is to isolate a 
type, among those contained in $\pi$,
which behaves well with respect to $\s$.

Our first main result is the following.
For further use in another context (see \cite{VS}),
we state it and prove it for cuspidal representations of $\G$ with 
coefficients not necessarily in $\CC$,
but more generally in an algebraically closed field $\FC$ of characteristic 
different from $p$.
For Bushnell-Kutzko's theory in this more general context, in particular the 
description of cuspidal $\FC$-representations by compact induction of 
extended maximal simple types, see \cite{Vig96,MSt}.

\begin{theorem}[Theorem \ref{PIMAIN}]
\label{PIMAINintro}
Let $\pi$ be a cuspidal representation of $\G$ with coefficients in $\FC$.
Then $\pi$ is $\sigma$-self-dual if and only if it contains a 
$\sigma$-self-dual type, that is, a type $(\bJ,\bl)$ such that $\s(\bJ)=\bJ$ 
and $\lambda^\sigma\simeq\lambda^\vee$. 
\end{theorem}

Theorem \ref{PIMAINintro} generalizes \cite[Lemma 2.1]{MurnaghanRepkaTAMS99}, 
which deals with the case where $\pi$ is essentially tame, that is, 
the number of unramified 
characters $\chi$ of $\G$ such that $\pi\chi\simeq\pi$ is prime to $p$.

The proof relies on Bushnell--Henniart's theory of 
endo-classes and tame lifting~\cite{BHLTL1,BHEffectiveLC};
although they are technical in nature, it is both natural and
necessary to consider endo-classes since, in the case of representations that 
are not essentially tame, there are no simpler canonical parameters to 
construct types.
The assump\-tion that $p\neq2$ is cru\-cial here,
since we use at various places the fact that the first cohomology set of 
$\Gal(\F/\F_0)$ in a pro-$p$-group is trivial.

\subsection{}

In general, a $\s$-self-dual type as in Theorem \ref{PIMAINintro}
is not unique up to 
$\G^\s$-conjugacy: see Proposition \ref{classesdetypesstables7}.
To construct explicit Whittaker functions, 
we need to go further and isolate those $\sigma$-self-dual types which are 
compatible with the Whittaker model of $\pi$. 

Recall that we have fixed a Whittaker datum $(\N,\psi)$ in 
Paragraph~\ref{P11}. 
A type $(\bJ,\bl)$ contained in~a cuspidal representation of $\G$ is said to 
be \textit{generic} (with respect to $\psi$) if $\Hom_{\bJ\cap\N}(\bl,\psi)$ 
is non-zero. 

\begin{proposition}[Proposition \ref{brunau}]
Any $\s$-self-dual cuspidal representation of~$\G$ 
{with coefficients~in $\FC$} 
contains a generic $\s$-self-dual type. 
Such a type is uniquely determined up to $\N^\s$-conjugacy. 
\end{proposition} 

We then prove the following result.

\begin{proposition}[Corollary \ref{aloi}]
\label{aloiintro}
A $\s$-self-dual cuspidal representation $\pi$ of $\G$ 
{with coefficients in $\FC$} 
is distin\-gui\-shed if and only if any 
{generic $\s$-self-dual type $(\bJ,\bl)$ contained in 
$\pi$ is distinguished, that is, the space 
$\Hom_{\bJ\cap\G^\s}(\bl,\one)$ is non-zero.}
\end{proposition}

Our proof of Proposition \ref{aloiintro} is based on a result of Ok~\cite{Ok} 
{proved for any irreducible complex~re\-pre\-sentation of 
$\G$, and which we prove for any cuspidal representation of $\G$ with 
coefficients in $\FC$ in Appendix \ref{AppB}.
Note that Pro\-po\-sition \ref{aloiintro} is proved~in another way in 
\cite{VS}, without using Ok's result
(see Remark \ref{coincoin}).}

\subsection{}

For the remainder of the introduction we go back to \textit{complex} 
representations. 
Given a generic~type $(\bJ,\bl)$ in a cuspidal~representation~$\pi$ of~$\G$, 
a construction of Paskunas--Stevens~\cite{StPa} defines an~ex\-pli\-cit 
Whit\-taker 
function~$\W_{\bl}\in \Ww(\pi,\psi)$. 
The key point for this paper is that, 
if~$(\bJ,\bl)$ is both~ge\-ne\-ric and $\sigma$-self-dual, 
then~$\W_{\bl}$ is well suited to computing the local Asai integral. 
We make Theorem \ref{Testvectorcorollaryintro} more precise.

\begin{theorem}[Theorem \ref{Testvectortheorem}]
Let~$\pi$ be a~distinguished cuspidal representation of $\G$, 
and~$(\bJ,\bl)$~be a generic~$\sigma$-self-dual type contained in~$\pi$. 
Then there is a unique right invariant
measure on $\N^\s\backslash\G^\s$ such that
\[
\I_\As(s,\Phi_0,\W_{\bl})=\L_\As(s,\pi)
\]
where $\Phi_0$ is the characteristic function of the lattice $\o_\so^n$ in 
$\F_\so^n$. 
\end{theorem}

Let us briefly explain how we prove this theorem.
Following the method of \cite{RKNM}, we compute the local Asai integral 
and get
\begin{equation*}
\I_\As(s,\Phi_0,\W_{\bl}) = \frac{1}{1-q_{\so}^{-sn/e_{\so}}}
\end{equation*}
where $e_\so$ a positive integer attached to the generic $\s$-self-dual 
type $(\bJ,\bl)$ in Paragraph \ref{sec:inv}.~On~the 
other hand, starting from Proposition \ref{MatMan} giving a formula 
for $\L_{\As}(s,\pi)$, and using the dichotomy theorem 
(Paragraph \ref{Dichotomy49}) together with Proposition 
\ref{prop:lemma16}, 
we find the same expression for $\L_{\As}(s,\pi)$.

\subsection{}

We now explain how we prove Theorem \ref{rootnumberintro}.
First, thanks to the functional equation \eqref{FE} 
{together with the fact that $\pi$ is distinguished,}
the~Asai root number 
$\e_{\As}(1/2,\pi,\psi_\so,\delta)$ must be equal to either $1$ or $-1$.
Then, by using the explicit integral ex\-pres\-sion for $\L_{\As}(s,\pi)$
provided by our test vectors, 
we prove that $\e_{\As}(1/2,\pi,\psi_\so,\delta)$ is positive.

\subsection{}

We now explain our global argument for proving 
Theorem \ref{epsilonFKLSintro}.
We first prove that, 
for any generic irreducible representation $\pi$ of $\G$
and any $\d\in\F$ such that $\tr_{\F/\F_0}(\d)=0$,
the quantity
\begin{equation*}
\label{eRSintro}
\om_\pi(\d)^{1-n}\cdot|\d|^{-n(n-1)(s-1/2)/2}
\cdot\l(\F/\F_\so,\psi_\so)^{n(n-1)/2}\cdot\e_\As(s,\pi,\psi_\so,\d)
\end{equation*}
does not depend on $\d$.
When $\pi$ is in addition unramified,
we also prove that it is equal to the local Asai epsilon factor 
$\e_\As^\LS(s,\pi,\psi_\so)$ obtained via the 
Langlands--Shahidi method.
This leads us to:

\begin{definition}[Definition \ref{paulveyne}]
\label{paulveyneintro}
For any generic irreducible representation $\pi$ of $\G$, we set:
\[\e_\As^\RS(s,\pi,\psi_\so)=\om_\pi(\d)^{1-n}\cdot
|\d|^{-n(n-1)(s-1/2)/2}\cdot\l(\F/\F_\so,\psi_\so)^{n(n-1)/2}\cdot
\e_\As^\FK(s,\pi,\psi_\so,\d).\] 
\end{definition}

Beuzart-Plessis also came up to the same 
normalization in \cite{B-P18}. 

Now consider a quadratic extension $k/k_\so$ of global fields of 
characteristic different from $2$ such that 
any place of $k_\so$ dividing $2$,
as well as any archimedean place in the number field case,
splits in $k$.

Let $\Pi$ be a cuspidal automorphic representation of $\GL_n(\AA)$, 
where $\AA$ is the ring of adeles of~$k$,~with 
local component $\Pi_v$ for each place $v$ of $k_\so$.
We also fix a non-trivial character $\psi_\so$ of $\AA_\so$ trivial on $k_\so$, 
where $\AA_\so$ is the ring of adeles of $k_\so$,
and denote by $\psi_{\so,v}$ its local component~at~$v$.
We then set
\begin{equation*}
\e_\As^\RS(s,\Pi) =
\prod_{v} \e_\As^\RS(s,\Pi_v,\psi_{\so,v})
\end{equation*}
where the product is taken over all places $v$ of $k_\so$,
where $\e_\As^\RS(s,\Pi_v,\psi_{\so,v})$ is defined by 
Definition~\ref{paulveyneintro} when $v$ is inert,
and is the Jacquet--Piatetski-Shapiro--Shalika epsilon factor 
$\e_\As^\RS(s,\pi_1,\pi_2,\psi_{\so,v})$
when $v$ is split and $\Pi_{v}$ identifies with 
$\pi_1\otimes \pi_2$ as representations of 
$\GL_n(k_v)\simeq\GL_n(k_{\so,v})\times\GL_n(k_{\so,v})$.

We define the global factor $\e_\As^{\LS}(s,\Pi)$ similarly. 
Using the equality
{(that we prove when $\F$ has~cha\-rac\-teristic $p$ in 
Appendix \ref{section positive char})}
of the Flicker and Lang\-lands--Shahidi Asai 
$\L$-functions of $\Pi_v$ for all $v$,
the comparison of the global functional equations gives: 

\begin{theorem}[Theorem \ref{global equality}]\label{global equality intro}
Let $\Pi$ be a cuspidal automorphic representation of $\GL_n(\AA)$. 
Then 
\[\e_\As^{\RS}(s,\Pi)=\e_\As^{\LS}(s,\Pi).\]
\end{theorem}  

Realizing any cuspidal representation of $\G$ as a local component 
of some cuspidal automorphic~re\-presentation of $\GL_n(\AA)$
{with prescribed ramification at other places,}
and combining with Theo\-rems \ref{global equality intro} and 
\ref{thm unramified epsilon equality},
we get Theorem~\ref{epsilonFKLSintro}.

\subsection{}

Finally, we must explain the interconnection between \cite{VS} and the present 
paper. 
The starting~point of both papers is the $\s$-self-dual type theorem for 
cuspidal $\FC$-representations,
namely Theorem \ref{PIMAINintro},
which is proved in Section \ref{S2} below.

Starting from this theorem, 
and independently from the rest of this paper,
one gives in \cite{VS} a~ne\-ces\-sary and sufficient condition of distinction for 
$\s$-self-dual \textit{supercuspidal} $\FC$-representations.
In~par\-ticular, for complex representations, 
{in which case all cuspidal representations are supercuspidal,}
this implies the two results stated in Paragraph \ref{Dichotomy49} 
(i.e. Theorem \ref{dichotomy} and Proposition \ref{prop:lemma16})
which~we use in the proof of Theorem \ref{Testvectorcorollaryintro}.
{We also use Proposition \ref{polonius},
which is also proved in \cite{VS}, for any $\s$-self-dual {supercuspidal} 
$\FC$-representation of $\G$.} 

\section*{Acknowledgements}

We thank Raphaël Beuzart-Plessis for noticing and providing a correction to a 
gap in an argument of a previous version of 
Section~\ref{section Asai RS epsilon factors}, and for useful conversations.
{We also thank Jiandi Zou for~noti\-cing a small mistake 
in a previous version of the proof of Lemma~\ref{marechal},
and the anonymous referees for valuable comments.}

\section{Notation}
\label{Notation}
Let $\F/\F_\so$ be a quadratic extension of locally compact non-archimedean 
fields of residual characteristic $p\neq2$. Write $\s$ for the 
non-trivial $\F_\so$-automorphism of $\F$. 

For any finite extension $\E$ of $\F_\so$, we denote by~$\o_\E$ its ring of 
integers, by~$\p_\E$ the unique maximal ideal of~$\o_\E$
and by~$\kk_\E$ its residue field. 
We abbreviate~$\o_{\F}$ to $\o$ and $\o_{\F_\so}$ to $\o_{\so}$,
and define similarly $\p$, $\p_\so$, $\kk$, $\kk_\so$. 
The involution $\s$ induces a 
$\kk_\so$-auto\-mor\-phism of $\kk$, still denoted $\s$. It is a generator~of 
the Galois group $\Gal(\kk/\kk_\so)$.
We write $q_\so$ for the cardinality of $\kk_\so$
and $|\cdot|_{\so}$ for the normalized~abso\-lute value on $\F_\so$.

Let $\FC$ be an algebraically  closed field of characteristic~$\ell$ different
from $p$; note that $\ell$ can be $0$. 
We will say we are in the ``modular case'' when we consider the case where 
$\ell>0$. 

We also denote by~$\ef$ 
the character of $\F_\so^\times$ whose kernel contains the subgroup of 
$\F/\F_\so$-norms and is non-trivial if~$\ell\ne 2$. 

Let $\G$ denote the locally profinite group $\GL_n(\F)$, with $n\>1$, equipped 
with the involution~$\s$~acting componentwise. Its $\s$-fixed points is the 
closed subgroup $\G^\s=\GL_n(\F_\so)$. We will identify the centre of $\G$ 
with $\mult\F$, and that of $\G^\s$ with $\mult\F_\so$. 

By a  \emph{representation} of  a locally  profinite group,  we mean  a smooth
representation on  a~$\FC$-module. Given  a representation  $\pi$ of  a closed
subgroup $\H$  of $\G$, we write  $\pi^\vee$ for the smooth  contragredient of
$\pi$ and $\pi^\s$ for the representation $\pi\circ\s$ of $\s(\H)$. 
We also write $\Ind_\H^\G(\pi)$ for the smooth induction of $\pi$ to 
$\G$, and $\cind_\H^\G(\pi)$ for the compact induction of $\pi$ to $\G$. 
If $\chi$ 
is a character of $\H$, we write $\pi\chi$ for the representation 
$g\mapsto\chi(g)\pi(g)$. 

A pair~$(\K,\pi)$, consisting of an open subgroup~$\K$ of~$\G$ and a
smooth
irreducible representation~$\pi$~of $\K$,~is called 
\emph{$\s$-self-dual}
if~$\K$ is $\s$-stable and $\pi^\s$ is isomorphic to $\pi^\vee$.
When~$\K=\G$, we will just talk
about~$\pi$ being $\s$-self-dual. 

Let $\chi$ be a character of $\mult\F_\so$.
A pair~$(\K,\pi)$, consisting of a {$\s$-stable}
open subgroup~$\K$ of~$\G$ and an irreducible representation~$\pi$
of~$\K$, is call\-ed~\emph{$\chi$-distinguished} if 
\[
\Hom_{\K^\s}(\pi,\chi\circ\det)\neq \{0\}
\] 
where $\det$ denotes the determinant on $\G$ and $\K^\s=\K\cap\G^\s$.
We say that~$(\K,\pi)$ is \emph{distinguished}~if~it is 
$\one$-distinguished, 
that is, 
distinguished by the trivial character of
$\mult\F_\so$. When~$\K=\G$, we will just talk about~$\pi$ 
being~$\chi$-distinguished. 

Given  $g\in\G$ and  a  subset $\X\subseteq\G$,  we  set $\X^g=\{g^{-1}xg\  |\
x\in\X\}$. If $f$ is a function on $\X$, 
we write $f^g$ for the function $x\mapsto f(gxg^{-1})$ on $\X^g$. 

For any finite extension $\E$ of $\F_\so$ and any integer $n\>1$, 
we write 
$\N_n(\E)$ for the subgroup of $\GL_n(\E)$ made of all upper triangular 
unipotent matrices
and $\P_n(\E)$ for the standard 
mirabolic subgroup~of all matrices in $\GL_n(\E)$ with final row 
$\begin{pmatrix}0&\cdots&0&1\end{pmatrix}$.

Throughout the paper, by a \textit{cuspidal} representation of $\G$, 
we mean a cuspidal irreducible (smooth) representation of $\G$.

\section{Preliminaries on simple types}

We recall the main results on simple strata, 
characters and types~\cite{BK,BHLTL1,BHEffectiveLC,MSt} that we will need.

\subsection{Simple strata}
\label{genitrix41}

Let $[\aa,\b]$ be a simple stratum in the $\F$-algebra $\Mat_{n}(\F)$ of 
$n\times n$ matrices with entries in $\F$ for some $n\>1$. Recall that $\aa$ 
is a hereditary $\o$-order in $\Mat_n(\F)$ and 
$\b$ is a matrix in $\Mat_{n}(\F)$ such that: 
\begin{enumerate}
\item
the $\F$-algebra $\E=\F[\b]$ is a field, whose degree over $\F$ is denoted $d$;
\item
the multiplicative group $\mult\E$ normalizes $\aa$;
\end{enumerate}
plus an additional technical condition (see \cite[(1.5.5)]{BK}).
The cen\-trali\-zer of $\E$ in $\Mat_{n}(\F)$, denoted~$\B$, is an 
$\E$-algebra isomorphic to $\Mat_{m}(\E)$, where $n=md$. 
The intersection $\aa\cap\B$, 
denoted $\bb$,
is a~here\-di\-ta\-ry $\o_\E$-order in $\B$. 
We write $\p_\aa$ for the Jacobson radical of~$\aa$ and $\U^1(\aa)$ for the 
compact open pro-$p$-subgroup $1+\p_\aa$ of $\G=\GL_n(\F)$,
{and define $\U^1(\bb)$ similarly.
  Note that $\U^1(\bb)=\U^1(\aa)\cap\B^\times$.} 

Note that we use the simplified notation of \cite{BHEffectiveLC} for
  simple strata: what we denote~by $[\aa,\b]$ would be denoted
  $[\aa,v,0,\b]$ in \cite{BK,BHLTL1}, where $v$ is the non-negative integer 
  defined by $\b\aa=\p_{\aa}^{-v}$.

Associated with $[\aa,\b]$, there are compact open subgroups:
\[
\H^1(\aa,\b)\subseteq\J^1(\aa,\b)\subseteq\J(\aa,\b)
\]
of $\mult\aa$
and a finite set $\Cc(\aa,\b)$ of characters of $\H^1(\aa,\b)$ called 
\emph{simple characters}. 
This set depends~on the choice of a character of $\F$, trivial on $\p$ but not on 
$\o$, which we assume to be $\s$-stable and~is~fixed from now on.
Such a choice is possible since $p\neq2$.
Write $\bJ(\aa,\b)$ for the compact mod
centre sub\-group of $\G$ 
generated by $\J(\aa,\b)$ and~the nor\-malizer of $\bb$ in 
$\mult\B$. 

\begin{proposition}[{\cite[2.1]{BHEffectiveLC}}]
We have the following properties:
\begin{enumerate}
\item
The group $\J(\aa,\b)$ is the unique maximal compact subgroup of $\bJ(\aa,\b)$.
\item
The group $\J^1(\aa,\b)$ is the unique maximal normal pro-$p$-subgroup of $\J(\aa,\b)$. 
\item
The group $\J(\aa,\b)$ is generated by $\J^1(\aa,\b)$ and $\mult\bb$, and we have:
\begin{equation*}
\J(\aa,\b)\cap\B^\times=\bb^\times,
\quad
\J^1(\aa,\b)\cap\B^\times=\U^1(\bb).
\end{equation*}
\item
The normalizer of any simple character $\t\in\Cc(\aa,\b)$ in $\G$ is equal to $\bJ(\aa,\b)$. 
\end{enumerate}
\end{proposition}

\subsection{Simple characters and endo-classes}\label{simplechar}

Simple characters have remarkable intertwining and transfer properties. 
Let $[\aa',\b']$ be another~sim\-ple~stra\-tum in $\Mat_{n'}(\F)$ for some 
integer
$n'\>1$, and suppose that we have an isomorphism of~$\F$-al\-ge\-bras 
$\h:\F[\b]\to\F[\b']$ such that $\h(\b)=\b'$. Then there is a canonical 
bijective map: 
\begin{equation*}
\Cc(\aa,\b) \to \Cc(\aa',\b')
\end{equation*}
called the \emph{transfer map}~\cite[Theorem~3.6.14]{BK}.

Now let $[\aa,\b_1]$ and $[\aa,\b_2]$ be simple strata in $\Mat_n(\F)$, 
and assume that we have two simple 
characters $\t_1\in\Cc(\aa,\b_1)$ and $\t_2\in\Cc(\aa,\b_2)$ that intertwine; 
that is, there is a $g\in\GL_n(\F)$ such that 
\begin{equation*}
\t_2(x)=\t_1(gxg^{-1}),
\quad
\text{for all } x\in\H^1(\aa,\b_2)\cap g^{-1}\H^1(\aa,\b_1)g.
\end{equation*}
For $i=1,2$, let $[\aa',\b'_i]$ be a simple stratum in $\Mat_{n'}(\F)$ for
some $n'\>1$, such that $\t_i$ transfers~to a simple character
$\t'_i\in\Cc(\aa',\b'_i)$. Then the simple characters $\t'_1$, $\t'_2$ are 
conjugate under $\GL_{n'}(\F)$ (see \cite[Theorem~8.7]{BHLTL1} and 
\cite[Theorem~3.5.11]{BK}). 

Now let $[\aa_1,\b_1]$, $[\aa_2,\b_2]$ be simple strata in $\Mat_{n_1}(\F)$
and $\Mat_{n_2}(\F)$, respectively, for $n_1,n_2\>1$. We say that two simple
characters $\t_1\in\Cc(\aa_1,\b_1)$ and $\t_2\in\Cc(\aa_2,\b_2)$ are 
\emph{endo-equivalent} if there are simple strata $[\aa',\b'_1]$, 
$[\aa',\b'_2]$ in $\Mat_{n'}(\F)$, for some $n'\>1$, such that $\t_1$ and $\t_2$ 
transfer to simple characters $\t'_1\in\Cc(\aa',\b'_1)$ and 
$\t'_2\in\Cc(\aa',\b'_2)$ which intertwi\-ne (or, equivalently, which are 
$\GL_{n'}(\F)$-conjugate). This defines an equivalence relation on the set 
\begin{equation*}
\bigcup\limits_{[\aa,\b]} \Cc(\aa,\b)
\end{equation*}
where  the union  is taken  over  all simple  strata of  $\Mat_n(\F)$ for  all
$n\>1$~\cite[Section~8]{BHLTL1}. An~equi\-va\-lence class for this relation is
called an \emph{endo-class}. 

Given a simple character $\t\in\Cc(\aa,\b)$, the degree of $\E/\F$, its
ramification order and its residue class  degree only depend on the endo-class
of $\t$. These integers are called  the degree, ramification order and residue
class degree of  this endo-class. The field extension $\E/\F$  is not uniquely
determi\-ned, but its maximal tamely ramified sub-extension is uniquely
determi\-ned, up  to an  $\F$-isomor\-phism, by the  endo-class of  $\t$. This
tamely ramified sub-extension is called  the \emph{tame parameter field} of the
endo-class~\cite[2.2,~2.4]{BHEffectiveLC}. 

Let $\Ee(\F)$ denote the set of all endo-classes of simple characters in all 
general linear groups over~$\F$. Given a finite tamely ramified extension $\T$ 
of $\F$, there is a surjective map: 
\begin{equation*}
\Ee(\T) \to \Ee(\F)
\end{equation*}
with finite fibers, called the \emph{restriction map}~\cite[2.3]{BHEffectiveLC}. Given $\TT\in\Ee(\F)$, the endo-classes $\boldsymbol{\Psi}\in\Ee(\T)$ which restrict to $\TT$ are called the $\T/\F$-{lifts} of $\TT$. If $\TT$ has tame parameter field $\T$, then $\Aut_\F(\T)$ acts transitively and faithfully on the set of $\T/\F$-lifts of $\TT$~\cite[2.3,~2.4]{BHEffectiveLC}.

\subsection{Simple types and cuspidal representations}

Let us write $\G=\GL_n(\F)$ for some $n\>1$. 
A family of pairs $(\J,\l)$ called \emph{simple types}, made of a com\-pact open
subgroup $\J$ of $\G$ and an ir\-reducible representation $\l$ of $\J$, has
been constructed in~\cite{BK} (see also \cite{MSt} for the modular case). 

By construction, given a simple type $(\J,\l)$ in $\G$, there are a simple
stratum $[\aa,\b]$ and a simple~cha\-rac\-ter $\t\in\Cc(\aa,\b)$ such that
$\J(\aa,\b)=\J$ and $\t$ is contained in the restriction of $\l$ to
$\H^1(\aa,\b)$. Such a simple character is said to be \emph{attached to} $\l$.

\begin{definition}
When the hereditary  order $\bb=\aa\cap\B$ is a maximal order  in $\B$, we say
that the simple stratum $[\aa,\b]$ and the simple characters in $\Cc(\aa,\b)$ 
are \emph{maximal}. A simple type with~a~maxi\-mal attached simple character is 
called a \emph{maximal simple type}. 
\end{definition}

When the simple stratum $[\aa,\b]$ is maximal, 
and given a homomorphism $\B\simeq\Mat_m(\E)$ 
of~$\E$-alge\-bras identifying $\bb$ 
with the standard maximal order,
one has group isomorphisms: 
\begin{equation}
\label{nonoc}
\J(\aa,\b)/\J^1(\aa,\b) \simeq \bb^\times/\U^1(\bb) \simeq \GL_{m}(\ee).
\end{equation}

The following proposition gives a description of cuspidal (irreducible)
representations of $\G$ in terms of maximal simple types. 

\begin{proposition}
\label{GeorgesSaval}
Let $\pi$ be a cuspidal representation of $\G$. 
\begin{enumerate}
\item
There is a maximal simple type $(\J,\l)$ such that $\l$ occurs as a subrepresentation of the~res\-triction of $\pi$ to $\J$. This sim\-ple type is uniquely deter\-mined up to $\G$-conjugacy.
\item
The simple character $\t$ attached to $\l$ is uniquely deter\-mined up to $\G$-conjugacy. Its endo-class $\TT$ is called the \emph{endo-class} of $\pi$.
\item
\label{sibelius}
If $\t'\in\Cc(\aa',\b')$ is a simple character in $\G$, then the restriction of $\pi$ to $\H^1(\aa',\b')$ contains $\t'$ if and only if $\t'$ is maximal and has endo-class $\TT$, that is, if and only if $\t$, $\t'$ are $\G$-conjugate.
\item
Let $[\aa,\b]$ be a maximal simple stratum such that $\J=\J(\aa,\b)$ and $\t\in\Cc(\aa,\b)$. The simple type $\l$ extends uniquely to a representation $\bl$ of the norma\-lizer $\bJ=\bJ(\aa,\b)$ of $\t$ in $\G$ such that the compact induction of $\bl$ to $\G$ is isomorphic to $\pi$.
\end{enumerate}
\end{proposition}

\begin{proof}
This follows from~\cite[6.2,~8.4]{BK}. 
See~\cite[Section~3]{MSt} in the case 
that~$\FC$ has positive characte\-ristic.
\end{proof}

A pair $(\bJ,\bl)$ constructed in this way is called an 
\emph{extended maximal simple type} in $\G$. 
Compact~in\-duction induces a bijection between
$\G$-conjugacy classes of extended maximal~sim\-ple types and~iso\-mor\-phism
classes of cuspidal representations of $\G$ (\cite[6.2]{BK} and 
\cite[Theorem 3.11]{MSt}). 

\section{The $\s$-self-dual type theorem}\label{S2}

We state our first main theorem. We fix an integer $n\>1$ and write $\G=\GL_n(\F)$.

\begin{theorem}
\label{PIMAIN}
Let $\pi$ be a cuspidal representation of $\G$. Then $\pi^\s\simeq\pi^\vee$ if and~only if 
$\pi$ contains an extended maximal simple type $(\bJ,\bl)$ such that $\bJ$ is $\s$-stable and $\bl^\s\simeq\bl^\vee$.
\end{theorem}

In other words, a cuspidal representation of $\G$ is 
$\s$-self-dual if  and only if  it contains a $\s$-self-dual  extended maximal
simple type. 

If $(\bJ,\bl)$ is an extended maximal simple type for the cuspidal
representation $\pi$, then $(\s(\bJ),\bl^\s)$~is~an extended maximal simple
type for $\pi^\s$ and $(\bJ,\bl^\vee)$ is  an extended maximal simple type for
$\pi^\vee$.~Thus, if $\pi$ contains an extended maximal simple type
$(\bJ,\bl)$ such that $\bJ$ is $\s$-stable and $\bl^\s$, $\bl^\vee$ 
are~iso\-mor\-phic, then $\pi^\s$, $\pi^\vee$ are isomorphic. 
The rest of Section~\ref{S2}
is devoted to the proof of the converse statement. 

\subsection{The endo-class}
\label{tec}

Start with a cuspidal representation $\pi$ of $\G$, and suppose 
that $\pi^\s\simeq \pi^\vee$. Let $\TT$ be its endo-class over $\F$. 
Associated with it, there are its degree $d=\deg(\TT)$ and its tame parameter 
field $\T$: this is a tamely ramified finite extension of $\F$, unique up to 
$\F$-isomorphism (see~\S\ref{simplechar}). 

If $\t\in\Cc(\aa,\b)$ is a maximal simple character contained in $\pi$, then 
$\t^{-1}\in\Cc(\aa,-\b)$ is con\-tained in~$\pi^\vee$ and 
$\t\circ\s\in\Cc(\s(\aa),\s(\b))$ is contained in $\pi^\s$.
{Note that we use the fact that the character of $\F$ fixed in
  Paragraph~\ref{genitrix41} is $\s$-stable in order to have
  $\t\circ\s\in\Cc(\s(\aa),\s(\b))$.}
We write 
$\TT^\vee$ for the endo-class of $\t^{-1}$, and $\TT^\s$ for that of 
$\t\circ\s$. The assumption on $\pi$ implies that $\TT^\s=\TT^\vee$. We will 
prove the following theorem. 

\begin{theorem}
\label{THETAMAINNONMAX}
Let $\TT\in\Ee(\F)$ be an endo-class of degree dividing $n$ such that 
$\TT^\s$ is equal to $\TT^\vee$, 
and let $\t\in\Cc(\aa,\b)$ be a simple character in $\G$ of
endo-class $\TT$. There are a simple stratum $[\aa',\b']$ and a simple 
character $\t'\in\Cc(\aa',\b')$ such that: 
\begin{enumerate}
\item
the character $\t'$ is $\G$-conjugate to $\t$,
\item
the group $\H^1(\aa',\b')$ is $\s$-stable and $\t'\circ\s=\t'^{-1}$,
\item
the order $\aa'$ is $\s$-stable and $\s(\b')=-\b'$.
\end{enumerate}
\end{theorem}

Before proving Theorem~\ref{THETAMAINNONMAX}, we show how it implies
Theorem~\ref{PIMAIN}. 
By applying Theorem~\ref{THETAMAINNONMAX} to~any simple 
character $\vartheta$ contained in $\pi$, 
{which is maximal by Proposition \ref{GeorgesSaval}\ref{sibelius},
we get a maximal sim\-ple character
$\t\in\Cc(\aa,\b)$, conjugate to~$\vartheta$, such that
$\aa$ is $\s$-stable and $\s(\b)=-\b$ and:
\begin{equation*}
\t\circ\s=\t^{-1}.
\end{equation*}
Thus $\t$ is contained in $\pi$ and its normalizer $\bJ$ in $\G$ is $\s$-stable.
Let $(\bJ,\bl)$ be an extended maxi\-mal sim\-ple type for $\pi$ with attached
simple character $\t$. 
Since $\pi$ is $\s$-self-dual, it contains both $(\bJ,\bl)$ and
$(\bJ,\bl^{\vee\s})$. By~Pro\-po\-sition~\ref{GeorgesSaval}, this implies that they
are conju\-ga\-te  by an element $g\in\G$,  that~is, $g$ normalizes  $\bJ$ and
$\bl^{\vee\s}$ is isomorphic  to $\bl^{g}$. Now consider  the simple characters
$\t^{-1}\circ\s=\t$ and $\t^g$. 
Both of them are contained in $\bl^g$.
Restricting $\bl^g$ to the inter\-section: 
\begin{equation}
\label{IGH}
\H^1(\aa,\b)\cap\H^1(\aa,\b)^g
\end{equation}
we get a direct sum of copies of $\t$ containing the restriction of
$\t^g$ to \eqref{IGH}. It follows~that~$g$ inter\-twines $\t$.}
By \cite[Theorem 3.3.2]{BK}, 
which describes the intertwining set of a simple character, 
we have $g\in\bJ\B^\times\bJ$.
We thus may assume that $g\in\B^\times$. By uniqueness of 
the maximal compact~sub\-group in $\bJ$, the identity $\bJ^g=\bJ$ gives us 
$\J^g=\J$. Inter\-secting with $\B^\times$ gives $\mathfrak{b}^{\times 
  g}=\mathfrak{b}^{\times}$. It follows that $g$ norma\-lizes the 
order $\bb$. We thus have $g\in\bJ$, thus $\bl^\s\simeq\bl^\vee$. 
Theorem~\ref{PIMAIN} is proved. 

\begin{remark}
Assuming that Theorem~\ref{THETAMAINNONMAX} holds, and using Intertwining 
Implies Conjugacy~\cite[Theorem~5.7.1]{BK}, the same argument shows that, if 
$\pi$ is a $\s$-self-dual irreducible representation of $\G$ that contains a 
simple type, then $\pi$ contains a $\s$-self-dual simple type.
{In particular, any $\s$-self-dual discrete series representation of $\G$
  contains a $\s$-self-dual simple type.}
\end{remark}

\begin{remark}
However, an arbitrary $\s$-self-dual irreducible representation of $\G$ may 
not contain a $\s$-self-dual semisimple type. See \cite{BKsemi,MSt} 
for the notion of semisimple type and~Para\-graph \ref{CESST} for a 
counter\-example. 
\end{remark}

It thus remains to prove Theorem~\ref{THETAMAINNONMAX}. For this, one can forget about the representation $\pi$.

\subsection{A prelude}

We first show how to deal with the 
(second part of the)
third condition of Theorem \ref{THETAMAINNONMAX}. 
Recall (see \cite{BK}) that a stratum $[\aa,v,r,\b]$ in
$\Mat_n(\F)$ is \emph{pure} if $\F[\b]$ is a field, $\F[\b]^\times$
normalises $\aa$ and $\b\aa=\p_{\aa}^{-v}$.
{(See Paragraph \ref{genitrix41} for the comment on the notation.)}

{Here again (see Paragraph~\ref{tec}),}
we use the fact that the character of $\F$ fixed in 
Paragraph \ref{genitrix41} is {$\s$-stable}.

\begin{lemma}
\label{poussinjaune}
Let $[\aa,v,r,\b]$ be a pure stratum in $\Mat_n(\F)$ with $\s(\aa)=\aa$ 
and $\s(\b)+\b\in\p_{\aa}^{-r}$. 
There is a simple stratum $[\aa,v,r,\g]$ such that $\b-\g\in\p_{\aa}^{-r}$ 
and $\s(\g)+\g=0$.
\end{lemma}

\begin{proof}
The proof is exactly as in \cite[Proposition~1.10]{StevensIntertwining}, 
using the involution $\s$ instead of the adjoint involution 
$x\mapsto\overline{x}$ used in \cite{StevensIntertwining}.
\end{proof}

\begin{proposition}
\label{Pompidou}
Let $[\aa,\b]$ be a simple stratum in $\Mat_n(\F)$ with $\s(\aa)=\aa$.
Suppose that there is a simple character
$\t\in\Cc(\aa,\b)$ such that $\H^1(\aa,\b)$ is $\s$-stable and 
$\t\circ\s=\t^{-1}$. Then there is a simple stratum $[\aa,\g]$ such that 
$\t\in\Cc(\aa,\g)$ and $\s(\g)+\g=0$. 
\end{proposition}

\begin{proof}
The proof is exactly the same as in \cite[Theorem~6.3]{StevensIntertwining}, 
using the involution $\s$ instead of the adjoint involution used 
in~\cite{StevensIntertwining}, and replacing~\cite[Proposition 
1.10]{StevensIntertwining} by Lemma~\ref{poussinjaune}. 
\end{proof}

\subsection{The tame parameter field}

From  now  on, and  until  the  end of  this  section,  $\TT\in\Ee(\F)$ is  an
endo-class, with degree $d$ dividing  $n$, which is~$\s$-self-dual -- that is,
such that $\TT^\s=\TT^\vee$. In this paragraph, we will see that this symmetry
condition on $\TT$ implies that its tame parameter field $\T/\F$ inherits 
certain properties. 

Note that we \emph{do not} assume that $\TT$ is the endo-class of 
some~$\s$-self-dual cuspidal~repre\-sentation $\pi$ of $\G$. For the 
notion of a $\T/\F$-lift of $\TT$, we refer to \S\ref{simplechar}. 

\begin{lemma}
\label{L1}
Let $\TT$ be a $\s$-self-dual endo-class and $\T/\F$ be its tame parameter 
field. 
\begin{enumerate}
\item\label{L1.i}
Given a $\T/\F$-lift $\boldsymbol{\Psi}$ of $\TT$, there is a unique involutive $\F_\so$-automorphism $\a$ of $\T$ extending $\s$ such that $\boldsymbol{\Psi}^\vee=\boldsymbol{\Psi}^\a$.
\item For any $\F$-automorphism $\g$ of $\T$, the $\F_\so$-involution of $\T$ associated with $\boldsymbol{\Psi}^\g$ is $\g^{-1}\a\g$.
\end{enumerate}
\end{lemma}

\begin{proof}
The tame parameter  field of $\TT^\vee$ is  $\T$, and that of  $\TT^\s$ is the
field $\T$ endowed with the map $x\mapsto\s(x)$ from $\F$ to $\T$. The 
assumption on $\TT$ implies that these tame parameter fields are 
$\F$-isomorphic. Thus there exists an $\F_\so$-automorphism of $\T$ whose 
restriction to $\F$ is $\s$. 

Let $\boldsymbol{\Psi}$  be a $\T/\F$-lift of  $\TT$ (see \S\ref{simplechar}).
Then  $\boldsymbol{\Psi}^\vee$  is  a  $\T/\F$-lift  of  $\TT^\vee$,  and  the
bijection $\a\mapsto\boldsymbol{\Psi}^\a$ between automorphisms of $\T/\F_\so$
and    $\T/\F_\so$-lifts    of    $\TT$   induces    a    bijection    between
$\F_\so$-automorphisms of  $\T$ extending $\s$ and  $\T/\F$-lifts of $\TT^\s$.
Thus there is a unique $\F_\so$-automorphism  $\a$ of $\T$ extending $\s$ such
that            $\boldsymbol{\Psi}^\vee=\boldsymbol{\Psi}^\a$.           Since
$\boldsymbol{\Psi}^{\vee\vee}=\boldsymbol{\Psi}$,      we     deduce      that
$\boldsymbol{\Psi}^{\a^2}=\boldsymbol{\Psi}$. That $\a^2$ is trivial follows 
from the fact that $\a^2$ is in $\Aut_\F(\T)$, which acts faithfully on the 
set of $\T/\F$-lifts of $\TT$. 
\end{proof}

\begin{remark}
{It is \emph{not} in general true that every involutive 
  $\F_\so$-automorphism $\a$ of $\T$  extending the $\F_\so$-automorphism $\s$
  of $\F$  has the  additional property required  by Lemma~\ref{L1}\ref{L1.i}.
  For example, if~$\F/\F_\so$ is unramified and~$\T/\F$ is ramified quadratic,
  then~$\T/\F_\so$ is a biquadratic extension and~the two automorphisms fixing 
  the ramified quadratic sub-extensions of~$\F_\so$ in~$\T$ are both 
  involutions~ex\-tending~$\s$; however, they are not conjugate so, by the
  uniqueness statement in Lemma~\ref{L1}, cannot both have the additional 
  property.} 
\end{remark}

Let $\a$ be an $\F_\so$-involution of $\T$ given by Lemma~\ref{L1}, and let 
$\T_\so$ be the fixed points of $\a$ in $\T$. Thus $\T_\so\cap\F=\F_\so$. 

\begin{lemma}
\label{L2}
The canonical homomorphism $\T_\so\otimes_{\F_\so}\F\to\T$ of $\T_\so\otimes_{\F_\so}\F$-modules is an isomorphism.
\end{lemma}

\begin{proof}
The canonical homomorphism is an isomorphism if and only if $\F$ does not embed in $\T_\so$ as an $\F_\so$-algebra. Assume that there is such an embedding. Since $\F$ is Galois over $\F_\so$, its image is $\F$. Thus $\F$ is contained in $\T_\so$, which contradicts $\T_\so\cap\F=\F_\so$.
\end{proof}

Write $t$ for the degree of $\T$ over $\F$.

\begin{corollary}
\label{L3}
There is an embedding of $\F$-algebras $\iota:\T\hookrightarrow\Mat_{t}(\F)$ such that:
\begin{equation*}
\iota(\a(x))=\s(\iota(x))
\end{equation*}
for all $x\in\T$. In particular, the image of $\iota$ in $\Mat_t(\F)$ is $\s$-stable.
\end{corollary}

\begin{proof}
Fix an $\F_\so$-embedding $\iota_\so$ of $\T_\so$ in $\Mat_{t}(\F_\so)$. Then $\iota=\iota_\so\otimes\F$ has the required property, thanks to Lemma~\ref{L2}. 
\end{proof}

\begin{remark}
The natural group homomorphism:
\begin{equation*}
\Aut_{\F_\so}(\T) \to \Aut_{\F_\so}(\T_\so) \rtimes \Gal(\F/\F_\so)
\end{equation*}
(where the semi-direct product is defined with respect to $\a$) is an isomorphism.
\end{remark}

\subsection{The maximal and totally wild case}

In this paragraph, we will assume that $d=n$ and $\T=\F$. 

\begin{proposition}
\label{SensenbrinkMTW}
Let $\t$ be a simple character in $\G$ with endo-class $\TT$. There is a simple character $\t'\in\Cc(\aa',\b')$ which is $\G$-conjugate to $\t$, such that $\aa'$ is $\s$-stable and $\t'\circ\s=\t'^{-1}$.
\end{proposition}

Let $[\aa,\b]$  be a simple stratum  such that $\t\in\Cc(\aa,\b)$. We  may and
will assume that the principal order $\aa$ is standard (that is,~$\aa$ is made
of matrices with coefficients in $\o$ and its reduction mod $\p$
is made of up\-per block triangular matrices), thus $\s$-stable. The
extension $\F[\b]$ is totally wildly~ra\-mi\-fied over $\F$. In particular, $\aa$
is a minimal order in $\Mat_n(\F)$. 

Write $\U=\aa^\times$, which is the standard Iwahori subgroup of $\G$. For all
$i\>1$, write $\U^i=1+\p_{\aa}^i$, which is a normal subgroup of $\U$. 
Then $\U/\U^1\simeq\kk^{\times n}$ is abelian, of order prime to $p$, and
$\U^i/\U^{i+1}$ is an abelian $p$-group for all $i\>1$. 

Since $\TT^\s=\TT^\vee$ and $\aa$ is $\s$-stable, the characters
$\t\circ\s\in\Cc(\aa,\s(\b))$ and $\t^{-1}\in\Cc(\aa,-\b)$~inter\-twine. 
By Intertwining Implies Conjugacy for simple 
characters~\cite[Theorem~3.5.11]{BK}, there is a $u\in\U$ such that
$\H^1(\aa,\s(\b))=u^{-1}\H^1(\aa,-\b)u$  and   $\t\circ\s=(\t^{-1})^u$.  Since
$\s$ is involutive and the $\G$-normalizer of $\t$ is $\bJ$, this gives us: 
\begin{equation}
\label{condu}
u\s(u)\in\bJ\cap\U=\J.
\end{equation}
We search for an $x\in\G$ such that the character $\t'=\t^x\in\Cc(\aa^x,\b^x)$ has the desired proper\-ty. This amounts to the condi\-tion $u\s(x)x^{-1}\in\bJ$.

Note that $\J=\o^\times\J^1$ since $\F[\b]$ is totally ramified over $\F$. Thus the image of $\J$ in $\U/\U^1\simeq\kk^{\times n}$ is the image of the diagonal embedding of $\kk^\times$ in $\kk^{\times n}$. Let $\M$ be the torus made of all diagonal matrices of $\G$.

\begin{lemma}
\label{pikotty}
There is a $y\in\M$ such that $u\s(y)y^{-1}\in\J\U^1=\EuScript{O}^\times\U^1$.
\end{lemma}

\begin{proof}
There are $u_1,\dots,u_n\in\mult\kk$ such that $u$ mod $\U^1$ is equal to $(u_1,\dots,u_n)$ in $\U/\U^1\simeq\kk^{\times n}$. Changing $u$ in the equivalence class $\EuScript{O}^\times u$, we may assume that $u_1=1$.

The condition~\eqref{condu} says  that $u\s(u)$ mod $\U^1$ is in  the image of
the diagonal embedding of $\kk^\times$ in $\kk^{\times n}$. Since $u_1=1$,
this gives us $u_i\s(u_i)=1$ for all $i\in\{1,\dots,n\}$.

Assume first that $\F$ is unramified over $\F_\so$. Then $\kk$ is quadratic over $\kk_\so$ and $\s$ induces the non-trivial $\kk_\so$-automorphism of $\kk$. We search for $y=(y_1,\dots,y_n)\in\kk^{\times n}$ such that $u\s(y)y^{-1}=1$ in $\kk^{\times n}$. This is possible by Hilbert's Theorem~90, since $u_i\s(u_i)=1$ for all $i$.

Assume  now  $\F$  is  ramified  over   $\F_\so$.  Then  $\s$  is  trivial  on
$\kk=\kk_\so$. We thus have $u_i^2=1$ which implies $u_i\in\{-1,1\}$. 
Let $\w$ be a uniformizer of $\F$ such that $\s(\w)=-\w$. 
Such a choice is possible since $p\neq2$.
We are searching
for a $y=(y_1,\dots,y_n)\in\F^{\times n}$ such that $\s(y)y^{-1}\in\U$ and
$u\s(y)y^{-1}=1$ in $\kk^{\times n}$. Let $y_i=1$ if $u_i=1$, and let $y_i=\w$
otherwise. This gives us a $y\in\M$ satisfying the required condition. 
\end{proof}

Let us write $zu\s(y)y^{-1}\in\U^1$ for some $y\in\M$ and $z\in\o^\times$
given by Lemma \ref{pikotty}. By replacing~the stratum $[\aa,\b]$ by
$[\aa^y,\b^y]$, the simple character  $\t$ by $\t^y\in\Cc(\aa^y,\b^y)$ and $u$
by  $y^{-1}zu\s(y)$,  which  does  not  affect the  fact  that  the  order  is
$\s$-stable, we  may and will assume  that $u\in\U^1$. We write  $\J^0=\J$ and
$\J^i=\J\cap\U^i$ for $i\>1$. 

\begin{lemma}
\label{FrancoisSturel}
Let $v\in\U^{i}$ for some $i\>1$, and assume that $v\s(v)\in\J^i$.
Then there are $j\in\J^i$ and $x\in\U^i$ such that 
$jv\s(x)x^{-1}\in\U^{i+1}$.
\end{lemma}

\begin{proof}
Recall that $\U^i/\U^{i+1}$ is abelian, and write $h=v\s(v)$. 
We have $\s(h)\equiv h$ mod $\U^{i+1}$. This implies that
$h\in\V=\J^i\U^{i+1}\cap\s(\J^i\U^{i+1})\supseteq\U^{i+1}$.
We thus have
$v\s(v)\equiv1$ mod $\V$. The quotient  $\W=\U^i/\V$ is an abelian, finite and
$\s$-stable $p$-group, and the first cohomology group of
{$\Gal(\F/\F_{\so})$}
in $\W$ is
trivial since $p\neq2$.  We thus have $v\equiv x\s(x)^{-1}$ mod  $\V$ for some
element $x\in\U^i$. This  gives us $v\s(x)x^{-1}\in\V\subseteq\J^i\U^{i+1}$ as
required. 
\end{proof}

\begin{lemma}
\label{marechal}
There is a sequence of triples~$(x_i,j_i,v_i)\in\U^i\times\J^i\times\U^{i+1}$, for $i\>0$, 
satisfying the following conditions: 
\begin{enumerate}
\item
$(x_0,j_0,v_0)=(1,1,u)$;
\item\label{Condition2}
for all $i\>0$, if we set $y_i=x_0x_1\dots x_{i}\in\U^1$, 
then the simple character $\t_i=\t^{y_i}\in\Cc(\aa,\b^{y_i})$ satisfies 
$\t_i\circ\s=(\mbox{$\t_i$}^{-1})^{v_{i}}$; 
\item\label{Condition3}
for all $i\>1$, we have 
{$y_iv_i=j_iy_{i-1}v_{i-1}\s(x_i)$.}
\end{enumerate}
\end{lemma}

\begin{proof}
Assume the triples  $(x_k,j_k,v_k)$ have been defined for all  $k<i$, for some
$i\>1$.  Applying   Lemma  \ref{FrancoisSturel}  to   $v_{i-1}\in\U^i$,  which
satisfies 
\begin{equation*}
v_{i-1}\s(v_{i-1})\in\bJ^{y_{i-1}}\cap\U^{i}=\J^i(\aa,\b^{y_{i-1}})
\end{equation*}
thanks to Condition~\ref{Condition2}, we obtain 
$h_{i}\in\J^{i}(\aa,\b^{y_{i-1}})$ and $x_{i}\in\U^{i}$ such that 
$h_{i}v_{i-1}\s(x_{i})\mbox{$x_{i}$}^{-1}\in\U^{i+1}$. Now define $j_i\in\J^i$ 
and $v_i\in\U^{i+1}$ by 
{$j_iy_{i-1}=y_{i-1}h_i$
and
$x_iv_i=h_iv_{i-1}\s(x_i)$.}
Setting $y_{i}=y_{i-1}x_i$ and $\t_{i}=\t^{y_{i}}$, we get: 
\begin{align*}
\t_{i}\circ\s 
&= (\t_{i-1}\circ\s)^{\s(x_i)} \\
&= (\t_{i-1}^{-1})^{v_{i-1}\s(x_i)} \\
&= (\t_{i-1}^{-1})^{x_iv_i} 
\end{align*}
since $h_i\in\J^i(\aa,\b^{y_{i-1}})$ normalizes $\t_{i-1}$. 
Since $\t_{i-1}^{x_i}$ is equal to $\t_{i}^{}$,
we get the expected result.
\end{proof}

Let $x\in\U^1$ be the limit of $y_i=x_0x_1\dots x_{i}$ and $h\in\J^1$ that of 
$j_i\dots j_1j_0$
when $i$ tends to infinity. We have:
\begin{equation*}
y_{i}v_{i}\mbox{$y_i$}^{-1} = (j_i\dots j_1j_0)u\s(y_i) \mbox{$y_i$}^{-1} \in \U^i.
\end{equation*}
Passing to the limit, we get $u\s(x)x^{-1}= h^{-1}\in\J$, as expected. 

\subsection{The maximal case}

In this paragraph, we assume that $d=n$ only. 
We generalize Proposition~\ref{SensenbrinkMTW} to this situation. 

\begin{proposition}
\label{SensenbrinkM}
Let $\t\in\Cc(\aa,\b)$ be a simple character in $\G$ of endo-class $\TT$. 
There is a simple character $\t'\in\Cc(\aa',\b')$ which is $\G$-conjugate to 
$\t$, such that $\aa'$ is $\s$-stable and $\t'\circ\s=\t'^{-1}$. 
\end{proposition}

\begin{proof}
Let $\E$ be the  field extension $\F[\b]$, and let $\T$  be the maximal tamely
ramified extension of $\F$ in $\E$. It is the tame parameter field for the
endo-class $\TT$. The simple character $\t$ determines~a $\T/\F$-lift
$\boldsymbol{\Psi}$ of $\TT$ as  in \cite[Section~9]{BHLTL1}. Namely, let $\C$
denote the centralizer of $\T$ in $\Mat_n(\F)$. The intersection
$\mathfrak{c}=\aa\cap\C$ is a  minimal order in $\C$, giving rise  to a simple
stratum $[\mathfrak{c},\b]$ in $\C$.~The~res\-triction of $\t$ to
$\H^1(\mathfrak{c},\b)$, denoted $\t_{\T}$, is a simple character associated
to this simple stratum, called the interior $\T/\F$-lift of $\t$ in
\cite{BHLTL1}. Its endo-class, denoted  $\boldsymbol{\Psi}$, is a $\T/\F$-lift
of $\TT$. 

Lemma~\ref{L1} gives  us a unique  $\F_\so$-involution $\a$ of $\T$  such that
$\a|_{\F}=\s$ and $\boldsymbol{\Psi}^\vee=\boldsymbol{\Psi}^{\a}$.  Let us fix
an $\F$-embedding $\iota$ of $\T$  in $\Mat_{t}(\F)$ as in Corollary~\ref{L3}.
Composing with the diagonal  em\-bed\-ding~of $\Mat_{t}(\F)$ in $\Mat_{n}(\F)$
gives us an $\F$-embedding of $\T$ in $\Mat_{n}(\F)$ such that: 
\[
\iota(\a(x))=\s(\iota(x)), \quad x\in\T.
\]
By the Skolem--Noether  theorem, this embedding is  implemented by conjugating
by some~$g\in\G$. Thus, conjugating $[\aa,\b]$ and  $\t$ by $g$, we may assume
that $\T$ is $\s$-stable and that the $\F_\so$-involution $\s$ of $\Mat_n(\F)$
induces $\a$ on $\T$. Note that $\C$ is $\s$-stable and~is cano\-ni\-cally
isomorphic to the $\T$-alge\-bra $\Mat_{n/t}(\T)$.  The restriction of $\s$ to
$\C$ identifies with the involution $\a$ acting componentwise. From now on, we
will abuse the notation and write $\s$ instead of $\a$. 

We now apply Proposition~\ref{SensenbrinkMTW} to the simple character
$\t_{\T}$ whose endo-class $\boldsymbol{\Psi}$ satisfies
$\boldsymbol{\Psi}^\vee=\boldsymbol{\Psi}^{\s}$. We thus get a $y\in\C^\times$
such   that  $\mathfrak{c}^y$   is  $\s$-stable   and  the   simple  character
$\vartheta=\t_\T^y$ satisfies $\vartheta\circ\s=\vartheta^{-1}$.
Since the map 
$\aa\mapsto\aa^\times\cap\C^\times$ is injective on hereditary orders of 
$\Mat_n(\F)$ normalized by $\mult\T$
{(see for instance \cite[Section 2]{BHLTL1}),}
we deduce that the order $\aa'=\aa^y$ is $\s$-stable.
Since interior $\T/\F$-lifting~is injective 
{from $\Cc(\aa^y,\b^y)$ to $\Cc(\mathfrak{c}^y,\b^y)$
  by \cite[Theorem 7.10]{BHLTL1}}, 
the simple character $\t'=\t^y$ satisfies the expected property $\t'\circ\s=\t'^{-1}$.
\end{proof}

\subsection{The general case}

In this paragraph, we prove Theorem~\ref{THETAMAINNONMAX} in the general case. Write $n=md$, with $m\>1$. 

Let $\t\in\Cc(\aa,\b)$ be a simple character of endo-class $\TT$. By conjugating in $\G$, we may assume that $\aa$ is $\s$-stable. 

Fix an $\F$-algebra homomorphism $\iota:\F[\b]\to\Mat_d(\F)$. Let $\aa_0$ denote the unique hereditary order in $\Mat_d(\F)$ normalized by $\F[\iota\b]^\times$ and $\t_0\in\Cc(\aa_0,\iota\b)$ denote the trans\-fer of $\t$. By Proposition~\ref{SensenbrinkM}, there are a maximal simple stratum $[\aa'_0,\b'_0]$ and a simple character $\t'_0\in\Cc(\aa'_0,\b'_0)$ such that: 
\begin{enumerate}
\item
the character $\t'_0$ is conjugate to $\t_0^{}$ under $\GL_d(\F)$,
\item
the group $\H^1(\aa'_0,\b'_0)$ is $\s$-stable and $\t'_0\circ\s=\t_0'^{-1}$, 
\item
the order $\aa_0$ is $\s$-stable.
\end{enumerate}
Proposition~\ref{Pompidou} implies that, without changing $\aa_0$, we may 
assume that $\s(\b'_0)=-\b'_0$. 

Let us now embed $\Mat_d(\F)$ diagonally in the $\F$-algebra $\Mat_n(\F)$. 
This gives us an $\F$-algebra~homomor\-phism 
$\iota':\F[\b'_0]\to\Mat_n(\F)$. Write $\b'=\iota'\b'_0$ and $\E'=\F[\b']$. 
Since $\s(\b')=-\b'$, the field $\E'$ is stable by $\s$. The centralizer $\B'$ 
of $\E'$ in $\Mat_n(\F)$ naturally identifies with $\Mat_m(\E')$. 

Let $\bb'$ be a standard hereditary order in $\B'$, and let $\aa'$ be the unique hereditary order in $\Mat_n(\F)$ normalized by $\E'^\times$ such that $\aa'\cap\B'=\bb'$. Then we have a simple stratum $[\aa',\b']$ in $\Mat_n(\F)$. Let $\t'\in\Cc(\aa',\b')$ be the transfer of $\t$. Since $\aa'$ is $\s$-stable and $\s(\b')=-\b'$, we have:
\begin{equation*}
\s(\H^1(\aa',\b')) 
= \H^1(\s(\aa'),\s(\b')) 
= \H^1(\aa',-\b')
= \H^1(\aa',\b').
\end{equation*}
Let   $\M$   be  the   standard   Levi   subgroup   of  $\G$   isomorphic   to
$\GL_d(\F)\times\dots\times\GL_d(\F)$. Write $\P$ for the stan\-dard parabolic
subgroup of $\G$ generated by $\M$ and upper triangular matrices, and $\N$ for
its unipotent  radical. Let $\N^-$ be  the unipotent radical of  the parabolic
subgroup opposite to $\P$ with respect to $\M$. By \cite[Paragraph~7.1]{BK}, 
we have: 
\begin{equation}
\label{ganascia}
\left.\begin{array}{rl}
\H^1(\aa',\b') &=\ (\H^1(\aa',\b')\cap\N^-)\cdot(\H^1(\aa',\b')\cap\M)
\cdot(\H^1(\aa',\b')\cap\N), \\[5pt]
\H^1(\aa',\b')\cap\M &=\ \H^1(\aa'_0,\b'_0)\times\dots\times\H^1(\aa'_0,\b'_0). 
\end{array}\right\}
\end{equation}
By  \cite[Proposition   7.1.19]{BK},  the   character  $\t'$  is   trivial  on
$\H^1(\aa',\b')\cap\N$  and $\H^1(\aa',\b')\cap\N^-$,  and the  restriction of
$\t'$ to  $\H^1(\aa',\b')\cap\M$ is equal  to $\t'_0\otimes\dots\otimes\t'_0$.
As $\M$, $\N$, $\N^-$ and $\H^1(\aa',\b')$ are $\s$-stable, 
and by
unique\-ness of the Iwahori decomposition~\eqref{ganascia}, we get
$\t'\circ\s=\t'^{-1}$. Finally, as $\F[\b]$ and $\E'$ have the same
ramification index over $\F$ (see~\S\ref{simplechar})  we may choose the order
$\bb'$  such that  $\aa$  and  $\aa'$ are  conjugate.  The  transfer map  from
$\Cc(\aa,\b)$ to $\Cc(\aa',\b')$ is thus  im\-ple\-mented by con\-jugacy by an
element of $\G$. It follows that $\t$ and $\t'$ are $\G$-conjugate. 

\begin{definition}
\label{sstdmss}
A maximal simple stratum $[\aa,\b]$ in $\Mat_n(\F)$ is said to be \emph{$\s$-standard} if:
\begin{enumerate}
\item
the hereditary order $\aa$ is $\s$-stable and $\s(\b)=-\b$;
\item
the element $\b$ has the block diagonal form:
\begin{equation*}
\b=
\begin{pmatrix}
\b_0&&\\
&\ddots&\\
&&\b_0\\
\end{pmatrix}
=\b_0\otimes1\in\Mat_d(\F)\otimes_\F\Mat_m(\F)=\Mat_n(\F)
\end{equation*}
for some $\b_0\in\Mat_d(\F)$, where $d={\rm deg}_\F(\b)$ and $n=md$; the centralizer $\B$ of $\E=\F[\b]$ in $\Mat_n(\F)$ is thus equal to $\Mat_m(\E)$, equipped with the invo\-lu\-tion $\s$ acting componentwise;
\item
the order $\bb=\aa\cap\B$ is the standard maximal order of $\Mat_m(\E)$.
\end{enumerate}
\end{definition}

In conclusion, the following corollary refines Theorem~\ref{PIMAIN}.

\begin{corollary}
\label{StandardStableType}
Let $\pi$ be a $\s$-self-dual cuspidal representation of $\G$. 
Then $\pi$~contains a $\s$-self-dual type attached to a $\s$-standard stratum. 
\end{corollary}

\begin{remark}
\label{Turron}
Let  $\pi$   be  a  $\s$-self-dual   cuspidal  representation  of   $\G$,  and
$\t\in\Cc(\aa,\b)$   be    a   simple    character   in   $\pi$    such   that
$\t\circ\s=\t^{-1}$  and $\s(\b)=-\b$.  Let  $\E$ denote  the field  extension
$\F[\b]$ and write $\E_\so=\E^\s$. Let $\T$ denote the maximal tamely ramified
sub-extension of $\E/\F$,  that is, the tame parameter field  of the endo-class
of $\pi$, and write $\T_\so=\T^\s$. 
\begin{enumerate}
\item
The canonical homomorphism $\E_\so\otimes_{\F_\so}\F\to\E$ of $\E_\so\otimes_{\F_\so}\F$-modules is an isomorphism.
\item
The extensions $\E/\E_\so$ and $\T/\T_\so$ have the same ramification index. 
\end{enumerate}
For the first property, see Lemma~\ref{L2} and its proof. The second one 
follows from the fact that~$\E$ is totally wildly ramified over $\T$ and $p$ 
is odd, thus $[\E:\T]$ is odd. 
\end{remark}

\subsection{Classification of $\s$-self-dual types}

From now on, we will abbreviate 
\emph{$\s$-self-dual extended maximal simple type} to 
\emph{$\s$-self-dual type}.~In this para\-graph, we determine the
$\G^\s$-orbits of $\s$-self-dual types in a $\s$-self-dual cuspidal 
representation of $\G$. 

\begin{lemma}
\label{machin34}
Let $\pi$ be a cuspidal representation of $\G$ containing a $\s$-self-dual type $(\bJ,\bl)$. The $\s$-self-dual types in $\pi$ are the $(\bJ^g,\bl^g)$ for $g\in\G$ such that $\s(g)g^{-1}\in\bJ$.
\end{lemma}

\begin{proof}
By Proposition~\ref{GeorgesSaval}, any (extended maximal simple) 
type contained in $\pi$ is $\G$-conjugate to $(\bJ,\bl)$. 
Given $g\in\G$, we have $(\bl^g)^\s=(\bl^\s)^{\s(g)}$ and 
$(\bl^g)^\vee=(\bl^\vee)^g$. Thus $(\bJ^g,\bl^g)$ is $\s$-self-dual if and 
only if $\s(g)g^{-1}$ normalizes $\bl$, that is $\s(g)g^{-1}\in\bJ$. 
\end{proof}

\begin{corollary}
Let $(\bJ,\bl)$ be a $\s$-self-dual type in $\G$. 
There is a maximal simple stratum $[\aa,\b]$~in $\Mat_n(\F)$
such that:
\begin{enumerate}
\item $\aa$ is $\s$-stable and $\s(\b)=-\b$, 
\item $\bJ=\bJ(\aa,\b)$ and the simple character $\t$ associated to $\bl$ belongs to 
$\Cc(\aa,\b)$.
\end{enumerate}
\end{corollary}

\begin{proof}
Let $(\bJ,\bl)$ be a $\s$-self-dual type in $\G$. It induces to a $\s$-self-dual cuspidal~repre\-sen\-ta\-tion $\pi$ of $\G$. Let $(\bJ_0,\bl_0)$ be a $\s$-self-dual type in $\pi$ defined with respect to a simple stratum $[\aa_0,\b_0]$ such that $\aa_0$ is $\s$-stable and $\s(\b_0)=-\b_0$. Then $(\bJ,\bl)=(\bJ_0^g,\bl_0^g)$ for some $g\in\G$~such that $\g=\s(g)g^{-1}\in\bJ_0$. We thus may assume that $(\bJ,\bl)$ is defined with respect to the maximal simple stratum $[\aa_0^g,\b_0^g]$. We have $\s(\aa_0^g)=(\aa_0^\g)^g$ which is equal to $\aa_0^g$ since $\bJ_0$ is contained in the normalizer of $\aa_0$. The result now follows from Proposition~\ref{Pompidou}.
\end{proof}

\begin{lemma}
\label{parcimonie}
Let $[\aa,\b]$ be a $\s$-standard maximal simple~stra\-tum in $\Mat_n(\F)$
in the sense of Definition {\rm\ref{sstdmss}}.
Write $\E=\F[\b]$ and $\E_\so=\E^\s$. Let $g\in\G$ and suppose that 
$\s(g)g^{-1}\in\bJ=\bJ(\aa,\b)$. 
\begin{enumerate}
\item
If $\E$ is unramified over $\E_\so$, then $g\in\bJ\G^\s$.
\item
If $\E$ is ramified over $\E_\so$, and $\w_\E$ is a uniformizer of $\E$, then:
\begin{enumerate}
\item
there is a unique integer $i$ such that
$0\<2i\<m$ and $g\in\bJ t_i\G^\s$, where
\begin{equation}
\label{defti}
t_i={\rm diag}(\w_\E,\dots,\w_\E,1,\dots,1)\in\B^\times=\GL_m(\E)
\end{equation} 
with $\w_\E$ occurring $i$ times;
\item 
the double cosets $\bJ t_i\G^\s$, $0\<i\<\lfloor m/2\rfloor$, are all distinct.
\end{enumerate}
\end{enumerate}
\end{lemma}

\begin{proof}
For  any  group  $\Ga$  equipped  with  an  action  of  $\s$,  we  will  write
$\CH^1(\s,\Ga)$ for the first~co\-homology set of ${\Gal(\F/\F_{\so})}$
in $\Ga$.
Write
$\g=\s(g)g^{-1}$.   The   identity    $\s(\g)=\g^{-1}$   implies   that   $\g$
has~valua\-tion  $0$  in  $\bJ$.  We  thus  have  $\g\in\J=\J(\aa,\b)$.  Write
$\J^1=\J^1(\aa,\b)$ and identify $\J/\J^1$ with $\GL_m(\ee)$, denoted $\GG$, 
as in \eqref{nonoc}.
Let $x$ denote the image of $\g$ in $\GG$. It satisfies $x\s(x)=1$. 

If $\E$ is unramified over $\E_\so$, then $x=\s(y)y^{-1}$ for some $y\in\GG$, thus:
\begin{equation}
\label{trainpoitiers}
\s(a^{-1})\g a = \s(a^{-1}g)g^{-1}a \in \J^1
\end{equation} 
for some  $a\in\J$. Since $\J^1$ is  a pro-$p$-group and $p\neq  2$, the first
cohomology set $\CH^1(\s,\J^{1})$
is tri\-vial.  The left hand side of
\eqref{trainpoitiers} can thus be written $\s(j)j^{-1}$ for some $j\in\J^{1}$, 
thus we have $g\in\J\G^\s$. 

Suppose now that  $\E$ is ramified over $\E_\so$, so  that $\s$ acts trivially
on $\ee$. 
We may and will assume that $\w_\E$ has been chosen such that 
$\s(\w_\E)=-\w_\E$.
Then $x$ is conjugate in $\GG$ to a class $\d\J^1$ where: 
\begin{equation*}
\d=\d_i={\rm diag}(-1,\dots,-1,1,\dots,1)\in\bb^\times\subseteq\GL_m(\E)
\end{equation*}
with $-1$ occurring $i$ times for some $i\in\{0,\dots,m\}$. We thus have
$\s(a)\g a^{-1}\in\d\J^1$ for some $a\in\J$. 
Notice that $\d t_i=\s(t_i)$. If
we write $h=t_i^{-1}xg$, we get $\s(h)h^{-1}\in \J^{1 t_i}$. 
Since $\J^{1 t_i}$ is a $\s$-stable pro-$p$-group,
the set $\CH^1(\s,\J^{1 t_i})$ is trivial, 
thus $h\in\J^{1 t_i}\G^\s$, which implies that $g\in\J t_i\G^\s$.

Now suppose that $\bJ t_i\G^\s=\bJ t_k\G^\s$ for some integers 
$0\<i,k\<m$.
Then $\d_k=\s(a)\d_i a^{-1}$ for some $a\in\bJ$.
If we write $a=ut^r$~for some $r\in\ZZ$ and $u\in\J$, 
then the images of $\d_k$ and $(-1)^r\d_i$ in $\GG$ are conjugate, 
thus either $r$ is even and $k=i$, or $r$ is odd and $k=m-i$.

Finally, we have $\bJ t_i\G^\s=\bJ t_{m-i}\G^\s$ since 
$t_{m}\in\bJ$, $t_i^2\in\G^\s$ and the group 
of permutation matrices in $\B^\times=\GL_m(\E)$
is contained in $\bJ\cap\G^\s$.
\end{proof}

\begin{remark}\label{rmk:index}
If $\E$ is ramified over $\E_\so$, then the pairs
\begin{equation*}
(\bJ^{t_i},\bl^{t_i}),
\quad
i\in\{0,\dots,\lfloor m/2\rfloor\},
\end{equation*}
where $t_i$ is defined by~\eqref{defti}, form a set of representatives of the 
$\G^\s$-conjugacy classes of $\s$-self-dual types in $\pi$. 
The integer $i$ is called the  \emph{index} of the $\G^\s$-conjugacy class. If
one identifies the quotient $\J(\aa,\b)^{t_i}/\J^1(\aa,\b)^{t_i}$ with 
$\GL_m(\ee)$ via 
\begin{equation*}
\J(\aa,\b)^{t_i}/\J^1(\aa,\b)^{t_i} 
\simeq
\J(\aa,\b)/\J^1(\aa,\b)
\simeq
\U(\bb)/\U^1(\bb)
\simeq
\GL_m(\ee),
\end{equation*}
then $\s$ acts on $\GL_m(\ee)$ by conjugacy by the diagonal element
\begin{equation*}
\d_i = {\rm diag}(-1,\dots,-1,1,\dots,1),
\end{equation*}
where $-1$ occurs $i$ times, and 
the group $(\J(\aa,\b)^{t_i}\cap\G^\s)/(\J^1(\aa,\b)^{t_i}\cap\G^\s)$ of 
$\s$-fixed points
identifies with the Levi subgroup 
$(\GL_i\times\GL_{m-i})(\ee)$ of $\GL_m(\ee)$. 
\end{remark}

The inconvenience of the extension $\E/\E_\so$ 
is that it is not canonically determined by $\pi$. 
We remedy this in the next paragraph.


\subsection{The quadratic extension $\T/\T_{\so}$} 

Let $\TT\in\Ee(\F)$ be an endo-class of degree $d$, such that $\TT^\s=\TT^\vee$. 
By Theorem \ref{THETAMAINNONMAX}, given any~multi\-ple $n$ of $d$,
there are a maximal simple stratum $[\aa,\b]$ in $\Mat_n(\F)$ and 
a simple~char\-ac\-ter $\t\in\Cc(\aa,\b)$ of endo-class $\TT$ 
such that $\t\circ\s=\t^{-1}$,
the order $\aa$ is $\s$-stable and $\s(\b)=-\b$.
Thus $\E=\F[\b]$, its centralizer $\B$ 
and the maximal order $\bb=\aa\cap\B$ are stable by $\s$.

Denote by $\E_{\so}$ the field of $\s$-fixed points in $\E$,
by $\T$ the maximal tamely ramified sub-ex\-ten\-sion of~$\E$ over $\F$, 
and set $\T_{\so}=\T\cap\E_{\so}$.
Note that $\T$ is the~tame parameter field of $\TT$, 
and that $d$ is the degree $[\E:\F]$. 
We also write $n=md$.

\begin{lemma}
The $\F_{\so}$-isomorphism class of the extension 
$\T/\T_{\so}$ only depends on $\TT$. 
Namely, if $\T'/\T_{\so}'$ is another extension obtained 
from $\TT$ as above, then there is an isomorphism
$\phi:\T\to\T'$ of $\F_{\so}$-algebras 
such that $\phi(\T_{\so}^{})=\T_{\so}'$.
\end{lemma}

\begin{proof}
Let $[\aa',\b']$ be a maximal simple stratum in $\Mat_{n'}(\F)$ for some 
multiple $n'$ of $d$, and let
$\t'$~be a simple~char\-ac\-ter in $\Cc(\aa',\b')$
of endo-class $\TT$ such that $\t'\circ\s=\t'^{-1}$,
the order $\aa'$ is $\s$-stable and $\s(\b')=-\b'$.
Associated with this, there are a tamely ramified extension $\T'$ of $\F$
and its $\s$-fixed points $\T'_{0}$.

Suppose first that $\t'=\t$.
Write $\J^1$ for the maximal normal 
compact open pro-$p$-subgroup of the $\G$-normalizer of $\t$. 
By \cite[Proposition 2.6]{BHEffectiveLC}, 
one has $\T'=\T^x$ for some $x\in\J^1$. 
Since $\T'$ is~sta\-ble by~$\s$, the element $y=\s(x)x^{-1}\in\J^1$ 
normalizes $\T$, thus centralizes it by \cite[Proposition~2.6]{BHEffectiveLC}.
Applying Hilbert's Theorem 90 to the element $y$~in the centralizer 
$\G_\T$ of $\T$ in $\G$ implies that $x\in\G_\T\G^\s$.
It follows that $\T'$ is $\G^\s$-conjugate to $\T$.
The $\F_{\so}$-isomorphism class~of $\T/\T_{\so}$ thus only depends on 
$\t$, not on the simple stratum $[\aa,\b]$ such that $\t\in\Cc(\aa,\b)$. 

Suppose now that $n'=n$.
Since $\t$, $\t'$ have the same endo-class, 
we have $\t'=\t^g$ for some~$g\in\G$.
Since they are both $\s$-self-dual, we have $\s(g)g^{-1}\in\bJ$,
where $\bJ$ is the $\G$-normalizer of $\t$.
By Lemma \ref{parcimonie}, we may even assume, up to $\G^\s$-conjugacy, 
that $g\in\B^\times$, thus $\s(g)g^{-1}\in\B^\times$ centralizes $\T$.
Thanks to the first case, we may also assume that $\aa'=\aa^g$ and 
$\b'=\b^g$. 
We thus have $\T'=\T^g$ with $\s(g)g^{-1}\in\G_\T$.
By the same cohomological argument as above, we deduce that 
$\T'$ is $\G^\s$-conjugate to $\T$.

We now consider the general case.
Thanks to the first two cases and Corol\-la\-ry \ref{StandardStableType},
we may assume,
replacing $\t$,~$\t'$~by $\G$-conjugate characters if necessary, 
that $[\aa,\b]$ and $[\aa',\b']$ are $\s$-standard.
We~thus may transfer $\t$~and $\t'$ to $\GL_{d}(\F)$ without changing the 
$\F_{\so}$-isomorphism classes of $\T/\T_{\so}$ and $\T'/\T'_{\so}$.
We are thus reduced to the previous case.
\end{proof}

Now let $\pi$ be a $\s$-self-dual cuspidal representation of $\G$.
Its endo-class, denoted~$\TT$, has degree 
dividing $n$ and satis\-fies $\TT^\s=\TT^\vee$. 
Associated~with it, there is thus~a~qua\-dra\-tic extension 
$\T/\T_{\so}$,~uni\-que\-ly determined up to $\F_{\so}$-isomorphism.
Let~us record this fact for~fu\-tu\-re reference.

\begin{proposition}
\label{TT0canonique}
The $\F_{\so}$-isomorphism class of $\T/\T_{\so}$ depends only 
on the endo-class of $\pi$.
\end{proposition}

Unlike $\E/\E_{\so}$, the quadratic extension $\T/\T_{\so}$ is canonically attached to 
$\pi$. 
By applying Lemmas~\ref{machin34} and~\ref{parcimonie}
together with Remarks \ref{Turron} and \ref{rmk:index}, 
we get the following proposition.

\begin{proposition}
\label{classesdetypesstables7}
Let $\pi$ be a $\s$-self-dual cuspidal representation of $\G$,
and $\T/\T_{\so}$ be the quadratic extension canonically attached to it. 
\begin{enumerate}
\item
If $\T$ is unramified over $\T_{\so}$, 
the $\s$-self-dual types contained in $\pi$ 
form a single $\G^\s$-conjugacy class.
\item
If $\T$ is ramified over $\T_{\so}$, 
the $\s$-self-dual types contained in $\pi$ 
form exactly $\lfloor m/2\rfloor+1$ different $\G^\s$-conjugacy classes, 
characterized by their index. 
\end{enumerate}
\end{proposition}

\subsection{A counterexample in the semisimple case}
\label{CESST}
We  end this  section by  looking at  a natural  question which  lies slightly
outside the main thrust of this paper but which we find intriguing: namely, is
there,   for   \emph{any}~$\s$-self-dual   irreducible   representation~$\pi$,
a~$\s$-self-dual type  contained in~$\pi$. If  one requires  the type to  be a
\emph{semisimple} type (in the sense of~\cite{BKsemi,MSt}) then the answer is 
\emph{no}, as the following example shows. 

Let  $\chi$ be  a  tamely  ramified character  of  $\F^\times$  such that  the
character $\chi(\chi\circ\s)$ is ramified. 
We~consider the repre\-sentation $\pi$ of $\GL_2(\F)$ obtained by applying 
the functor or normalized parabolic induction to the character 
$\chi\otimes(\chi^{-1}\circ\s)$ 
of the Levi subgroup $\F^\times\times\F^\times$.
This is an ir\-reducible and $\s$-self-dual 
representation of level $0$. 
By looking at its cuspidal support, 
one deduces that any semisimple type in $\pi$ is conjugate to 
one of the following: 
\begin{enumerate}
\item 
the pair $(\I,\l)$ where $\I$ is the standard Iwahori subgroup 
(the one whose reduction mod $\p_\F$ is made of upper triangular matrices) 
and $\l$ is the character:
\begin{equation*}
\begin{pmatrix}
a&b\\ c&d 
\end{pmatrix}
\mapsto\chi(a)\chi(\s(d))^{-1},
\end{equation*}
\item
the pair $(\I,\l')$ where $\l'$ is the character:
\begin{equation*}
\begin{pmatrix}
a&b\\ c&d 
\end{pmatrix}
\mapsto\chi(\s(a))^{-1}\chi(d).
\end{equation*}
\end{enumerate}
Note that the latter one is conjugate to the first one by the element:
\begin{equation*}
h=
\begin{pmatrix}
0&1\\ \w&0
\end{pmatrix}
\in\GL_2(\F)
\end{equation*}
where $\w$ is a uniformizer of $\F$.
Thus any semisimple type in $\pi$ is conjugate to $(\I,\l)$.

Now assume that $\pi$ contains a~$\s$-self-dual semisimple type. 
There is then a $g\in\GL_2(\F)$ such that:
\begin{equation*}
\s(\I^g)=\I^g,
\quad
\l^g\circ\s=(\l^g)^{-1}.
\end{equation*}
The first condition says that $\g=\s(g)g^{-1}$ normalizes $\I$. 
The second one gives us $(\l\circ{\s})^\g=\l^{-1}$. 
But $(\l\circ{\s})^{-1}=\l'=\l^h$,
thus $h\g$ normalizes $\l$. 
Let us write $\N$ for the normalizer of $\I$ in $\GL_2(\F)$. 
It is generated by $\I$ and $h$,
and carries a valuation homomorphism $v:\N\to\ZZ$
with kernel $\I$.
Since $\I$ is $\s$-stable, we have $v\circ\s=v$. 
Since $\s(\g)=\g^{-1}$ we have $v(\g)=0$,
thus $\g\in\I$. Since $h\g$ and $\I$ normalize $\l$, this implies that $h$ 
normalizes $\l$: a contradiction. 

\begin{remark}
The example above shows that there is no~$\s$-self-dual semisimple type for
$\pi$.
This also implies that there is
no~$\s$-self-dual type for~$\pi$ which is a cover of type for its cuspidal 
support (in the sense of \cite{BKviatypes}). 
However, writing~$\K=\GL_2(\o)$  for  the  standard maximal  compact
subgroup of~$\GL_2(\F)$ and using the other notation above, the
pair~$(\K,\cind_{\I}^{\K}\l)$ is a type for~$\pi$, which is 
$\s$-self-dual. Thus the question of whether or not all
irreducible~$\s$-self-dual representations of~$\G$~possess 
a~$\s$-self-dual
type remains as an interesting open question. 
\end{remark}

\section{Generic $\s$-self-dual types} 

For this section, we place ourselves in a slightly more general setting. 
We again take~$\F_\so$ to~be~a non\-ar\-chi\-medean local field of odd residual 
characteristic~$p$, but we allow~$\G$ to be the group of~ratio\-nal points 
of any {connected} reductive group defined over~$\F_\so$ equipped with a 
non-trivial involution~$\s$ defi\-ned over~$\F_\so$. 

\subsection{}

Let~$\N$ be a~$\s$-stable unipotent subgroup of~$\G$. 

\begin{lemma}\label{lem:Uunion}
The group~$\N$ is a union of~$\s$-stable pro-$p$ subgroups.
\end{lemma}

\begin{proof}
We   write~$\N=\bigcup_{i\ge  0}   \N_i$  as   a  nested   union  of   compact
subgroups~$\N_i$   which  are   open  in~$\N$,   so  that~$\N_i\subseteq\N_j$,
for~$0\le   i\le   j$.  For   any~$u\in\N$,   there   exist~$i,j\ge  0$   such
that~$u\in\N_i$   and~$\s(u)\in   \N_j$.  Then,   taking~$k=\max\{i,j\}$,   we
have~$u\in\N_k\cap\s(\N_k)$. Thus~$\N=\bigcup_{k\ge  0} \(\N_k\cap\s(\N_k)\)$,
as required. 
\end{proof}

\begin{lemma}[{cf.~\cite[Lemma~2.1]{S2}}]
\label{lem:coset}
Let~$\K$ be a~$\s$-stable open subgroup of~$\G$, and let~$g\in\G$.
\begin{enumerate}
\item\label{lem:coset.i} 
If the double coset~$\N g\K$ is~$\s$-stable then it contains a~$\s$-stable left~$\K$-coset.
\item\label{lem:coset.ii} 
If~$g\K$ is~$\s$-stable then every~$\s$-stable left~$\K$-coset in~$\N g\K$ lies in~$\N^\s g\K$.
\item\label{lem:coset.iii} 
$(\N\K)^\s=\N^\s \K^\s$.
\end{enumerate}
\end{lemma}

\begin{proof}
\ref{lem:coset.i}~Suppose~$\N g\K$ is~$\s$-stable, so that~$\s(g)=ugk$, for 
some~$u\in\N$ and~$k\in\K$. By Lemma~\ref{lem:Uunion}, there is a~$\s$-stable 
pro-$p$ subgroup~$\N_0$ of~$\N$ containing~$u$, so that~$\s(g)\in \N_0g\K$. In 
particular, the double coset~$\N_0g\K$ is~$\s$-stable. 

Now we decompose~$\N_0g\K$ as a union of~$\K$-cosets. 
Since~$\N_0 g\K/\K$ is
in bijection with~the quotient $\N_0/(\N_0\cap g\K g^{-1})$, which is finite of order a
power of~$p$ (odd), there is some coset~$h\K \subset \N_0g\K$ which 
is~$\s$-stable. 

\ref{lem:coset.ii}~Suppose~$g\K$ is~$\s$-stable so that~$g^{-1}\s(g)=k\in\K$. 
If~$ug\K$ is~$\s$-stable, then
\begin{equation*}
u^{-1}\s(u)
=u^{-1}\s(ug)\s(g^{-1})
=gk_1k^{-1}g^{-1}
\end{equation*}
for some $k_1\in\K$.
Thus the map~$\tau\mapsto u^{-1}\tau(u)$ 
defines a~$1$-cocycle in~$\CH^1(\langle\s\rangle,g\K g^{-1}\cap 
\N)$, which~is trivial, so there exists~$v\in g\K g^{-1}\cap \N$ such 
that~$u^{-1}\s(u)=v\s(v^{-1})$. Then~$ug\K=uvg\K$ and~$uv\in\N^\s$. 

\ref{lem:coset.iii}~Suppose $h\in (\N\K)^\s$. Then certainly $h\K$ is
$\s$-stable. On the other hand, $\N h\K=\N\K$ and $\K$ itself is also
$\s$-stable so applying \ref{lem:coset.ii} with $g=1$ we get that every
$\s$-stable left coset in $\N\K$ lies in $\N^\s \K$; thus $h$ is in $\N^\s
\K$. Writing $h=uk$ with $u\in \N^\s$ and $k\in \K$, the fact that $h$ is
$\s$-invariant implies $k\in \K^\s$, so $h\in \N^\s \K^\s$. 
\end{proof}

\subsection{}

We suppose from now on that~$\G$ is quasi-split. As before, a 
pair~$(\K,\tau)$, consisting of an open~sub\-group~$\K$ of~$\G$ and an 
irreducible representation~$\tau$ of~$\K$, is called~\emph{$\s$-self-dual} 
if~$\s(\K)=\K$ and~$\tau^{\s}\simeq \tau^{\vee}$. 

A \emph{Whittaker datum} for~$\G$ is a pair~$(\N,\psi)$ consisting of (the
$\F_\so$-points of) the unipotent radical~$\N$ of an $\F_\so$-Borel subgroup
of~$\G$ and a character~$\psi$ of~$\N$ such that the stabilizer of~$\psi$
in~$\G$ is~$\Z\N$, where~$\Z$ denotes the $\F_\so$-points of the centre
of~$\G$. If a Whittaker datum~$(\N,\psi)$ is~$\s$-self-dual then,
since~$\F_\so$ is not of characteristic two,~$\psi$ is trivial on~$\N^\s$. 

\begin{proposition}[{cf.~\cite[Proposition~1.6]{BH98}}]
\label{prop:Upsi}
Suppose that $\G$ is quasi-split.
Let~$(\N,\psi)$ be  a~$\s$-self-dual Whittaker datum in~$\G$  and let~$\pi$ be
an irreducible~$\s$-self-dual cuspidal representation of~$\G$ such that the 
space~$\Hom_\N(\pi,\psi)$~is one-dimensional. Suppose that~$(\bJ,\rho)$ is 
a~$\s$-self-dual pair, with~$\bJ$ a compact-mod-centre open subgroup of~$\G$, 
such that~$\pi\simeq\cind_\bJ^\G\rho$. 
\begin{enumerate}
\item\label{prop:Upsi.i} 
There exists a~$\s$-self-dual pair~$(\bJ',\rho')$ conjugate to~$(\bJ,\rho)$ such that
\[
\Hom_{\bJ'\cap\N}(\rho',\psi) \ne 0.
\]
\item\label{prop:Upsi.ii} 
The pair~$(\bJ',\rho')$ as in~\ref{prop:Upsi.i} is uniquely determined up to conjugacy by~$\N^\s$.
\item\label{prop:Upsi.iii} 
For any pair~$(\bJ',\rho')$ as in~\ref{prop:Upsi.i}, the space~$\Hom_{\bJ'\cap\N}(\rho',\psi)$ is one-dimensional.
\end{enumerate}
\end{proposition}

\begin{proof}
We follow the proof of~\cite[Proposition~1.6]{BH98} which, although it is 
written only for~$\G=\GL_n(\F)$, is valid more generally. 
{Let us write $\Vv_\rho$ for the space of $\rho$, 
and $\Hh(\G,\rho,\psi)$ for the space of 
func\-tions $\varphi:\G\to\Hom_\FC(\Vv_\rho,\FC)$ such 
that~$\varphi(ugk)=\psi(u)\varphi(g)\circ\rho(k)$, 
for all~$u\in\N$, $g\in\G$ and $k\in\bJ$.}
By the main Theorem of~\cite{Kutzko77}
(which is valid also for~$\FC$-representations), 
we have a natural~$\G$-isomorphism
\[
\Hh(\G,\rho,\psi) \simeq \Hom_\G(\cind_\bJ^\G \rho, \Ind_\N^\G \psi).
\]
In particular, we see 
that~$\dim_\FC\Hh(\G,\rho,\psi)=1$, whence (cf.~\cite[(1.8)]{BH98}) there is a
unique   double   coset~$\N   g\bJ$   which   supports   a   non-zero   element
of~$\Hh(\G,\rho,\psi)$ (that is, intertwines~$\psi$ with~$\rho$), and moreover
the   space  of~$\varphi\in\Hh(\G,\rho,\psi)$   supported   on~$\N  g\bJ$   is
one-dimensional   --   that   is,~$\Hom_{\N^g\cap   \bJ}   (\rho,\psi^g)$   is
one-dimensional. Note that~$\N^g\cap\bJ$ is a compact subgroup of~$\N^g$ so is
pro-$p$; in particular, the restriction of~$\rho$ to~$\N^g\cap\bJ$ is
semisimple. 

Applying~$\s$ and taking contragredients, we see 
that~$\Hom_{\N^{\s(g)}\cap\bJ} (\psi^{\s(g)},\rho)$ is also non-zero; 
by~semi\-sim\-pli\-ci\-ty, the same is true of~$\Hom_{\N^{\s(g)}\cap\bJ}
(\rho,\psi^{\s(g)})$ so,  by uniqueness,~$\s(g)$ lies in~$\N  g\bJ$. Since the
double coset~$\N g\bJ$ is then~$\s$-stable, Lemma~\ref{lem:coset} implies that
it contains  a~$\s$-stable coset~$h\bJ$, and  that any~$\s$-stable~$\bJ$-coset
in~$\N  g\bJ$   lies  in~$\N^\s  h\bJ$.  Then   the  pair~$({}^h\bJ,{}^h\rho)$
satisfies the hypotheses of~\ref{prop:Upsi.i}, while the uniqueness statements
in~\ref{prop:Upsi.ii} and~\ref{prop:Upsi.iii} also follow. 
\end{proof}

\subsection{}
\label{p53}

Finally in this subsection, we specialize to the case~$\G=\GL_n(\F)$, 
where~$\F/\F_\so$ is a quadratic extension and~$\s$ the Galois involution as 
in the rest of the paper. 
By the $\s$-self-dual type Theorem~\ref{PIMAIN}~to\-ge\-ther 
with~\cite[Corollary~1]{GK} (or~\cite[III.5.10]{Vig96} in the modular case), 
the hypotheses~of~Pro\-po\-si\-tion~\ref{prop:Upsi} are satisfied for any 
irreducible~$\s$-self-dual cuspidal representation~$\pi$ of~$\GL_n(\F)$. 

\begin{remark}
\label{guderian}
Note that \cite[III.5.10]{Vig96} is for cuspidal representations with 
coefficients in an algebraic closure $\overline{\FF}_\ell$ of a finite field 
of characteristic $\ell\neq p$ only, 
but one can easily extend it to representations with coefficients 
in a general $\FC$ of characteristic $\ell$.
Indeed, if $\pi$ is a cuspidal $\FC$-representation, 
then, by twisting it by a character, 
we may assume that its central character has values in 
$\overline{\FF}_\ell\subseteq\FC$. 
Then by 
\cite[II.4]{Vig96} there is a cuspidal $\overline{\FF}_\ell$-representation $\pi_1$ 
such that $\pi$ is isomorphic to $\pi_1\otimes_{\overline{\FF}_\ell}\FC$.
It now follows that the hypotheses~of~Pro\-po\-si\-tion~\ref{prop:Upsi} are 
satisfied by $\pi$, since they are satisfied by $\pi_1$.
\end{remark}

\begin{proposition}
\label{brunau}
Let $\pi$ be a~$\s$-self-dual cuspidal representation 
of~$\GL_n(\F)$,
and let~$\T/\T_\so$ be the quadratic extension 
associated with it by Proposition {\rm\ref{TT0canonique}}.
Let $d$ be the degree of the endo-class of $\pi$,
and write $n=md$. 
\begin{enumerate}
\item
Let $(\N,\psi)$ be a $\s$-self-dual Whittaker datum in~$\GL_n(\F)$. 
Then the representation $\pi$ contains a $\s$-self-dual type $(\bJ,\bl)$
such that 
\begin{equation}
\label{rathenau}
\Hom_{\bJ\cap\N}(\bl,\psi) \ne 0.
\end{equation}
The pair $(\bJ,\bl)$ is uniquely determined up to conjugacy by~$\N^\s$
and~$\Hom_{\bJ\cap\N}(\bl,\psi)$ has dimension $1$. 
\item
\label{gsst}
The set of all $\s$-self-dual types contained in $\pi$ 
and satisfying \eqref{rathenau}
for some $\s$-self-dual Whittaker datum $(\N,\psi)$ is a single 
$\GL_n(\F_\so)$-conjugacy class. 
\item
If $\T$ is unramified over $\T_\so$,
the conjugacy class in \ref{gsst} is the unique $\GL_n(\F_\so)$-conjugacy 
class of $\s$-self-dual types in $\pi$.
\item
If $\T$ is ramified over $\T_\so$,
the conjugacy class in \ref{gsst} is the unique $\GL_n(\F_\so)$-conjugacy 
class of $\s$-self-dual types in $\pi$ of  
index $\lfloor m/2\rfloor$
(see Remark~\ref{rmk:index}). 
\end{enumerate}
\end{proposition} 

\begin{proof}
Assertion (i) follows from Proposition \ref{prop:Upsi}
and Assertion (ii) follows from (i) together with~the fact that any two $\s$-self-dual 
Whittaker data in~$\GL_n(\F)$ are $\GL_n(\F_\so)$-conjugate.
Indeed, if $(\N',\psi')$~is a $\s$-self-dual Whittaker datum, 
it can be written $(\N^g,\psi^g)$ for some $g\in\GL_n(\F)$ such that 
$\s(g)g^{-1}$~is in $\Z\N$.
Writing $\s(g)g^{-1}=zu$ with $z\in\Z\simeq\F^\times$ and $u\in\N$,
we get $z\s(z)=u\s(u)=1$.
The result now follows from a simple cohomological argument. 

Assertion (iii) follows from Proposition \ref{classesdetypesstables7}. 

We now prove (iv).
By Proposition \ref{classesdetypesstables7}, 
there are~$\lfloor m/2\rfloor+1$ conjugacy class of~$\s$-self-dual~types~in 
$\pi$
and each conjugacy class has an
\emph{index}~$i$ as in Remark~\ref{rmk:index}. 
If~$(\bJ,\bl)$ is a
$\s$-self-dual type~with index~$i$ then, identifying~$\J/\J^1$
with~$\GL_m(\ee)$, the involution~$\s$ acts via conjugation 
by the diagonal~element 
\begin{equation*}
\d=\d_i=\diag(-1,\ldots,-1,1\ldots,1)
\end{equation*} 
with~$-1$ occurring~$i$ times. 

If~$(\bJ,\bl)$ is as in (ii),
then the image $\EuScript{U}$ of~{$\J\cap\N$}
in~$\GL_m(\ee)$ is a~$\s$-stable maximal unipotent~sub\-group on which~$\psi$ 
induces a~$\s$-{self-dual} 
character $\overline{\psi}$. 
By~\cite[Remark~4.15 and Theorem~3.3]{StPa}, the character 
$\overline{\psi}$ is non-degenerate.

Now there is a $g\in\GL_m(\kk_\E)$ such that $\EuScript{U}^g$ is equal to 
$\EuScript{N}$, the standard maximal unipotent subgroup.
Since $\EuScript{U}$ and $\EuScript{N}$ are $\s$-stable, 
the element $\g=\s(g)g^{-1}$ normalizes $\EuScript{N}$.
It thus can be written 
$\g= n_0t$ with $n_0\in\EuScript{N}$ and $t$ diagonal.
Since $\g^{-1}=\s(\g)=\d\g\d^{-1}$, we have $t^{-1}=t$.
Write $\d'=t\d$ and let $\s'$ be the involution of $\GL_m(\kk_\E)$ 
given by conjugacy by $\d'$.
Then 
\begin{equation*}
n_0\s'(n_0) 
= n_0t\d n_0 \d^{-1}t^{-1}
= \g\d\g\d^{-1} = 1.
\end{equation*}
Since $\EuScript{N}$ is a $p$-group with $p$ odd, 
there is $n_1\in\EuScript{N}$ such that $n_0=\d' n_1^{}\d'^{-1} n_1^{-1}$.
Write $h=n_1^{-1}g$.
Then $\EuScript{U}^h=\EuScript{N}$ and $\s(h)h^{-1}=t$.
Thus, replacing $g$ by $h$, we may assume that $n_0=1$.
Moreover, if we identify $\EuScript{U}$ with $\EuScript{N}$,
then $\s$ is replaced by $\s'$, that is, conjugacy by the diagonal matrix $\d'$.

Now consider the $\s'$-self-dual non-degenerate character 
$\psi'=(\overline{\psi})^g$ 
of $\EuScript{N}$.
There are $a_1,\dots,a_{m-1}\in\kk_\E^\times$ such that
\begin{equation*}
\psi'(n) = \varphi(a_1n_{1,2}+\dots+a_{m-1}n_{m-1,m})
\end{equation*}
for all $n\in\EuScript{N}$,
where $\varphi$ is a fixed non-trivial character of $\kk_\E$.
The fact that $\psi'$ is $\s'$-self-dual implies that $\d'_{k+1}=-\d'_k$ 
for all $k=1,\dots,m-1$.
Since the number of $-1$ and $1$ differ by at most $1$, 
and since $i\<\lfloor m/2\rfloor$ by definition,
it follows that $i=\lfloor m/2\rfloor$.
\end{proof}

\begin{definition}
{We call a type in the conjugacy class of Proposition \ref{brunau}\ref{gsst} 
a \textit{generic $\s$-self-dual type} for~$\pi$.}
\end{definition} 

{Proposition \ref{brunau} thus says that,
when $\T$ is unramified over $\T_\so$,
any $\s$-self-dual type contained in $\pi$ is generic, and,
when $\T$ is ramified over $\T_\so$,
a $\s$-self-dual type contained in $\pi$ is generic if and only if its index 
is $\lfloor m/2\rfloor$.}

\subsection{}
\label{sec:inv}

{We continue with the notation of Paragraph \ref{p53}.
The main result of this paragraph is Lemma \ref{lem:e00},
which will be useful in Sections \ref{exWhitt} and \ref{S7VS}.} 

We assume, in this paragraph, that $\pi$ is a $\s$-self-dual 
{\textit{supercuspidal}} representation of $\G=\GL_n(\F)$.
{Recall that a cuspidal representation of $\G$ is supercuspidal 
if it does not occur as a subquotient of the parabolic induction 
of an irreducible representation of a proper Levi subgroup of $\G$.}

By Proposition \ref{brunau},
this representation 
con\-tains a ge\-ne\-ric $\s$-self-dual type~$(\bJ,\bl)$,
uniquely determined up to $\G^\s$-conjugacy.
Fix a maximal~sim\-ple stra\-tum $[\aa,\b]$ such that 
$\bJ=\bJ(\aa,\b)$ 
with~$\aa$~a $\s$-stable hereditary order and~$\s(\b)=-\b$.
Let~$\E$ denote the $\F$-extension $\F[\b]$.
Let $\T$ be the~maxi\-mal tamely ramified sub-ex\-ten\-sion of~$\E$ over 
$\F$, and let $\T_0$ denote its $\s$-fixed points.
We also~write $m=n/\deg_\F(\b)$.

\begin{proposition}[\cite{VS} Proposition {8.1}] 
\label{polonius}
Let $\pi$ be a $\s$-self-dual {supercuspidal} 
representation of the group $\GL_n(\F)$. 
If $\T/\T_\so$ is ramified, then either~$m=1$ or~$m$ is even.
\end{proposition}

\begin{remark}
\begin{enumerate}
\item
Note that Proposition \ref{polonius} does not hold if $\pi$ is only assumed to 
be $\s$-self-dual cuspidal: see \cite[Remark 7.5]{VS}.
\item
In the situation of Proposition \ref{polonius}, 
but with $\T/\T_\so$ unramified instead of ramified,
it is proved in \cite[Proposition 8.14]{VS} that $m$ is odd,
but we will not need this result.
\end{enumerate}
\end{remark}

The parahoric subgroup~$\aa^\times$ of $\G$ is $\s$-stable; 
thus~$\aa^\times\cap\G^\s$ 
is a parahoric subgroup of~$\G^\s$ and has the form $\aa_\so^\times$, 
for some~$\o_\so$-hereditary order~$\aa_\so$ in~$\Mat_n(\F_\so)$. 
Let~$e_\so$ denote the~$\o_\so$-period of~$\aa_\so$.
As~usual, we also write $\B$ for the centralizer of $\E$ in $\Mat_n(\F)$.
Then $\bb=\aa\cap\B$ is an $\o_\E$-hereditary order in $\B$,
and $\bb_\so=\bb\cap\aa_\so$ is an $\o_{\E_\so}$-hereditary order in 
$\B^\s\simeq\Mat_m(\E_\so)$.

\begin{lemma}\label{lem:e00}
Let $\pi$ be a $\s$-self-dual supercuspidal representation of $\GL_n(\F)$.
The orders $\aa_\so$ and~$\bb_\so$ {defined as above} are principal.
Moreover:
\begin{enumerate}
\item
If $\T/\T_\so$ is ramified and $m\ne 1$, 
then $e_\so=2e(\E_\so/\F_\so)$
and $\bb_\so$ has $\o_{\E_\so}$-period $2$.
\item
Otherwise, 
we have $e_\so=e(\E_\so/\F_\so)$
and $\bb_\so$ is maximal.
\end{enumerate}
\end{lemma}

\begin{proof}
Note that $\aa_\so$ is a hereditary order of $\Mat_n(\F_\so)$ 
normalized by $\E_\so^\times$.
Suppose first that $\T/\T_\so$ is unramified.
Then one may assume 
that $(\bJ,\bl)$ is attached to a $\s$-standard stratum
(Definition~\ref{sstdmss}), 
thus $\bb$ is the standard maximal order of $\Mat_m(\E)$.
It follows that $\bb_\so$ is the standard maximal order of $\Mat_m(\E_\so)$, 
thus $\aa_\so$ is the unique hereditary order of $\Mat_n(\F_\so)$ 
normalized by $\E_\so^\times$. 
It is thus principal, and its period is equal to $e(\E_\so/\F_\so)$.

Suppose now that $\T/\T_\so$ is ramified.
By Proposition \ref{polonius},
the integer $m$ is either $1$ or even,
and~Pro\-po\-si\-tion \ref{brunau} says that 
the index of~$(\bJ,\bl)$ is $\lfloor m/2\rfloor$.
Suppose first that $m=1$.
Then $\bb_\so$ identifies with $\o_{\E_\so}$.
It is thus maximal and $\aa_\so$ is principal of 
period $e(\E_\so/\F_\so)$, as in the unramified case.

Suppose now that $m=2r$ for some $r\>1$.
Since $(\bJ,\bl)$ has index $r$, we may assume that 
\begin{equation*}
\bb = 
\begin{pmatrix}
\o_{\E} &\cdots& \o_{\E} & \p_{\E}^{-1} &\cdots& \p_{\E}^{-1} 
\\
\vdots & & \vdots & \vdots & & \vdots \\
\o_{\E} &\cdots& \o_{\E} & \p_{\E}^{-1} &\cdots& \p_{\E}^{-1} 
\\
\p_{\E} &\cdots& \p_{\E} & \o_{\E} &\cdots& \o_{\E} \\
\vdots & & \vdots & \vdots & & \vdots \\
\p_{\E} &\cdots& \p_{\E} & \o_{\E} &\cdots& \o_{\E} 
\end{pmatrix}
\end{equation*}
where each block has size $r\times r$.
Since $\E$ is ramified over $\E_\so$, 
a simple calculation based on the fact that $\p_\E^{-1}\cap\E_\so=\o_{\E_\so}$ 
shows that $\bb_\so$ is principal of period $2$. 
We have a similar description of $\aa$;~if~we 
let $\aa_{\E/\F}$ denote the unique $\o$-order of $\End_\F(\E)$ 
normalized by $\E^\times$, then we have
\begin{equation*}
\aa = 
\begin{pmatrix}
\aa_{\E/\F} &\cdots& \aa_{\E/\F} & \p_{\E/\F}^{-1} &\cdots& \p_{\E/\F}^{-1}
\\
\vdots & & \vdots & \vdots & & \vdots \\
\aa_{\E/\F} &\cdots& \aa_{\E/\F} & \p_{\E/\F}^{-1} &\cdots& \p_{\E/\F}^{-1}
\\
\p_{\E/\F} &\cdots& \p_{\E/\F} & \aa_{\E/\F} &\cdots& \aa_{\E/\F} \\
\vdots & & \vdots & \vdots & & \vdots \\
\p_{\E/\F}&\cdots& \p_{\E/\F} & \aa_{\E/\F} &\cdots& \aa_{\E/\F} 
\end{pmatrix}
\end{equation*}
where $\p_{\E/\F}=\p_{\E}\aa_{\E/\F}$ is the Jacobson radical of 
$\aa_{\E/\F}$.
As in the $m=1$ case, $(\aa_{\E/\F})_\so$ is principal of period 
$e(\E_\so/\F_\so)$.
A simple calculation shows that $\aa_\so$ is principal of period 
$2e(\E_\so/\F_\so)=e(\E/\F_\so)$.
\end{proof}

\begin{remark}\label{rmk:pilambda}
Let~$\varpi_\bl$ be an element of the principal order $\bb_\so$ 
generating its Jacobson radical.~We 
get that~$\bJ\cap\G^\s$ is generated by 
$\varpi_\bl$ and $\J\cap\G^\s$. 
\end{remark}

\section{Distinction and Whittaker functions}
\label{exWhitt}

{We return to the notation of~the rest of the paper so that~$\F/\F_\so$ is a 
quadratic extension,~$\G=\GL_n(\F)$ for some $n\>1$ 
and~$\s$ is the involution~on $\G$ induced by the Galois involution.}

\subsection{Distinguished linear forms and Whittaker functions}

In this subsection we begin to look at the question of distinction. 
Recalling that~$\P=\P_n(\F)$ denotes the standard
mirabolic~sub\-group of~$\G$, we will prove the following analogue of a result 
of Ok~\cite{Ok}.

\begin{proposition}\label{Ok for types}
Let $(\bJ,{\bl})$ be a $\s$-self-dual 
type such that~the compactly induced representation 
$\pi=\cind_{\mathbf{J}}^\G{\bl}$ is distinguished. Then 
\[
\Hom_{\bJ^\s\cap \P}({\bl},\one)=\Hom_{\bJ^\s}({\bl},\one).
\]
\end{proposition}

Recall that saying that~$(\bJ,{\bl})$ is distinguished means that the space on 
the right hand side is non-zero.
The condition in the proposition that~the $\s$-self-dual cuspidal 
representation $\pi$ is dis\-tinguished is \emph{a priori} 
weaker than this; {however, see Remark \ref{coincoin}.}

In order to prove this proposition, we  need a small lemma which again applies
in a more general~set\-ting. Let~$\GGG$ be a locally profinite group, let~$\KKK$
be an open subgroup of~$\GGG$ and let~$\HHH'\subseteq\HHH$ 
be closed sub\-groups
of~$\GGG$. Let~$\rho$ be a smooth representation of~$\KKK$ and let~$\tau$ be a
smooth representation~of~$\HHH$. 
For~$g\in\GGG$, we write~$\cind_\KKK^{\KKK g\HHH}\rho$ for the subspace 
of~$\cind_\KKK^\GGG\rho$ consisting of
functions with support con\-tained in~$\KKK g\HHH$. Then the Mackey
decomposition gives 
\[
\cind_{\KKK\cap\HHH}^\HHH \Res^\KKK_{\KKK\cap\HHH} \rho \simeq 
\cind_\KKK^{\KKK\HHH}\rho \subseteq \bigoplus_{\KKK\backslash\GGG/\HHH} 
\cind_\KKK^{\KKK g\HHH}\rho = \Res^\GGG_\HHH \cind_\KKK^\GGG \rho
\]
and,  by  Frobenius reciprocity  applied  to  the  natural projection  in  the
opposite direction, we get natural maps 
\[
\Hom_{\KKK\cap\HHH}(\rho,\tau) \simeq 
\Hom_\HHH(\cind_\KKK^{\KKK\HHH}\rho,\tau) 
\hookrightarrow \Hom_\HHH(\cind_\KKK^\GGG\rho,\tau).
\]
We get similar maps with~$\HHH$ replaced by~$\HHH'$ and the following diagram
commutes: 
\[
\xymatrix{
\Hom_{\KKK\cap\HHH}(\rho,\tau)\ar[r]^{\sim\quad}\ar[d]_{{\iota_0}} & 
\Hom_{\HHH}(\cind_\KKK^{\KKK\HHH} \rho,\tau)\ar[r]\ar[d] & \Hom_\HHH(\cind_\KKK^\GGG\rho,\tau)\ar[d]^{\iota_1} \\
\Hom_{\KKK\cap\HHH'}(\rho,\tau)\ar[r]^{\sim\quad} & \Hom_{\HHH'}(\cind_\KKK^{\KKK\HHH'} \rho,\tau)\ar[r] & \Hom_{\HHH'}(\cind_\KKK^\GGG\rho,\tau)
}
\]
where the vertical maps are given by natural inclusion. 

\begin{lemma}\label{lem:bijection}
Suppose, in the situation above, that the inclusion~$\iota_1$ is an equality. 
Then the~inclu\-sion~$\iota_0$ is also an equality. 
\end{lemma}

\begin{proof}
Certainly the inclusion~$\iota_0$ is an injection. Conversely, 
any~$\varphi\in\Hom_{\KKK\cap\HHH'}(\rho,\tau)$ corresponds~to a 
map~$\Phi\in\Hom_{\HHH'}(\cind_\KKK^\GGG\rho,\tau)$ which is trivial on all 
the summands~$\cind_{\KKK g\HHH'}\rho$ with~$g\not\in\KKK\HHH'$. Then, 
since~$\iota_1$ is an equality,~$\Phi\in\Hom_{\HHH}(\cind_\KKK^\GGG\rho,\tau)$; 
moreover, it is trivial on all~$\HHH$-submodules of~$\cind_\KKK^\GGG\rho$ 
which do not contain~$\cind_{\KKK\HHH'}\rho$, whence trivial on all 
summands~$\cind_{\KKK g\HHH}\rho$ with~$g\not\in\KKK\HHH$. In particular, we 
see that~$\Phi\in\Hom_{\HHH}(\cind_{\KKK\HHH} \rho,\tau)$ so 
that~$\phi\in\Hom_{\KKK\cap\HHH}(\rho,\tau)$, as required. 
\end{proof}

\begin{proof}[Proof of Proposition~\ref{Ok for types}]
We apply the lemma to our situation, where we recall that~$\G=\GL_n(\F)$, 
$\P=\P_n(\F)$ is a~$\s$-stable mirabolic subgroup, and~$(\bJ,{\bl})$ is a 
$\s$-self-dual type -- by which, we recall, we mean a
$\s$-self-dual extended 
maximal simple type -- with~$\pi=\cind_{\mathbf{J}}^\G{\bl}$ an
irreducible distingui\-shed $\s$-self-dual cuspidal representation of~$\G$. 
{The result of Ok~\cite[Theorem~3.1.2]{Ok} 
(see also~\cite[Proposition~2.1]{MatringeManuscripta}), 
proved for any \textit{irreducible complex} representation of $\G$
and which we generalize to any \textit{cuspidal} representation of $\G$ 
\textit{with coefficients in $\FC$ }in Appendix \ref{AppB} 
(see Proposition \ref{Ok in general}), says that, in this 
situation, we have an equality 
\begin{equation*}
\Hom_{\P^\s}(\pi,\one)=\Hom_{\G^\s}(\pi,\one).
\end{equation*}
We set~$\GGG=\G$,~$\HHH=\G^\s$ and~$\HHH'=\P^\s$,~with $\tau=\one$ the trivial 
representation of~$\HHH$, and~$(\KKK,\rho)=(\bJ,{\bl})$. 
Then the result follows at once from Lemma~\ref{lem:bijection}.}
\end{proof}

We turn now to Whittaker functions. 
Let~$\N=\N_n(\F)$ denote the standard maximal
unipotent subgroup (consisting of upper triangular unipotent matrices) and
let~$\psi$ be  a~$\s$-self-dual non-degenerate  character of~$\N$.  If~$\pi$ is
any generic irreducible representation of~$\G$, recall also that 
its~\textit{Whittaker~mo\-del (with respect to~$\psi$)} is the
subspace~$\Ww(\pi,\psi)$ of~$\Ind_\N^\G\psi$ which is the image of~$\pi$ under
any non-zero map in the one-dimensional space~$\Hom_\G(\pi,\Ind_\N^\G\psi)$. 

Now let~$\pi$ be an irreducible~$\s$-self-dual cuspidal representation
of~$\G$. By Theorem~\ref{PIMAIN} and Proposition~\ref{brunau}, 
it contains a~$\s$-self-dual type~$(\bJ,\bl)$ such that
$\Hom_{\bJ\cap\N}(\bl,\psi)\ne 0$. 
We use the usual~no\-ta\-tion for data associated to this type; in particular, we 
have the unique maximal simple character~$\t$ contained in~$\bl$ and 
normalized by $\bJ$, defined on 
the normal subgroup~$\H^1$ of~$\bJ$, as well as the normal 
subgroups~$\J\supseteq\J^1$ of~$\bJ$. 

Let $\Uu=(\N\cap \J)\H^1$ and extend $\psi$ to a character~$\psi_\bl$ of $\Uu$ 
as in~\cite[Definition~4.2]{StPa}: 
\[
\psi_\bl(uh) = \psi(u)\t(h),\qquad\text{for }u\in \N\cap\J, h\in\H^1.
\]

We fix a normal compact open subgroup $\Nn$ of $\Uu$ contained in 
$\ker(\psi_\bl)$ and define the \emph{Bessel function} $\Jj_{\bl}:\bJ\to \FC$ 
of $\bl$ by 
\[
\Jj_{\bl}(j)=
{\frac{1}{(\Uu:\Nn)}}
\sum_{u\in\Uu/\Nn}\psi_\bl(u)^{-1}\tr\bl(ju), 
\qquad\text{for }j\in\bJ,
\]
where~$\tr\bl$ is the trace character of $\bl$. 
This is independent of the choice of $\Nn$.
{Note that this~de\-finition makes sense over $\FC$,
since $\Uu$ is a pro-$p$-group.}

We then define a function~$\W_{\bl}:\G\to \FC$ supported in $\N\bJ$ by
\begin{equation}
\label{numrapideW}
\W_{\bl}(nj)=\psi(n)\Jj_{\bl}(j), \qquad\text{ for }n\in\N, j\in\bJ.
\end{equation}
{One checks that the function $\W_{\bl}$ is well defined,
and that $\W_{\bl}(ng)=\psi(n)\W_{\bl}(g)$ for all $n\in\N$ and $g\in\G$.}

We set~$\Mm=(\P\cap\J)\J^1$ and note that, by~\cite[Corollary~4.8]{StPa}, 
the subgroup $\P\cap\bJ = \P\cap\J$~is~contained in $\Mm$.
Let $\Ss_\bl$ denote the 
space of functions~$f:\Mm\to\FC$ 
such that~$f(um)=\psi_\bl(u)f(m)$~for~all~$u\in\Uu$ and~$m\in\Mm$. 
For each $j\in\bJ$, we define an operator $\L(j)$ on $\Ss_\bl$ by
\begin{equation}
\label{numrapide}
\L(j)f : m \mapsto \sum\limits_{x\in\Mm/\Uu} \Jj_{\bl}(mjx)f(x^{-1})
\end{equation}
for all $f\in\Ss_\bl$ and $m\in\Mm$.
This defines a representation $\L$ of $\bJ$ on $\Ss_\bl$.
We claim that this~repre\-sen\-ta\-tion is isomorphic to $\bl$.
When $\FC$ is the field of complex numbers, 
or more generally when $\FC$ has characteristic $0$,
this is \cite[Theo\-rem~5.4]{StPa}.
Let us explain briefly how to deduce the modular case from the 
characteristic $0$ case. 
Assume that $\FC$ has characteristic $\ell>0$.
First, by the same~ar\-gu\-ment as in Remark \ref{guderian},
it is enough to prove the result when $\FC$ is the field 
$\overline{\FF}_\ell$. 
Then fix~an~ex\-ten\-ded maximal simple type $\widetilde{\bl}$ with 
coefficients in $\overline{\QQ}_\ell$ whose reduction mod $\ell$ 
is isomorphic to $\bl$
(which is possible by \cite[Proposition 2.39]{MSt}).
We thus have an isomorphism between $\widetilde{\bl}$ and the 
representation on $\Ss_{\widetilde{\bl}}$ defined as 
in \eqref{numrapide}.
Reducing mod $\ell$, we get the claimed result. 
In the sequel, we will identify the space of~$\bl$ with~$\Ss_\bl$.
It follows as in \cite[Section~5.2]{StPa} that the function $\W_{\bl}$ 
defined by \eqref{numrapideW} 
belongs to the Whittaker model~$\Ww(\pi,\psi)$ of $\pi$.
Note also (see~\cite[Proposition~5.3(iii)]{StPa}) 
that the restriction of $\Jj_\bl$ to $\Mm$ lies in~$\Ss_\bl$.

\begin{proposition}
\label{invariant integral} 
{Let~$\pi$ be a~$\s$-self-dual cuspidal representation
of~$\G$,
and let $(\bJ,\bl)$ be a generic $\s$-self-dual type contained in $\pi$.} 
\begin{enumerate}
\item 
Let $dm$ be a right invariant measure on $(\J\cap\N^\s)\backslash (\J\cap\P^\s)$. 
The linear form on~$\bl$ defined by 
\[
\LL_{\bl}(f)=\int_{(\J\cap\N^\s)\backslash (\J\cap\P^\s)} 
f(m)dm
\] 
for any $f\in\Ss_\bl$ is~$\J\cap\P^\s$-invariant, 
and~$\LL_{\bl}(\Jj_{\bl})$ is non-zero. 
\item
Moreover, if $\pi$ is distinguished, 
then~$\LL_{\bl}$ is~$\bJ^\s$-invariant.
\end{enumerate}
\end{proposition}

\begin{proof}
The form~$\LL_\bl$ is clearly~$\J\cap\P^\s$-invariant by its definition. 
By~\cite[Proposition 5.3(iv)]{StPa},~{the 
proof of which is written for complex representations but still works
in the modular case,} 
the~func\-tion $\Jj_{\bl}$ is identically
zero on {the complement of $\Uu^\s$ in $\Mm^\s$}. 
On the other hand, for~$u\in\Uu^\s$, 
we have~$\Jj_\bl(u)=\psi_\bl(u)=1$, since~$\psi_\bl$ is a~$\s$-self-dual
character of a pro-$p$ group~$\Uu$ and~$p$ is odd.~Hence the value 
{$\LL_{\bl}(\Jj_{\bl})=dm((\J\cap\N^\s)\backslash(\J\cap\N^\s)(\H^1\cap\P^\s))$
is non-zero, since $\H^1$ is pro-$p$.}
The final~sta\-te\-ment follows immediately from 
the fact that 
$\bJ\cap\P=\J\cap\P$ together with Proposition~\ref{Ok for types}. 
\end{proof}

We deduce the following corollary from Proposition \ref{invariant integral}.

\begin{corollary}
\label{aloi}
Let~$\pi$ be a~$\s$-self-dual cuspidal 
representation of~$\G$.
Then $\pi$ is distingui\-shed if and only if any of its generic $\s$-self-dual 
types is distinguished.
\end{corollary}

\begin{remark}
\label{coincoin}
Putting Corollary \ref{aloi} and Proposition \ref{brunau} together,
we obtain a different proof of~a result of \cite{VS} saying that 
a $\s$-self-dual cuspidal representation~$\pi$ of $\G$
is distinguished if and only if it contains a distinguished $\s$-self-dual 
type, 
and that, if the~quadratic extension $\T/\T_\so$ associated 
with $\pi$ by Proposition \ref{TT0canonique} is ramified, 
any distinguished $\s$-self-dual type contained in $\pi$ has
index $\lfloor m/2\rfloor$, where $n=md$ and $d$ is the degree of the 
endo-class of $\pi$.
\end{remark}

\subsection{Explicit Whittaker functions and restriction to $\GL_n(\F_\so)$}

We continue with the same notation, 
{and write $\K=\GL_n(\o)$ and $\K^\s=\GL_n(\o_\so)$}.
In order to make computations, we need to 
be somewhat more careful with our choice of non-degenerate character~$\psi$ to 
ensure that corresponding generic $\s$-self-dual type is well-positioned with 
respect to the standard maximal compact subgroup~$\K^\s$ of~$\G^\s$. 

Let $\mathfrak{S}_n$ denote the group of permutation matrices in $\G^\s$. 
The Bruhat decomposition in the finite quo\-tient of $\K^\s$ by its
pro-$p$ unipotent radical, together with the Iwasawa decomposition of $\G^\s$,
yields the \emph{Bruwasawa decomposition} 
$\G^\s=\B^\s\mathfrak{S}_n\I_\so$, where~$\B$ is
the standard Borel subgroup of~$\G$ and~$\I_\so$ is the standard Iwahori
subgroup of~$\G^\s$. In particular, this decomposition implies that any
parahoric subgroup of $\G^\s$ is conjugate by $\N^\s$ to a parahoric subgroup
in the standard apartment, where~$\N$ is the unipotent radical of~$\B$. 

If~$(\bJ,\bl)$ is a~$\s$-self-dual type in $\G$ 
then we can write~$\bJ=\bJ(\aa,\b)$, 
with~$\aa$ a~$\s$-stable hereditary order and~$\s(\b)=-\b$. As in
Paragraph~\ref{sec:inv}, we have~$\aa^\times\cap\G^\s=\aa_\so^\times$, 
for some $\o_\so$-hereditary order~$\aa_\so$~in $\Mat_n(\F_\so)$.
Write~$e_\so$ for the~$\o_\so$-period of~$\aa_\so$, 
and~$\La_\so$
for the~$\o_\so$-lattice chain 
in the vector space~$\F_\so^n$ consisting of $\aa_\so$-lat\-ti\-ces. 
These depend only on the pair~$(\bJ,\bl)$. 

Writing~$\be_1,\ldots,\be_n$ for the standard basis of~$\F^n$, and using the 
notation above, we get the following. 

\begin{lemma}\label{basis}
Let $\pi$ be a~$\s$-self-dual cuspidal representation of $\G$. 
There are a~$\s$-self-dual Whittaker datum~$(\N,\psi)$ and a generic~$\s$-self-dual 
type $(\bJ,\bl)$ in $\pi$ such that 
\begin{enumerate}
\item\label{part1basis} 
the space $\Hom_{\bJ\cap\N}(\bl,\psi)$ is non-zero;
\item\label{part2basis} there is a numbering on the~$\o_\so$-lattice 
chain~$\La_\so$ associated to~$(\bJ,\bl)$ such that
\[
\La_\so(k)=\bigoplus_{i=1}^n \p_\so^{a_i(k)} \be_i, \qquad\text{ for }k\in\ZZ, 
\]
where the $a_i:\ZZ\to\ZZ$ are increasing functions satisfying 
\begin{enumerate}
\item 
$a_i(k+e_\so)=a_i(k)+1$ for all $k\in\ZZ$ and $a_i(0)=0$, 
for $i=1,\ldots,n$, 
\item
$a_n(0)=\cdots=a_n(e_\so-1)=0$.
\end{enumerate}
\end{enumerate}
\end{lemma}

Note that condition~\ref{part2basis} implies in particular 
that~$\J^\s\subseteq\K^\s$ (though it is not equivalent to this). It is also 
worth noting that it is \emph{not} in general possible to find~$(\bJ,\bl)$ 
satisfying condition~\ref{part1basis} and the stronger 
condition~$\J\subseteq\K$ (see Remark \ref{zaza}).

\begin{proof}
We pick a~$\s$-self-dual Whittaker datum~$(\N,\psi)$ 
where~$\N=\N_n(\F)$ is the standard maximal uni\-potent subgroup. 
By
Proposition~\ref{brunau}, we have a 
$\s$-self-dual type~$(\bJ,\bl)$ satisfy\-ing~\ref{part1basis}. 
Fix~a~maxi\-mal simple stratum $[\aa,\b]$ as above, denote by~$\aa_\so$
the~$\o_\so$-hereditary order~associa\-ted to it and by~$e_\so$ its 
period.
There is an element
$u\in \N^\s$ which sends $\aa_\so$ to a point in the standard 
apartment.~Con\-jugating 
by $u$, we assume~$\aa_\so$ is itself in the standard apartment. 

Writing~$\La_\so$ for the~$\o_\so$-lattice chain in~$\F_\so^n$ consisting 
of~$\aa_\so$-lattices, we can number the lattices such that 
\[
\Lambda_\so(0)\cap \F_\so \be_n = \o_{\so}\be_n,\qquad \Lambda_\so(-1)\cap \F_\so \be_n= \p_{\so}^{-1} \be_n.
\] 
Since $\aa_\so$ lies in the standard apartment, we can find $t_1,\dots,t_{n-1}\in 
\F_\so^\times$ such that 
\[
\Lambda_\so(0) = \o_\so t_1 \be_1 \oplus\cdots\oplus \o_\so t_{n-1} \be_{n-1} \oplus \o_\so \be_n.
\]
Conjugating both~$(\bJ,\bl)$ and the Whittaker datum~$(\N,\psi)$ 
by~$t=\diag(t_1,..,t_{n-1},1)$ (which is in the diagonal torus of~$\G^\s$), we 
obtain the result. 
\end{proof}

\begin{remark}
\label{zaza}
Suppose that $\F/\F_\so$ is ramified, $n=2$ and $\pi$ is a $\s$-self-dual depth 
zero cuspidal~re\-pre\-sen\-tation of $\GL_2(\F)$.
Then any generic $\s$-self-dual type $(\bJ,\bl)$ in $\pi$ has index $1$ so 
$\bJ$ is $\GL_2(\F_\so)$-conjugate to $t_1^{}\K t_1^{-1}$ where $t_1=\diag(\w,1)$ 
and $\w$ is a uniformizer of $\F$.
In particular, the group 
$\bJ$ is not $\GL_2(\F_\so)$-conjugate to (any subgroup of) $\K$.
\end{remark}

\subsection{}

Suppose now that $\pi$ is a~$\s$-self-dual {supercuspidal}
representation {(see Paragraph \ref{sec:inv})}
and choose our non-degenerate character~$\psi$ and 
generic~$\s$-self-dual type $(\bJ,\bl)$ as in Lemma~\ref{basis}. 
We have an order $\aa_\so$ as above. 
By Lemma~\ref{lem:e00}, it is a principal order. 
We choose~$\w_\bl\in\bJ^\s$ as
in Re\-mark~\ref{rmk:pilambda},~so 
that~$\bJ^\s$ is generated by $ \varpi_\bl$ and $\J^\s$. 

The following lemma shows a useful property of the Iwasawa decomposition of 
$\w_{\bl}$ in $\G^\s$, which will be key to our computation to come. 

\begin{lemma}\label{unramramiwas}
Let $i\in\ZZ$.
We have~$\w_{\bl}^i\in \P^\s \K^\s$ if and only if $i\in\{0,\ldots,e_\so-1\}$.
In that case,~if~we choose $p_i\in \P^\s$ and $k_i\in \K^\s$ such that 
$\w_{\bl}^i=p_ik_i$, then 
$\left|\det (p_i)\right|_{\so}=\left|\det (\w_{\bl}^i)\right|_\so=q_\so^{-in/e_\so}$. 
\end{lemma}

\begin{proof}
Note first that~$\P^\s \K^\s$ consists  precisely of those matrices whose last
row  lies in  $(\o_\so,\ldots,\o_\so)$  but  not in  $(\p_\so,\ldots,\p_\so)$.
Considering the action of $\w_{\bl}$ on the lattice chain $\Lambda_\so$, 
it follows at~once~from the previous lemma that the last row of $\w_{\bl}^i$
belongs to 
$(\p_\so^{\lfloor i/e_\so\rfloor},\ldots,\p_\so^{\lfloor i/e_\so\rfloor})$ -- 
that is, the entries of this row all have valuation 
$\>\lfloor i/e_\so\rfloor$ -- 
but does not belong to 
$(\p_\so^{\lceil i/e_\so\rceil},\ldots,\p_\so^{\lceil i/e_\so\rceil})$, 
which implies the first statement. 
The second is
immediate, since~$\left|\det (k)\right|_\so=1$, for all~$k\in\K^\s$. 
\end{proof}

For $i\in\{0,\ldots,e_\so-1\}$, we fix from now $p_i\in \P^\s$ and~$k_i\in 
\K^\s$ such that~$\w_{\bl}^i=p_ik_i$, as in Lemma~\ref{unramramiwas}. 

We recall from the previous section that we have an explicit Whittaker
function $\W_{\bl}\in\Ww(\pi,\psi)$~with support~$\N\bJ$. 

\begin{proposition}\label{Supportwhittakers}
{For each~$l\in\ZZ$, let~$\W^l_\bl$ denote the function from $\G^\s$ to $\FC$ 
supported on the~sub\-set
$\{g\in \G^\s\cap\N\bJ\mid|\det(g)|_\so=q_\so^{-l}\}$
and coinciding with $\W_\bl$ on it.}
\begin{enumerate}
\item\label{point1} 
The function~$\W^l_{\bl}|_{\P^\s\K^\s}$ is zero unless~$l=in/e_\so$, 
with~$i\in\{0,\dots,e_\so-1\}$, in which case 
\[
\supp\(\W^l_{\bl}|_{\P^\s\K^\s}\) \subseteq \N^\s\w_{\bl}^i \J^\s.
\]
\item\label{point2} 
If~$\W_{\bl}^{in/e_\so}(pk)\neq 0$, with $p\in \P^\s$, $k\in \K^\s$ 
and~$i\in\{0,\dots,e_\so-1\}$, then~$k\in (\P^\s\cap \K^\s)k_i\J^\s$. 
\item\label{point3} 
If~$\W_{\bl}^{in/e_\so}(p\w_{\bl}^i j)\neq 0$ with~$p\in\P^\s$,~$j\in\J^\s$ 
and $i\in\{0,\dots,e_\so-1\}$, then~$p\in \N^\s(\P^\s\cap \J^\s)$.
\end{enumerate}
\end{proposition}

\begin{proof}
Note that $\W_\bl|_{\G^\s}$ is supported in $\G^\s\cap\N\bJ$, 
which is equal to~$\N^\s\bJ^\s$ by Lemma~\ref{lem:coset}\ref{lem:coset.iii}. 
By definition of~$\w_{\bl}$, the set~$\N^\s\bJ^\s$ is the disjoint union of
the $\N^\s\w_{\bl}^i\J^\s$ for $i\in\ZZ$, 
and then~\ref{point1} follows from Lemma~\ref{unramramiwas}. 
The  remaining  parts follow  exactly  as  in the  proof  of~\cite[Proposition
8.4]{RKNM}. 
\end{proof}

{Finally,
using that $\J^\s\subseteq\K^\s$ thanks to our choice of basis, 
as in \cite[Lemma~7.2]{RKNM}, we have the~fol\-lowing lemma, which 
we will use in Section~\ref{S7VS}.}


\begin{lemma}
\label{LemmaVolume}
There is a unique right invariant {complex valued}
measure $dk$ on $(\P^\s\cap \K^\s)\backslash\K^\s$ such 
that we have
\[
dk((\P^\s\cap \K^\s)\backslash (\P^\s\cap \K^\s)k_i\J^\s)=  q_\so^{-in/e_\so}.
\]
for all $i\in \{0,\ldots,e_\so-1\}$.
\end{lemma}

\begin{proof}
Let $dk$ be any right invariant measure on $(\P^\s\cap \K^\s)\backslash\K^\s$. 
Following the first part of the proof of \cite[Lemma~7.2]{RKNM},
and thanks to Lemma \ref{unramramiwas}, we have:
\begin{equation*}
dk((\P^\s\cap \K^\s)\backslash (\P^\s\cap \K^\s)k_i\J^\s) =
q_\so^{-in/e_\so} \cdot dk((\P^\s\cap \K^\s)\backslash (\P^\s\cap \K^\s)\J^\s)
\end{equation*}
for all $i\in \{0,\ldots,e_\so-1\}$.
{Thus the required measure is that for which
$\K^\s\cap\P^\s)\backslash (\K^\s\cap\P^\s)\J^\s$
has volume $1$.}
\end{proof}

\section{Asai~$\mathbf{L}$-functions and test vectors}
\label{S7VS}

\textit{From now on, until the end of the paper,
all representations are complex,}
that is, $\FC$ is now the field $\CC$ of complex numbers.

\subsection{Distinction and dichotomy}
\label{Dichotomy49}

We will need two further key results on distinction of~$\s$-self-dual 
cuspidal complex representations, which we recall from \cite{VS}. 
Recall that~$\ef$ denotes the non-trivial character of~$\F_\so^\times$ 
which is trivial on $\Nm_{\F/\F_\so}(\F^\times)$. 
The first result is \emph{dichotomy}.
It is proved for discrete series representations when~$\F$ has 
characteristic~$0$ by Flicker~\cite{Flicker}, Kable~\cite{Kable}
and Anandavardhanan, Kable and Tandon~\cite{MR2063106},
and we prove in Appendix \ref{section positive char}
(see Theorem \ref{global proof of dichotomy} below) 
that the global arguments of \cite{Kable} and \cite{MR2063106} remain 
valid when $\F$ has characteristic $p$.
It is also proved by S\'echerre \cite{VS} for {cuspidal} 
representations,
in a purely local way, 
with no assumption on the characteristic of $\F$
(see also Remark \ref{alabama}).

\begin{theorem}[\cite{VS} Theorem {10.8}]
\label{dichotomy}
Let~$\pi$ be a cuspidal (complex) representation of~$\GL_n(\F)$, $n\>1$. 
\begin{enumerate}
\item 
$\pi$ is~$\s$-self-dual 
if and only if it is either distinguished or $\ef$-distinguished. 
\item
$\pi$ cannot be both distinguished and $\ef$-distinguished. 
\end{enumerate} 
\end{theorem}

Given~a $\s$-self-dual cuspidal representation $\pi$ of~$\GL_n(\F)$, 
we denote by $\T/\T_\so$ the~quadratic ex\-ten\-sion associated to $\pi$ by 
Proposition \ref{TT0canonique}.
Let $d$ denote the degree of the endo-class of $\pi$.
It is a divisor of $n$, and we write $n=md$.

\begin{proposition}[\cite{VS} Proposition {10.12}]
\label{prop:lemma16}
Let $\pi$ be a distinguished cuspidal (complex) 
representa\-tion of~$\GL_n(\F)$. 
Then $\pi$ has an $\ef$-distinguished unramified twist if and only if 
either~$\T/\T_{\so}$~is un\-ra\-mi\-fied or $m>1$.
\end{proposition}

\begin{remark}
\label{alabama}
These two results are proved in \cite{VS} in a more general setting: 
$\pi$ is a \textit{supercuspidal} representation of $\GL_n(\F)$ 
with coefficients in $\FC$,
where $\FC$ has characteristic different from $p$.
Note that, when $\FC$ has characteristic $0$,
any cuspidal representation is supercuspidal.
\end{remark}

\subsection{Definition of the integrals}

As before, we suppose that $\psi$ is a $\s$-self-dual non-degenerate character 
of $\N$. 
Let $\pi$ be a generic~irre\-du\-cible representation of $\G$. 
For $\W$ a function 
in the Whittaker model $\Ww(\pi,\psi)$ of $\pi$ and $\Phi$ in~the 
space $\Cc_c^\infty(\F_\so^n)$ of locally constant functions on~$\F_\so^n$ 
with compact support, define the local Asai~in\-te\-gral 
\begin{equation}
\label{Iasai}
\I_\As(s,\Phi,\W)= \int_{\N^\s\backslash \G^\s} \W(g)
\Phi(\etavec g)|\det(g)|_\so^s \ dg,
\end{equation}
where $\etavec$ is the row vector
$\begin{pmatrix}0&\dots&0&1\end{pmatrix}$
and $dg$ is a right invariant 
measure on $\N^\s\backslash \G^\s$ which will be fixed later 
in Paragraph \ref{tvp73}. 
It turns out (see~\cite[Theorem 2]{Kable})
that, for $s\in\CC$ with suf\-fi\-cient\-ly large real part, 
the integral \eqref{Iasai} is a rational function in
$q_\so^{-s}$; moreover, as $\W$ varies in $\Ww(\pi,\psi)$ and $\Phi$ varies in
$\Cc_c^\infty(\F_\so^n)$, these functions generate a fractional ideal of
$\CC[q_\so^{s},q_\so^{-s}]$ which has a  unique generator $\L_\As(s,\pi)$ which is an
Euler factor (i.e.~of the form $1/\P(q_\so^{-s})$ where $\P$ is a polynomial
with constant term $1$). 

Now let $\pi$ be a cuspidal representation of $\G$ and let 
$\X(\pi)$ denote the set of unramified characters~$\chi$ of $\G^\s$ such that 
$\pi$ is $\chi$-distinguished. We recall the following description of the Asai 
$\L$-function~of a cuspidal representation,
the proof of which is valid 
(as the rest of \cite{MatringeManuscripta}) 
when $\F$ has positive~cha\-rac\-teristic as well: 

\begin{proposition}[{\cite[Proposition 3.6]{MatringeManuscripta}}]\label{MatMan}
Let $\pi$ be a cuspidal representation of $\G$. Then 
\[
\L_\As(s,\pi)=\prod_{\chi\in\X(\pi)} \(1-\chi(\w_\so)q_\so^{-s}\)^{-1}
\]
where $\w_\so$ is a fixed uniformizer of $\F_\so$.
\end{proposition}

Let $t(\pi)$ denote the \textit{torsion number} of $\pi$,
that is the number of unramified characters of $\mult\F$~such 
that $\pi(\chi\circ\det)$ is isomorphic to $\pi$.
Thanks to Theorem~\ref{dichotomy}, we deduce the following formula. 

\begin{corollary}
\label{cor:equation}
Let $\pi$ be a distinguished cuspidal representation of $\G$. 
Then
\[
\L_\As(s,\pi)=\begin{cases}
\frac{1}{1-q_\so^{-st(\pi)}}&\text{if $\F/\F_\so$ is unramified;}\\[10pt]
\frac{1}{1-q_\so^{-st(\pi)}}&\text{if $\F/\F_\so$ is ramified and no unramified twist of $\pi$ is $\ef$-distinguished;}\\[10pt]
\frac{1}{1-q_\so^{-st(\pi)/2}}&\text{if $\F/\F_\so$ is ramified and an unramified twist of $\pi$ is $\ef$-distinguished.}
\end{cases}
\]
\end{corollary}

As the Rankin--Selberg and Langlands--Shahidi Asai local $\L$-functions 
agree (see Theorem \ref{manciniA1}),
one can deduce 
Corollary~\ref{cor:equation} from~\cite[Theorem 1.1]{AnandMondal}.
We give another proof, based on Proposition~\ref{MatMan} and 
Theorem~\ref{dichotomy}.

\begin{proof}
Let $\R(\pi)$ denote the group of unramified characters of $\mult\F$ such~that 
$\pi(\chi\circ\det)$ is isomorphic to $\pi$.
It is cyclic and has order $t(\pi)$.
Let us fix uniformizers $\w$ and $\w_\so$ of $\F$ and $\F_\so$, 
respectively.
Since $\pi$ is distinguished, it~is $\s$-self-dual.
Let $\U(\pi)$ denote the subgroup of~un\-rami\-fied characters $\chi$ of 
$\mult\F_\so$ such that 
$\pi (\widetilde{\chi}\circ\det)$~is $\s$-self-dual
for any unramified character $\widetilde{\chi}$~of~$\F^\times$ ex\-tending $\chi$.
An un\-ra\-mi\-fied 
character $\chi$ belongs to $\U(\pi)$ if and only if 
\[
\chi\(\Nm_{\F/\F_\so}(\w)\)^{t(\pi)}=1.
\]
Note that we have $\ef\in\U(\pi)$.

Let $\Y(\pi)$ denote the set of unramified characters $\chi$ of 
$\mult\F_\so$ such that $\pi$ is $\ef\chi$-distinguished.
Then Theorem~\ref{dichotomy} says that $\U(\pi)$ decomposes as the 
disjoint union of $\X(\pi)$ and $\Y(\pi)$.

We first treat the case where $\F/\F_\so$ unramified. 
If $\chi$ is an 
unramified character of $\mult\F_\so$, then $\chi\in\U(\pi)$ if and only if 
$\chi(\w_\so)^{2t(\pi)}=1$, hence $\U(\pi)$ is cyclic of order $2t(\pi)$. 
But we have $\ef\in\U(\pi)$, hence $\Y(\pi)=\ef\X(\pi)$, and $\X(\pi)$ is of
order $t(\pi)$, this proves the expected equality in the first case. 

We now suppose that $\F/\F_\so$ is ramified, hence for an unramified character
$\chi$  of   $\mult\F_\so$,  one  has   $\chi\in  \U(\pi)$  if  and   only  if
$\chi(\w_\so)^{t(\pi)}=1$ so $\U(\pi)$ is cyclic of order $t(\pi)$. If no unramified
twist of $\pi$ is $\ef$-dis\-tinguished, then $\Y(\pi)$ is empty, hence
$\X(\pi)=\U(\pi)$ and $\X(\pi)$ is of order $t(\pi)$, whereas if an~unra\-mi\-fied 
twist $\pi\mu$ of $\pi$ is $\ef$-distinguished, then setting
$\chi=\mu|_{\mult\F_\so}$, one has $\Y(\pi)=\chi  \X(\pi)$,~thus $\X(\pi)$ is
of order $t(\pi)/2$. The last two equalities follow immediately. 
\end{proof}

\begin{remark}
By \cite[6.2.5]{BK}, the torsion number $t(\pi)$ is equal to $n/e$, 
where $e$ is a divisor of $n$ equal to the ramification index of the 
endo-class of $\pi$ (see Paragraph \ref{simplechar}),
that is $e(\E/\F)$ with the notation of Paragraph \ref{sec:inv}.
Using the invariant $e_\so$ introduced in Paragraph \ref{sec:inv}
and computed in Lemma \ref{lem:e00},
together with~Pro\-po\-sition \ref{prop:lemma16}, 
we deduce that 
Corollary \ref{cor:equation} is equivalent to the equality 
\begin{equation}
\label{panait}
\L_\As(s,\pi)=\frac{1}{1-q_\so^{-s\frac{n}{e_\so}}}.
\end{equation}
\end{remark}

\subsection{A decomposition of the integral}

We continue with $\pi$ a cuspidal (complex)
representation of $\G$. 
For computational convenience, we introduce~a second integral: for $\W$ in the 
Whittaker model $\Ww(\pi,\psi)$ of $\pi$, we put 
\begin{equation}
\label{izero}
\I^{(0)}_\As(s,\W)=\int_{\N^\s\backslash\P^\s} \W(p)|\det(p)|_\so^{s-1}\ dp
\end{equation}
where $dp$ is a right invariant
measure on $\N^\s\backslash\P^\s$ which will be fixed later 
in Proposition \ref{coef}.
Again, if $s\in\CC$ has suf\-ficiently large real part, 
$\I^{(0)}_\As(s,\W)$ is a rational function in $q_\so^{-s}$. 

Now let $dk$ be the measure on $(\P^\s\cap \K^\s)\backslash\K^\s$ given by 
Lemma \ref{LemmaVolume} and $d^\times a$ be the Haar measure 
on~$\F^\times_\so$ giving measure $1$ to $\o_\so^\times$.
Then, as noticed in \cite[Section~4]{Flicker}, 
if $s$ has a sufficiently large real part and if the function
$\Phi\in\Cc_c^\infty(\F_\so^n)$ is chosen to be $\K^\s$-invariant, 
there is a unique right invariant measure $dg$ on $\N^\s\backslash\G^\s$, 
depending only on the choice of $dp$,
such that: 
\begin{equation}
\label{noumea}
\I_\As(s,\Phi,\W)= \int_{\mult\F_\so}\Phi(\etavec a)\om_{\pi}(a)|a|_\so^{ns}\ d^\times a
\int_{(\K^\s\cap\P^\s)\backslash\K^\s}\I^{(0)}_\As(s,k\cdot\W) \ dk
\end{equation}
where $\om_{\pi}$ denotes the central character of $\pi$
and $g\cdot\W$ denotes the action of $g\in\G$ on $\Ww(\pi,\psi)$, 
that is $(g\cdot\W)(x)=\W(xg)$ for $x\in\G$.
From now on, we will assume that $dg$ is chosen with respect to $dp$ 
so that \eqref{noumea} holds. 

Suppose that $\omega_\pi$ is trivial when restricted to $\mult\F_\so$, 
which is the case when $\pi$ is distinguished. 
If $\Phi$ is the 
characteristic function~$\one_{\o_\so^n}$ of $\o_\so^n$, then we have 
\[
\int_{\mult\F_\so}\one_{\o_\so^n}(\etavec a)\om_{\pi}(a)|a|_\so^{ns}\ d^\times a
=\int_{\o_\so\setminus\{0\}}|a|_\so^{ns}\ d^\times a
=\frac{1}{1- q_\so^{-ns}},
\]
by Tate's thesis \cite{CasselsFrohlich}.
Therefore, we have the following decomposition which we record as a lemma:

\begin{lemma}\label{Lemma1}
Let $\pi$ be a distinguished cuspidal complex 
representation of $\G$. 
Then, for~all~func\-tions $\W\in \Ww(\pi,\psi)$, we have 
\[
\I_\As\left(s,\one_{\o_\so^n},\W\right)= \frac{1}{1- q_\so^{-ns}}
\int_{(\K^\s\cap\P^\s)\backslash\K^\s}\I^{(0)}_\As\left(s,k\cdot\W\right)\ dk.
\]
\end{lemma}

For $\W\in\Ww(\pi,\psi)$ and $l\in\ZZ$, we write $\W_\so^l$ for the 
function from $\G^\s$ to $\CC$ supported on the subset 
$\{g\in \G^\s\mid|\det(g)|_\so=q_\so^{-l}\}$
and coinciding with $\W$ on it.
Finally we decompose the integral given~in Lemma~\ref{Lemma1} by the absolute 
value: 
\begin{align*}
\int_{(\K^\s\cap\P^\s)\backslash\K^\s}\I_\As^{(0)}(s,k\cdot\W)\ dk 
&=\sum_{l\in\ZZ}\int_{(\K^\s\cap\P^\s)\backslash\K^\s}\int_{\N^\s\backslash\P^\s}
  \W_\so^l(pk)|\det(pk)|_\so^{s-1}\ dp\ dk\\ 
&=\sum_{l\in\ZZ}q_\so^{-l(s-1)}\int_{(\K^\s\cap\P^\s)\backslash\K^\s}\int_{\N^\s\backslash\P^\s}
  \W_\so^l(pk)\ dp\ dk.
\end{align*}
{Since $\pi$ is cuspidal,
the right hand term of the equality above is a finite sum~\cite{BZ76}.}
We call 
\[
\CCC_l(\W)=\int_{(\K^\s\cap\P^\s)\backslash\K^\s}\int_{\N^\s\backslash\P^\s}
\W_\so^l(pk)\ dp\ dk
\]
the $l^{\text{th}}$ \emph{coefficient} of the integral, and we record that:

\begin{lemma}\label{Lemma2}
Let $\pi$ be a distinguished cuspidal complex 
representation of $\G$. 
Then, for~all~func\-tions $\W\in \Ww(\pi,\psi)$, we have 
\[
\I_\As(s,\one_{\o_\so^n},\W) = \frac{1}{1- q_\so^{-ns}} 
\( \sum_{l\in\ZZ} \CCC_l(\W) q_\so^{-l(s-1)} \).
\]
\end{lemma}

\subsection{Test vectors}
\label{tvp73}

Until the end of this section, 
$\pi$ is a distinguished cuspidal representation of $\G$
and $(\bJ,\bl)$ is a generic $\sigma$-self-dual type 
as in Lemma \ref{basis}. 
Now we compute the Asai integral of the 
explicit Whittaker vector $\W_{\bl}$ showing it is a test vector, making use of 
the decomposition of Lemma~\ref{Lemma2}. 

\begin{proposition}\label{coef}
Let~$l\in \ZZ$.
\begin{enumerate}
\item
The~$l^{\text{th}}$ coefficient~$\CCC_l(\W_\bl)$ is zero unless we 
have~$l=in/e_\so$ for some $i\in \{0,\dots,e_\so-1\}$. 
\item
There is a unique right invariant measure $dp$ on $\N^\s\backslash\P^\s$ 
such that $\CCC_{in/e_\so}(\W_\bl)=q_\so^{-in/e_\so}$
for $i\in \{0,\dots,e_\so-1\}$.
\end{enumerate}
\end{proposition}

\begin{proof}
By definition, we have
\begin{equation*}
\CCC_l(\W_\bl)=\int_{(\K^\s\cap\P^\s)\backslash\K^\s}\int_{\N^\s\backslash\P^\s}
\W_\bl^l(pk)\ dp\ dk
\end{equation*}
where $dk$ is the measure given by Lemma \ref{LemmaVolume} 
and $dp$ is a right invariant measure on $\N^\s\backslash\P^\s$.

Part (i) of the proposition follows from Proposition 
\ref{Supportwhittakers}(i). 

Now assume that $l=in/e_\so$ for some $i\in \{0,\dots,e_\so-1\}$. 
We recall that we have our fixed decompo\-si\-tions~$\w_\bl^i=p_i k_i$, 
with~$p_i\in\P^\s$, $k_i\in\K^\s$.
Then it follows from Proposition \ref{Supportwhittakers}\ref{point2} that
\begin{align*}
\CCC_{in/e_\so}(\W_\bl) 
&=
\int_{(\K^\s\cap\P^\s)\backslash(\K^\s\cap\P^\s)k_i\J^\s}
\int_{\N^\s\backslash\P^\s}\W_\bl^{in/e_\so}(pk)\ dp\ dk \\
&=
\int_{(\J^\s\cap(\K^\s\cap\P^\s)^{k_i})\backslash\J^\s}
\int_{\N^\s\backslash\P^\s}\W_\bl^{in/e_\so}(pk_{i}j)\ dp\ dj
\end{align*}
where $dj$ is the right invariant measure on 
$(\J^\s\cap(\K^\s\cap\P^\s)^{k_i})\backslash\J^\s$ 
corresponding to $dk$.

Let us compute the inner integral.
By applying the change of variable $p\mapsto pp_i^{-1}$ 
and then~Pro\-po\-sition \ref{Supportwhittakers}\ref{point3}, 
we get 
\begin{align*}
\int_{\N^\s\backslash\P^\s}\W_\bl^{in/e_\so}(pk_{i}j)\ dp
&=
\int_{\N^\s\backslash\P^\s}\W_\bl^{in/e_\so}(p\w_{\bl}^{i}j)\ dp \\
&= 
\int_{\N^\s\backslash\N^\s(\P^\s\cap\J^\s)} 
\W_\bl^{in/e_\so}(p\w_{\bl}^{i}j)\ dp \\
&=
\int_{(\N^\s\cap\J^\s)\backslash(\P^\s\cap\J^\s)} 
\W_\bl^{in/e_\so}(m\w_{\bl}^{i}j)\ dm
\end{align*}
where $dm$ is the right invariant measure on 
$(\N^\s\cap\J^\s)\backslash(\P^\s\cap\J^\s)$ corresponding to $dp$.
Since $\w_{\bl}^{i}j\in\bJ^\s$ by Lemma \ref{unramramiwas}, 
and thanks to Proposition \ref{invariant integral}(ii),
this is equal to $\LL_{\bl}(\Jj_{\bl})$. 

Now let us fix $dm$ so that $\LL_{\bl}(\Jj_{\bl})=1$,
which is possible thanks to Proposition \ref{invariant integral}(i).
This defines $dp$ uniquely. 
Then our choice of $dk$ gives us
\begin{equation*}
\CCC_{in/e_\so}(\W_\bl) 
= dk((\P^\s\cap \K^\s)\backslash (\P^\s\cap \K^\s)k_i\J^\s)
=  q_\so^{-in/e_\so}
\end{equation*}
as expected.
\end{proof}

We now prove our main result on test vectors for Asai $\L$-functions.

\begin{theorem}\label{Testvectortheorem}
Suppose $\pi$ is a distinguished cuspidal representation of $\G$. Then
\[
\I_\As(s,\one_{\o_\so^n},\W_{\bl})=\frac{1}{1-q_\so^{-s\frac{n}{e_\so}}}=\L_\As(s,\pi)
\]
where the right invariant measure $dg$ defining the left hand side is chosen so that 
\eqref{noumea} holds 
and~the measure $dp$ defining \eqref{izero}
is the one given by Proposition \ref{coef}.
\end{theorem}

\begin{proof}
By Lemma~\ref{Lemma1}, we have
\[
\I_\As(s,\mathbf{1}_{\o_\so^n},\W_{\bl})= \frac{1}{1-
  q_\so^{-sn}}\int_{(\K^\s\cap\P^\s)\backslash
  \K^\s}\I^{(0)}_\As(s,k\cdot\W_{\bl})\ dk.
\]
By Proposition~\ref{coef}, we have
\begin{align*}
\I_\As(s,\one_{\o_\so^n},\W_{\bl})&=\frac{1}{1- q_\so^{-sn}}\sum_{i=0}^{e_\so-1}q_\so^{\frac{-in}{e_\so}}q_\so^{-\frac{in}{e_\so}(s-1)}=\frac{1}{1-q_\so^{-s\frac{n}{e_\so}}}.
\end{align*}
The result then follows immediately from
\eqref{panait}. 
\end{proof}

\begin{corollary} \label{Testvectorcorollary}
Let~$\pi$ be a cuspidal representation of~$\G$ such that~$\L_{\As}(s,\pi)$ is 
not $1$.
Let~$\chi$ be an unramified character of~$\F_\so^\times$ such 
that~$\pi$ is $\chi$-distinguished. 
Then 
\[
\I_\As(s,\one_{\o_\so^n},(\widetilde{\chi}\circ\det)\W_{\bl})=\L_\As(s,\pi)
\]
for any unramified character $\widetilde{\chi}$ of $\F^\times$ extending 
$\chi$. 
\end{corollary}

\section{Flicker--Kable root numbers for cuspidal representations}

In this section, using Theorem \ref{Testvectortheorem}, we compute the local 
Asai root number, as defined by Flicker and Kable, of a cuspidal distinguished 
representation of $\G=\GL_n(\F)$. 
Our methods here are purely local. 

Let us fix once and for all a non-trivial complex character $\psi_\so$ of 
$\F_\so$, 
and a non-zero element $\d\in \mult\F$ such that $\tr_{\F/\F_\so}(\d)=0$. 
We consider the character
\begin{equation}
\label{psifd}
\psi_\so^{\F,\d} : x\mapsto \psi_\so(\tr_{\F/\F_\so}(\d x))
\end{equation}
of $\F$.
As the characteristic of $\F_\so$ is not $2$, this character is 
a non-trivial character of $\F$ trivial on~$\F_\so$. 
Conversely, a non-trivial character of $\F$ trivial on $\F_\so$ is of the 
form $\psi_\so^{\F,t\d}$ for a unique $t\in\F_\so^\times$.

We denote by $\psi=\psi^{\d}$ the standard $\s$-self-dual non-de\-ge\-ne\-rate 
character of~$\N$ attached to \eqref{psifd}, namely
\begin{equation}
\label{Muller}
\psi=\psi^{\d} : u\mapsto 
\psi_\so(\tr_{\F/\F_\so}(\d (u_{1,2}+\dots+u_{n-1,n})).
\end{equation}
Given a generic irreducible complex representation $\pi$ of $\G$, 
its Asai integrals satisfy a local func\-tional equa\-tion
(see the appendix of \cite{Flicker-zeroes} and \cite[Theorem 3]{Kable}): 
there is a unique element $\epsilon_{\As}^{\FK}(s,\pi,\psi_\so,\d)$ in the 
units of $\CC[q_{\so}^{s},q_\so^{-s}]$,
called the local Asai epsilon factor, 
such that, for all functions $\W\in\Ww(\pi,\psi^\d)$ and 
$\Phi\in\Cc_c^\infty(\F_\so^n)$, we have
\begin{equation}
\label{FE}
\frac{\I_\As(1-s,\widehat{\Phi},\widetilde{\W})}{\L_\As(1-s, \pi^{\vee})} =
\epsilon_\As^{\FK}(s,\pi,\psi_\so,\d) \cdot \frac{\I_\As(s,\Phi,\W)}{\L_\As(s,\pi)}
\end{equation}
where:
\begin{enumerate}
\item
$\widehat{\Phi}=\widehat{\Phi}^{\psi_{\so}}$ 
denotes the Fourier transform of $\Phi$ with respect 
to the character $\psi_\so\otimes\dots \otimes \psi_\so$ 
of $\F_\so^n$ and its associated self-dual Haar measure, 
and
\item
$\widetilde{\W}$ is the function in $\Ww(\pi^\vee,\psi^{-\d})$
defined by
\begin{equation*}
\widetilde{\W}(g)=\W(w_0g^*),
\quad g\in\G,
\end{equation*}
where $w_0$ is the antidiagonal permutation matrix of maximal length
and $g^*$ is the transpose of $g^{-1}$.
\end{enumerate}
Notice that the epsilon factor defined above is the one used 
in \cite{Kable}; it differs by a sign from the one defined in 
\cite{Flicker-zeroes}. 
In the next section we will address the question of proper normalization.

Before stating the main result of this section, let us make one observation
on Asai root numbers~of distinguished generic representations of $\G$. If
$\pi$ is such a representation, then applying the functional equation for 
$\I_\As(s,\Psi,\W)$ and $\I_\As(1-s,\widehat{\Phi},\widetilde{\W})$ gives us 
\begin{equation*}
\epsilon_\As^{\FK}(s,\pi,\psi_\so,\d) \cdot
\epsilon_\As^{\FK}(1-s,\pi^{\vee},\psi_\so,-\d) 
= \omega_{\pi}(-1) 
\end{equation*}
as in \cite[Theorem 3]{Kable}.
Since $\pi$ is distinguished, its central character is trivial on
$\F_\so^\times$ and $\pi^\vee\simeq\pi^\s$.
Since $\pi$ and $\pi^\s$ have the same local Asai $\L$-factor 
and 
$\epsilon_\As^{\FK}(s,\pi,\psi_\so,\d)=\epsilon_\As^{\FK}(s,\pi^\s,\psi_\so,-\d)$, 
we get
\[\epsilon_\As^{\FK}\left(\frac 1 2,\pi,\psi_\so,\d\right)
\in\{-1,1\}.
\] 
It is expected that this number is $1$ (cf. \cite[Remark 4.4]{AnandRoot}). 
Here we prove it when $\pi$ is a distingui\-shed cus\-pidal representation. 

\begin{theorem}
\label{distinguished cuspidal Asairootnumber}
Let $\pi$ be a distinguished cuspidal representation of $\G$.
Then 
\begin{equation*}
\epsilon_\As^{\FK}\left(\frac 1 2,\pi,\psi_\so,\d\right)=1.
\end{equation*}
\end{theorem}

\begin{proof}
Since we have already observed that the possible values for this epsilon 
factor are $-1$ and $1$, we just need to show that
$\epsilon_\As^\FK(1/2,\pi,\psi_0,\d)$ is positive. 
To show this it is sufficient to show that 
$\epsilon_\As^\FK(0,\pi,\psi_\so,\d)$ is positive since
\begin{equation*}
\epsilon_\As^\FK(s,\pi,\psi_\so,\d)= 
q_\so^{m(s-1/2)}\cdot\epsilon_\As^{\FK}\left(\frac 1 2,\pi,\psi_\so,\d\right)
\end{equation*} 
for some $m\in\ZZ$ as $\epsilon_\As^\FK(s,\pi,\psi_\so,\d)$ is just a unit in 
$\CC[q_{\so}^{s},q_\so^{-s}]$. 

Fix a Whittaker datum $(\N,\psi_1)$ and a $\s$-self-dual type $(\bJ,\bl)$ 
as in Lemma \ref{basis}. 
The symbol $\sim$ will stand for equality up to a positive constant. By
Theorem \ref{Testvectortheorem}, there is 
$\W_{\bl}\in\Ww(\pi,\psi_1)$ such that 
\[\I_\As(s,\Phi_0, \W_{\bl})\sim \L_\As(s,\pi)\] 
where $\Phi_0$ is the characteristic function of 
$\o_\so^n$ in $\F_\so^n$.
As $\psi_1(u)=\psi(tut^{-1})$ for some diagonal matrix $t$ with 
coefficients in $\F_\so^\times$, the function 
\[\W_0 : g \mapsto \W_{\bl}(t^{-1}g)\] belongs to $\Ww(\pi,\psi)$.
We may (and will) 
even assume that the bottom coefficient on the diagonal~of $t$ is $1$. 
Applying the change of variable $g\mapsto t^{-1}g$, 
we check that 
\begin{equation*}
\I_\As(s,\Phi_0,\W_0)\sim \I_\As(s,\Phi_0,\W_{\bl}).
\end{equation*}
Applying the functional equation, we get 
\begin{equation*}
\frac{\I_\As(1,\widehat{\Phi}_0,\widetilde{\W}_0)}{\L_\As(1,\pi^\vee)} 
\sim \epsilon^\FK_\As(0,\pi,\psi_\so,\d).
\end{equation*}
Let $l$ and $l'$ denote the linear forms on $\Ww (\pi,\psi)$ defined by 
\[l:\W\mapsto \int_{\N^\sigma\backslash \P^\sigma} \W(h)\ dh
\quad 
\text{and}
\quad
l':\W\mapsto
  \int_{\N^\sigma\backslash \P^\sigma} \widetilde{\W}(h)\ dh.\] 
Both these linear forms are defined by convergent integrals: 
by \cite[Corollary 5.19]{BZ76}
the supports in $\G^\s$ of the
integrands~are compact mod $\N^\sigma$ on $\G^\sigma$. 
They are $\G^\s$-invariant by \cite[Theorem 3.1.2]{Ok}; 
by multiplicity $1$, 
they thus differ by a scalar,
which is positive by \cite[Theorem 7.2]{AnandMatringe}. 
By the proof~of \cite[{Theorem 1.4}]{MR2063106}, 
we have 
\[\I_\As(1,\widehat{\Phi}_0,\widetilde{\W}_0)
\sim 
\Phi_0(0)l'(\W_0).\] 
On the other hand, we have 
\[\L_\As(1,\pi^\vee)=\L_\As(1,\pi^\s)=\L_\As(1,\pi)\sim
  \I_\As(1,\Phi_0,\W_0)\sim \widehat{\Phi}_0(0)l(\W_0),\] 
the last equality by \cite{MR2063106} again. 
In particular we get 
\begin{equation*}
\frac{\I_\As(1,\widehat{\Phi}_0,\widetilde{\W}_0)}{\L_\As(1,\pi^\vee)}
\sim
\Phi_0(0)\widehat{\Phi}_0(0)^{-1}
\end{equation*} 
and the right hand side is positive thanks to our choice of $\Phi_0$.
Hence we have $\epsilon_\As^\FK(0,\pi,\psi_\so,\d)>0$, which implies 
that $\epsilon_\As^\FK(1/2,\pi,\psi_\so,\d)=1$. 
\end{proof}

\begin{remark}
In the proof above, we used results written in 
characteristic $0$ only.
Let us explain why they are valid in characteristic $p$ as well. 
First notice that, 
{as $\Hom_{\G^\s}(\pi,\one)$ 
and $\Hom_{\P^\s}(\pi,\one)$ are equal by Ok \cite[Theorem 3.1.2]{Ok},}
the computation borrowed from~the proof of
\cite[{Theorem~1.4}]{MR2063106} holds for $\F$ of arbitrary 
characteristic.
In \cite[Theorem 7.2]{AnandMatringe}, 
and more ge\-ne\-rally in \cite{AnandMatringe}, 
the field~$\F$~is assumed to have~cha\-rac\-teristic $0$. 
In fact, appealing to \cite[Theorem 6.3]{AnandMatringe} 
is enough in the cuspidal~ca\-se,
since a distinguished cuspidal representation of $\G$ is always 
unitary (as its central character is). 
Now the only ingredient in the proof of \cite[Theorem 6.3]{AnandMatringe} 
which uses this rectriction on the~charac\-teristic of $\F$ is that the
{Godement--Jacquet}
ep\-si\-lon 
factor $\e(1/2,\pi,\psi)$ is equal to $1$, for which 
\cite{AnandMatringe}~re\-fers 
to \cite{MatringeOffen}, but the cuspidal case of this~re\-sult is
already in \cite{Ok} and this reference does not assume the 
characteristic of $\F$ to be $0$.
\end{remark}

\section{Comparing Asai epsilon factors}
\label{section Asai RS epsilon factors}

In this section, we
compare, for $\pi$ a generic unramified representation of $\G=\GL_n(\F)$ 
(not neces\-sarily distinguished), the Flicker--Kable Asai epsilon factor 
{to the Asai epsilon factor of $\pi$ defined via the
Langlands--Shahidi method}.
This naturally leads to the normalization 
we give in De\-fi\-ni\-tion \ref{paulveyne}.
Beuzart-Ples\-sis came up to the same normalization in \cite{B-P18}.
Then, we show by a global~ar\-gument that,~for cuspidal~re\-presentations, 
all these definitions of the Asai epsilon factor coincide.~In
par\-ti\-cu\-lar
we answer some
questions posed in \cite[Remark 4.4]{AnandRoot}.

\subsection{Changing the additive character}
\label{not91}

We denote by $\WW_\F$ the Weil group of $\F$ 
with repect to a given separable closure $\overline{\F}$ of $\F$, 
and by $\WW'_\F$ the cor\-responding Weil--Deligne 
group, that is, its direct product by ${\rm SL}(2,\CC)$. 
We use a similar~no\-tation for $\F_{\so}$.
We will write $\Ind'_{\F/\F_\so}$ and $\M'_{\F/\F_\so}$ for induction and 
multiplicative induction 
{(de\-fi\-ned for instance in \cite[Section 7]{Pra92})}
from $\WW'_{\F}$ to $\WW'_{\F_\so}$.
We will also write $\Ind_{\F/\F_\so}$ and $\M_{\F/\F_\so}$ 
for induction and multiplicative induction from $\WW_\F$ to 
$\WW_{\F_\so}$.

Given an irreducible~re\-pre\-sen\-tation $\pi$ of $\G$,
we denote by $\rho(\pi)$ its Langlands parameter, which is a 
finite dimensional representation of $\WW'_\F$.
Then, using local class field theory to identify characters of $\WW'_{\F}$ and of 
$\F^\times$, we have:
\begin{equation}
\label{jeanjacques}
\det\left(\M'_{\F/\F_{\so}}(\rho(\pi))\right)
=\om_{\F/\F_\so}^{n(n-1)/2}\cdot
{\om_\pi^n}|_{\F_\so^\times}.
\end{equation}
When $\pi=\chi$ is a character of $\F^\times$, 
which we identify with $\rho(\chi)$, 
this tells us that $\M'_{\F/\F_{\so}}(\chi)$ is the restriction of 
$\chi$ to ${\F_\so^\times}$.

Given a generic irreducible representation $\pi$ of $\G$,
we denote 
\begin{enumerate}
\item
by $\e_\As^\LS(s,\pi,\psi_\so)$ and $\L_\As^\LS(s,\pi)$ 
the Asai local factors attached 
to $\pi$ via the Langlands--Shahidi method 
(see \cite{Sh-Plancherel90}
when $\F$ has characteristic $0$ and \cite{LomeliLfunctions} 
when $\F$ has characteristic $p$), 
\item
by $\e_\As^{\Gal}(s,\pi,\psi_\so)$ and $\L_\As^{\Gal}(s,\pi)$
the Langlands--Deligne local constants 
of the local Asai transfer $\M'_{\F/\F_{\so}}(\rho(\pi))$
of the Langlands parameter of $\pi$ 
(see \cite[Section 7]{Pra92}). 
\end{enumerate}
When $\F$ has characteristic $0$,
the Asai local $\L$-functions 
$\L_\As^{}(s,\pi)$, $\L_\As^{\LS}(s,\pi)$ and $\L_\As^{\Gal}(s,\pi)$
are known to be all equal.
We will see in the appendix (Theorem \ref{manciniA1})
that this still holds in characteristic $p$.

By \cite{Sh-Plancherel90} when $\F$ has characteristic $0$ 
and by \cite{HenniartLomeli} when $\F$ has characteristic $p$, 
when $\pi$ is unramified and generic we have: 
\[\e_\As^{\LS}(s,\pi,\psi_\so)=\e_\As^{\Gal}(s,\pi,\psi_\so)\] 
whereas by \cite{Henn-L} when $\F$ has characteristic $0$ 
and \cite{HenniartLomeli} when $\F$ has characteristic $p$, 
when $\pi$ is generic we have: 
\[\e_\As^{\LS}(s,\pi,\psi_\so)=\zeta\cdot\e_\As^{\Gal}(s,\pi,\psi_\so)\] 
where $\zeta$ is a root of unity independent from $\psi_\so$, 
which is expected to be $1$, and known to be $1$ when $\F$ has characteristic 
$p$. 

We first describe how all these epsilon factors depend on 
$\psi_\so$. 
Given $t\in \F_\so^\times$, we write $\psi_{\so,t}$ for the character
$x\mapsto\psi_\so(tx)$ of $\F_\so$.

\begin{lemma}\label{lemma twisting the additive character}
Let $\pi$ be generic irreducible representation of $\G$ and $t\in \mult\F_\so$.
Then we have
\begin{enumerate}
\item
$\epsilon_\As^\FK(s,\pi,\psi_{\so,t},\d) 
= \om_\pi(t)^{n} \cdot |t|_\so^{n^2(s-1/2)} \cdot \epsilon_\As^\FK(s,\pi,\psi_\so,\d)$, 
\item
$\epsilon_\As^\LS(s,\pi,\psi_{\so,t})
= \om_\pi(t)^n\cdot |t|_\so^{n^2(s-1/2)} \cdot \om_{\F/\F_\so}(t)^{n(n-1)/2} \cdot 
\epsilon_\As^\LS(s,\pi,\psi_\so)$,
\item
$\e_\As^{\Gal}(s,\pi,\psi_{\so,t})
= \om_\pi(t)^n \cdot |t|_\so^{n^2(s-1/2)} \cdot \om_{\F/\F_\so}(t)^{n(n-1)/2} \cdot 
\e_\As^{\Gal}(s,\pi,\psi_{\so})$.
\end{enumerate}
\end{lemma}

\begin{proof}
We first give the proof of (i) for convenience of the reader; 
it follows verbatim the analogue for Rankin--Selberg $\L$-factors 
in p. 7 of \cite{JacquetRSarchi}. 
As before we set $\psi=\psi^\d$. 
We introduce the matrix \[a=\diag(t^{n-1},\dots,t,1).\] 
Thus we have $\W\in \Ww(\pi,\psi^\d)$ if and only if the function
$\W_{a}:g\mapsto\W(ag)$ is in $\Ww(\pi,\psi^{t\d})$. 
Now take 
$\W\in \Ww(\pi,\psi)$ and $\Phi\in \Cc_c^\infty (\F_\so^n)$, 
and notice that 
$\Phi(\etavec a^{-1} h)=\Phi(\etavec h)$ for all $h\in \G^\sigma$. 
Then
\begin{eqnarray*}
\I_\As(s,\Phi,\W_{a}) 
&=& 
\int_{\N^\s\backslash \G^\s}\W(ah)\Phi(\etavec h)|\det(h)|_\so^s\ dh \\ 
&=& 
\mu(t)\cdot\int_{\N^\s\backslash \G^\s}\W(h)\Phi(\etavec h)
|\det(a^{-1}h)|_\so^s\ dh \\
&=& 
\mu(t)\cdot|t|_\so^{-n(n-1)s/2}\cdot\I_\As(s,\Phi,\W)
\end{eqnarray*}
for some positive character $\mu$ of $\F_\so^\times$.
On the other hand, for all $h\in\G^\s$, we have 
\begin{equation*}
\W(aw_0h^*) 
= \W(w_0(a^{* w_0}h)^*)
= \widetilde{\W}(a^{* w_0} h)
= \widetilde{\W}(t^{1-n} ah).
\end{equation*}
It follows that
\begin{eqnarray*}
\I_\As(1-s,\widehat{\Phi}^{\psi_{\so,t}},\widetilde{\W}_{a})
&=& 
\int_{\N^\s\backslash \G^\s}\widetilde{\W}(t^{1-n} ah)
\widehat{\Phi}^{\psi_{\so,t}}(\etavec h) 
|\det(h)|_\so^{1-s}\ dh \\
&=& 
\om_\pi(t)^{n-1}\cdot
\int_{\N^\s\backslash 
  \G^\s}\widetilde{\W}(ah)\widehat{\Phi}^{\psi_{\so,t}}(\etavec h) 
|\det(h)|_\so^{1-s}\ dh \\
&=& 
\om_\pi(t)^{n-1}\cdot\mu(t)\cdot\int_{\N^\s\backslash \G^\s}
\widetilde{\W}(h)\widehat{\Phi}^{\psi_{\so,t}}(\etavec h) 
|\det(a^{-1}h)|_\so^{1-s}\ dh \\
&=&
\om_\pi(t)^{n-1}\cdot\mu(t)\cdot|t|_\so^{n(n-1)(s-1)/2}\cdot
\int_{\N^\s\backslash \G^\s}
\widetilde{\W}(h)\widehat{\Phi}^{\psi_{\so,t}}(\etavec h) |\det(h)|_\so^{1-s}\ dh.
\end{eqnarray*}
Now we use the relation 
\begin{equation}
\label{changeFourier}
\widehat{\Phi}^{\psi_{\so,t}}(x)=|t|_\so^{n/2}\cdot\widehat{\Phi}(tx),
\quad
x\in\F_\so^n,
\end{equation}
to get 
\begin{eqnarray*}
\int_{\N^\s\backslash \G^\s}
\widetilde{\W}(h)\widehat{\Phi}^{\psi_{\so,t}}(\etavec h) |\det(h)|_\so^{1-s}\ dh
&=& 
|t|_\so^{n/2}\cdot
\int_{\N^\s\backslash \G^\s}\widetilde{\W}(h)\widehat{\Phi}(\etavec t h) 
|\det(h)|_\so^{1-s}\ dh \\
&=& 
|t|_\so^{n/2}\cdot\int_{\N^\s\backslash \G^\s}
\widetilde{\W}(t^{-1}h)\widehat{\Phi}(\etavec h) 
|\det(t^{-1}h)|_\so^{1-s}\ dh \\
&=& 
|t|_\so^{n/2} \cdot
|t|_\so^{n(s-1)} \cdot
\om_\pi(t) \cdot
\I_\As(1-s,\widehat{\Phi},\widetilde{\W}).
\end{eqnarray*}
We thus get the relation:
\begin{equation*}
\e_{\As}^{\FK}(s,\pi,\psi_{\so,t},\d) =
\frac{\om_\pi(t)^{n} \cdot \mu(t) \cdot |t|_\so^{n(n-1)(s-1)/2 + n(s-1/2)}}
{\mu(t) \cdot |t|_\so^{-n(n-1)s/2}}
\cdot \e_{\As}^{\FK}(s,\pi,\psi_\so,\d) 
\end{equation*}
which gives us the expected result.

Now, as we noticed that $\e_\As^{\Gal}(s,\pi,\psi_\so)$ 
and $\e_\As^\LS(s,\pi,\psi_\so)$ are equal up to a non-zero constant 
which does not depend on $\psi_\so$, it is enough to prove (iii).
Then by the properties of the Langlands-Deligne constants in 
\cite{TateCorvallis}, one has 
\[\e_\As^{\Gal}(s,\pi,\psi_{0,t})=|t|_\so^{n^2(s-1/2)}\cdot
\det\left(\M'_{\F/\F_{\so}}(\rho(\pi))\right)\cdot\e_\As^{\Gal}(s,\pi,\psi_\so)\]
which, together with \eqref{jeanjacques}, 
gives the expected result.
\end{proof}

We will also need the following relation satisfied by 
$\e_{\As}^{\FK}(s,\pi,\psi_{\so},\d)$. 
Note that though 
$\psi_{\so,t}^{\F,\d}=\psi_{\so\phantom{,t}}^{\F, t\d}$, it is not true that
$\e_{\As}^{\FK}(s,\pi,\psi_{\so,t},\d)=\e_{\As}^{\FK}(s,\pi,\psi_{\so},t\d)$ 
since changing the character $\psi_\so$ changes the~Fou\-rier transform 
in the functional equation. 
Here is what happens when one changes $\delta$.

\begin{lemma}\label{lemma twisting delta}
Let $\pi$ be a generic irreducible representation of $\G$ and 
$t\in\F_\so^\times$. 
Then we have
\[\epsilon_\As^\FK(s,\pi,\psi_\so,t\d) = 
\om_\pi(t)^{n-1} \cdot |t|_\so^{n(n-1)(s-1/2)} \cdot \epsilon_\As^\FK(s,\pi,\psi_\so,\d).\] 
\end{lemma}

\begin{proof}
Going through the exact same computations as in the proof of Lemma 
\ref{lemma twisting the additive character}, 
but taking the Fourier transform of $\Phi$ with respect to $\psi_\so$ 
rather than $\psi_{\so,t}$ in the computations, we arrive at  
\begin{equation*}
\I_\As(1-s,\widehat{\Phi},\widetilde{\W}_{a})
= 
\om_\pi(t)^{n-1}\cdot \mu(t)\cdot|t|_\so^{n(n-1)(s-1)/2}\cdot 
\I_\As(1-s,\widehat{\Phi},\widetilde{\W})
\end{equation*}
whereas the relation 
\begin{equation*}
\I_\As(s,\Phi,\W_{a})
= \mu(t)\cdot|t|_\so^{-n(n-1)s/2}\cdot\I_\As(s,\Phi,\W)
\end{equation*} 
does not change. 
From this we obtain: 
\begin{eqnarray*}
\e_{\As}^{\FK}(s,\pi,\psi_\so,\d)
&=& \left(|t|_\so^{-n(n-1)(1-s)/2}\cdot\mu(t)\cdot\om_\pi(t)^{n-1}\right)^{-1}
\cdot\mu(t)\cdot|t|_\so^{-n(n-1)s/2}\cdot\e_{\As}^{\FK}(s,\pi,\psi_\so,t\d) \\
&=& \om_\pi(t)^{1-n}\cdot|t|_\so^{n(n-1)(1-2s)/2}\cdot\e_{\As}^{\FK}(s,\pi,\psi_\so,t\d)
\end{eqnarray*}
which gives us the expected result.
\end{proof}

\subsection{Unramified representations}\label{section unramified}

We are going to compute $\epsilon_\As^\FK(s,\pi,\psi_\so,\d)$ and
$\epsilon_\As^\LS(s,\pi,\psi_\so,\d)$ 
when $\pi$ is generic and unramified.
From now on, 
Haar measures on any closed subgroup $\H$ of $\G$ will be 
normalized so that they give volume $1$ to $\H\cap\K$. 
This also normalizes all right invariant measures on quotients of the type 
$\U\backslash\H$ when\-ever $\U$ is an unimodular closed subgroup of $\H$. 

First, we 
{perform a test vector computation similar to that done by Flicker 
when $\F/\F_\so$ is unramified.} 
We~sup\-po\-se that $\pi$ is a generic unramified
repre\-sen\-ta\-tion of $\G$; we denote by $\W_0$ the normalized spherical 
vector in $\Ww(\pi,\psi)$ and by 
$\Phi_0$ the characteristic function of $\o_\so^n$.

Recall that the \textit{conductor} of an additive character of a finite
extension $\E$ of $\F_\so$ is the largest integer $i$ such that it is 
trivial on $\p_\E^{-i}$. 

\begin{proposition}\label{prop test flicker}
Let $\pi$ be a generic unramified representation of $\G$ and suppose that 
the character $\psi_\so^{\F,\d}$ defined by \eqref{psifd} has conductor $0$.
Then
\[\I_\As(s,\Phi_0,\W_0)=\L_\As(s,\pi).\]
\end{proposition}

\begin{proof}
When $\F/\F_\so$ is unramified, 
this is proved in \cite[Section 3]{Flicker} where
the unitarity assumption~is unnecessary. 
In the ramified case, we have $q=q_\so$.~We write 
$\pi=\mu_1\times \dots \times \mu_n$ where the product notation stands for 
parabolic induction, and the characters 
$\mu_1,\dots,\mu_n$ of $\mult \F$ are unramified. 
Let~us fix a uniformizer $\w$ of~$\F$ such that $\w_\so=\w^2$ is a uniformizer 
of~$\F_\so$. 
For $i=1,\dots,n$, set $z_i=\mu_i(\w)$.
With notations as at p. 306 of \cite{Flicker}, as $\w_\so^\l=\w^{2\l}$, we find 
\[\I_\As(s,\Phi_0,\W_0)=\sum_{\l} q^{-s\cdot\tr(\l)}s_{2\l}(z_1,\dots,z_n)
=\sum_{\l} s_{2\l}(z_1q^{-s/2},\dots,z_nq^{-s/2}) \] 
where the sum ranges over all partitions of length $\<n$ 
and $s_{2\lambda}$ is the Schur function 
(see \cite[(3.1) p.~40]{Ronald})
asso\-cia\-ted to 
the partition $2\lambda$ obtained by multiplying the entries 
of $\l$ by $2$.
By~\cite[Exam\-ple 5a, p.~77]{Ronald}, 
the sum above is equal to 
\begin{equation}
\label{lineament}
\prod\limits_{1\<i\<n} (1-z_i^2q^{-s})\cdot
\prod\limits_{1\<k<l\<n} (1-z_k z_lq^{-s})
\end{equation}
Now the Langlands parameter $\rho(\pi)$ is the direct sum 
$\mu_1\oplus\dots\oplus\mu_n$.
Since it is trivial on ${\rm SL}_2(\CC)$ we consider it as a representation of 
$\WW_\F$ only. 
Since $\mu_i\circ\s=\mu_i$ for all~$i$,
we have 
\begin{equation*}
\label{garnement}
\M_{\F/\F_\so}(\mu_1\oplus \dots \oplus \mu_n) 
= \bigoplus\limits_{1\<i\<n} {\mu_i}|_{\F_\so^\times} \oplus 
\bigoplus\limits_{1\<k<l\<n} \Ind_{\F/\F_\so} (\mu_k\mu_l)
\end{equation*}
by \cite[Lemma 7.1]{Pra92}. 
Thus
$\L_\As^{\Gal}(s,\pi)$ is equal to \eqref{lineament}.
The result follows from Theorem \ref{manciniA1}.
\end{proof}

\begin{remark}
At this point, we note that the authors of \cite{AnandMatringe} appeal to 
Flicker's unramified computation even when $\F/\F_\so$ is ramified, however 
Proposition \ref{prop test flicker} shows that there is no harm in doing that. 
\end{remark}

When $\F/\F_\so$ is unramified, one can choose $\d$ to be a unit, whereas 
when $\F/\F_\so$ is ramified, one can choose $\d$ to have valuation $-1$.
In both cases, the character $\psi_\so^{\F,\d}$ has conductor $0$ 
if $\psi_\so$ has conductor $0$. 
In this case, the functional equation, 
together with Proposition \ref{prop test flicker}, 
the fact that $\widehat{\Phi}_0=\Phi_0$ and that 
$\widetilde{\W}_0$ is the normalized spherical vector in 
$\mathcal{W}(\pi^\vee,\psi^{-1})$, 
tells us that: 

\begin{corollary}\label{cor unramified Flicker-Asai}
Suppose that $\pi$ is a generic unramified representation of $\G$, 
that $\psi_\so$ has conductor $0$ and 
$\d$ has valuation $1-e(\F/\F_\so)$.
Then $\e_\As^\FK(s,\pi,\psi_\so,\d)=1$.
\end{corollary} 

Let us compare this with the unramified situation for the Asai constant 
defined 
via the Langlands-Shahidi method. 
To do this we introduce the local 
Langlands constant $\l(\F/\F_\so,\psi_\so)$
(see for instance \cite[(30.4.1)]{BHbook} for a definition).
We note that $\l(\F/\F_\so,\psi_\so)$ is equal to 
$\e(1/2,\om_{\F/\F_\so},\psi_\so)$,
the Tate~root number of the quadratic character $\om_{\F/\F_\so}$.
We will freely use the relation \cite[(30.4.2)]{BHbook}: 
\[\e(s,\Ind_{\F/\F_\so} \rho,\psi_\so)=
\l(\F/\F_\so,\psi_\so)^{\dim(\rho)}\cdot\e(s,\rho,\psi_\so\circ \tr_{\F/\F0})\] 
where $\rho$ a semi-simple representation of $\WW_\F$ 
and $\Ind_{\F/\F_\so}$ denotes induction from $\WW_\F$ to $\WW_{\F_\so}$.
We will also use the fact that if $\chi$ is an unramified character of
$\E^\times$ for any finite extension $\E$ of $\F_\so$ and $\psi_\E$ 
is a character of $\E$ of conductor $0$, 
then $\e(s,\chi,\psi_\E)=1$ 
(see the remark after (3.2.6.1)~in~\cite{TateCorvallis}). 
More generally, we refer to \cite{TateCorvallis} for the basic facts and 
relations concerning epsilon factors of characters that we will use in this 
section without necessarily recalling. 

\begin{proposition}\label{prop unramified LS-Asai}
Suppose that $\pi$ is a generic unramified representation of $\G$, 
that $\psi_\so$ has conductor $0$ and $\d$ has valuation $1-e(\F/\F_\so)$.
Then 
\[\e_\As^\LS(s,\pi,\psi_\so)=\e_\As^{\Gal}(s,\pi,\psi_0)
=\om_\pi(\d)^{1-n}\cdot|\d|^{-n(n-1)(s-1/2)/2}\cdot\l(\F/\F_\so,\psi_\so)^{n(n-1)/2}\]
where $|\d|$ denotes the normalized absolute value of $\d$.
\end{proposition}

\begin{proof}
We use the notation of the proof of Proposition \ref{prop test flicker}. 
We thus have
\begin{eqnarray*}
\e_\As^{\Gal}\left(\frac 1 2,\pi,\psi_\so\right)
&=&
\prod\limits_{1\<k<l\<n} \e\left(\frac 1 2,
\Ind_{\F/\F_\so}(\mu_k\mu_l),\psi_\so \right) \\
&=&
\l(\F/\F_\so,\psi_\so)^{n(n-1)/2}\cdot
\prod\limits_{1\<k<l\<n}\e\left(\frac 1 2,\mu_k\mu_l,\psi_\so \circ \tr_{\F/\F_\so}\right) \\
&=&
\l(\F/\F_\so,\psi_\so)^{n(n-1)/2}\cdot\prod\limits_{1\<k<l\<n}
    (\mu_k\mu_l)(\d)^{-1}\e\left(\frac 1 2,\mu_k\mu_l,
\psi_\so^{\F,\d}\right) \\
&=&
\l(\F/\F_\so,\psi_\so)^{n(n-1)/2}\cdot\prod\limits_{1\<k<l\<n}
    (\mu_k\mu_l)(\d)^{-1} \phantom{\Bigg)} \\
&=&
\l(\F/\F_\so,\psi_\so)^{n(n-1)/2}\cdot\omega_{\pi}(\d)^{1-n}. \phantom{\Bigg)}
\end{eqnarray*}
where we ignore the epsilon factors equal to $1$.
However 
\[\e_\As^{\Gal}(s,\pi,\psi_\so)=
\e_\As^{\Gal}\left(\frac 1 2,|\cdot |^{(s-1/2)/2}\pi,\psi_\so\right)\] 
hence the previous equality 
gives the result.
\end{proof}

\subsection{Rankin--Selberg epsilon factors}

Proposition \ref{prop unramified LS-Asai},
together with Corollary \ref{cor unramified Flicker-Asai}, 
suggests to introduce the following definition.

\begin{definition}
\label{paulveyne}
For any generic irreducible representation $\pi$, we set:
\[\e_\As^\RS(s,\pi,\psi_\so,\d)=\om_\pi(\d)^{1-n}\cdot
|\d|^{-n(n-1)(s-1/2)/2}\cdot\l(\F/\F_\so,\psi_\so)^{n(n-1)/2}\cdot
\e_\As^\FK(s,\pi,\psi_\so,\d).\] 
\end{definition}

The following result, 
which is an immediate consequence of Lemma \ref{lemma twisting delta},
was brought to our attention by Beuzart-Plessis.

\begin{lemma}\label{lemma independance on delta}
For any generic irreducible representation $\pi$ of $\G$ and $t\in \F_\so^\times$, 
we have: 
\[\e_\As^\RS(s,\pi,\psi_\so,t\d)=\e_\As^\RS(s,\pi,\psi_\so,\d).\] 
\end{lemma}

In particular we can remove $\d$ from the notations, and set 
 \[\e_\As^\RS(s,\pi,\psi_\so)=\e_\As^\RS(s,\pi,\psi_\so,\d)\] 
for any generic irreducible representation $\pi$,
any non-trivial character $\psi_\so$ 
and any element $\d\in \mult \F$ of trace $0$. 
By Lemma \ref{lemma twisting the additive character} and the 
relation 
\begin{equation*}
\label{eq twisting the langlands constant} 
\l(\F/\F_\so,\psi_{\so,t})=\om_{\F/\F_\so}(t)\cdot\l(\F/\F_\so,\psi_\so), 
\end{equation*} 
we have:

\begin{lemma}
\label{lemma twisting the additive character 2}
For any generic irreducible representation $\pi$ of $\G$ and $t\in \F_\so^\times$, 
we have
\begin{equation*}
\e_\As^\RS(s,\pi,\psi_{\so,t}) = 
\om_{\F/\F_\so}(t)^{n(n-1)/2}\cdot\om_\pi(t)^n \cdot |t|_\so^{n^2(s-1/2)}\cdot
\e_\As^\RS(s,\pi,\psi_{\so}).
\end{equation*}
\end{lemma}

Combining 
Proposition \ref{prop unramified LS-Asai}, 
Corollary \ref{cor unramified Flicker-Asai}
and Lemmas \ref{lemma independance on delta} and 
\ref{lemma twisting the additive character 2},
we get the following result.

\begin{theorem}
\label{thm unramified epsilon equality}
For any generic unramified irreducible representation $\pi$ of $\G$,
we have
\begin{equation}
\label{pradelle}
\e_\As^\RS(s,\pi,\psi_\so)=\e_\As^\LS(s,\pi,\psi_\so).
\end{equation}
\end{theorem}

At the end of this section (see Theorem \ref{thm cuspidal epsilon equality}), 
we will show that \eqref{pradelle}
holds for cuspidal representa\-tions as well.

\subsection{The local factors at split places}

Let $k/k_{\so}$ be a quadratic extension of global fields 
of characteristic different from $2$. 
Given a place $v$ of $k_\so$, we write $k_{\so,v}$ for the completion of 
$k_\so$ at $v$ and $k_v=k\otimes_{k_\so} k_{\so,v}$.
In this paragraph, we consider a place $v$ which splits in $k$, 
that is $k_v$ is a split $k_{\so,v}$-algebra.
There are thus two isomorphisms~of $k_{\so,v}$-algebras 
between $k_v$ and $k_{\so,v}\oplus k_{\so,v}$,
and one passes from one to the other by applying the 
auto\-morphism $(x,y)\mapsto(y,x)$.

Let $\pi_v$ be a generic irreducible representation of 
$\G_v=\GL_n(k_v)$ and set $\N_v=\N_n(k_v)$. 
We~fix~a~non-tri\-vial character $\psi_{\so,v}$ of $k_{\so,v}$
and an element $\d_v\in k_v^\times$ such that $\tr_{k_v/k_{\so,v}}(\d_v)=0$, 
and set 
\begin{equation*}
\psi_{\so,v}^{k_v,\d_v} : x \mapsto \psi_{\so,v}(\tr_{k_v/k_{\so,v}}(\d_v x))
\end{equation*}
which is a non-trivial character of $k_v$ trivial on $k_{\so,v}$. 
We denote by $|\cdot|_{\so,v}$ the normalized absolute value on $k_{\so,v}$. 

We first suppose that $v$ is finite, and set $q_{\so,v}$ for the cardinality of 
$k_{\so,v}$. 
Take $\W_v\in \Ww(\pi_v,\psi_v)$ and $\Phi_v\in \Cc_c^\infty(k_{\so,v}^n)$.
By \cite[Theorem 2.7]{JPSS}, the integral 
\[\I_\As(s,\Phi_v,\W_v)=\int_{\N(k_{\so,v})\backslash \G(k_{\so,v})} 
  \W_v(h)\Phi_v(\etavec h) |\det(h)|^s\ dh\] 
is absolutely convergent when the real part of $s$ is 
larger than a real number $r_v$ depending only on $\pi_v$, it
extends to an element of 
$\CC[q_{\so,v}^{s},q_{\so,v}^{-s}]$, and these integrals span a fractional ideal
of $\CC[q_{\so,v}^{s},q_{\so,v}^{-s}]$ generated by a unique Euler factor 
denoted $\L_{\As}(s,\pi_v)$. 
Also, there is a unit in $\CC[q_{\so,v}^{s},q_{\so,v}^{-s}]$, 
which~we denote by 
$\e_{\As}^{\FK}(s,\pi_v,\psi_{\so,v},\d_v)$ 
for the sake of coherent notations, 
such that 
\[\frac{\I_\As(1-s,\widehat{\Phi}_v,\widetilde{\W}_v)}
{\L_\As(1-s,\pi_v^\vee)}
= \e_\As^\FK(s,\pi_v,\psi_{\so,v},\d_v)\cdot
\frac{\I_\As(s,\Phi_v,\W_v)}{\L_\As(s,\pi_v)}\]
where the Fourier transform of $\Phi_v$ is defined with repect to the 
character $\psi_{\so,v}$. 

\begin{definition}
\label{paulveynesplit}
We set 
\begin{equation*}
\e_\As^\RS(s,\pi_v,\psi_{\so,v},\d_v) = 
\om_\pi(\d_v)^{1-n}\cdot
|\d_v|^{-n(n-1)(s-1/2)/2}\cdot
\e_\As^\FK(s,\pi_v,\psi_{\so,v},\d_v).
\end{equation*}
\end{definition}

\begin{remark}
\label{paulveynesplitrmk}
Comparing with Definition \ref{paulveyne} in the inert case,
there is no Langlands constant~ap\-pearing in Definition \ref{paulveynesplit}. 
However, note that the character $\om_{k_v/k_{\so,v}}$ of $k_{\so,v}^\times$ 
trivial on $k_v/k_{\so,v}$-norms is trivial. 
In analogy with the inert case, we may set 
the Langlands constant $\l(k_v/k_{\so,v},\psi_{\so,v})$ to be equal to 
$\e(1/2,\om_{k_v/k_{\so,v}},\psi_{\so,v})$, but this root number is equal to 
$1$ by the classical properties of Tate epsilon factors.
\end{remark}

A computation similar to the one carried out in the proof of 
Lemma \ref{lemma twisting delta} shows that this local factor
$\e_\As^\RS(s,\pi_v,\psi_{\so,v},\d_v)$ is independent of $\d_v$ hence we 
write 
\[\e_\As^\RS(s,\pi_v,\psi_{\so,v})=\e_\As^\RS(s,\pi_v,\psi_{\so,v},\d_v).\] 
When $v$ is archimedean, the discussion above remains true up to the appropriate 
modifications~(the $\L$-factor is meromorphic rather than an Euler factor, 
and the epsilon factor is entire rather than a Lau\-rent monomial) appealing 
to \cite[Theorem 2.1]{JacquetRSarchi} instead of \cite[Theorem 2.7]{JPSS}, 
and we define the local~fac\-tor 
$\e_\As^\RS(s,\pi_v,\psi_{\so,v},\d_v)$
as in Definition \ref{paulveynesplit}.

Now we compare these 
epsilon factors to the epsilon factors of pairs defined by the authors of 
\cite{JPSS} and \cite{JacquetRSarchi}.

\begin{lemma}
\label{lemma split asai vs rankin selberg epsilon factor}
Let $\phi$ be an isomorphism of $k_{\so,v}$-algebras between $k_v$ and 
$k_{\so,v}\oplus k_{\so,v}$.
{It induces~an isomorphism of groups between $\GL_n(k_v)$ and 
$\GL_n(k_{\so,v})\times\GL_n(k_{\so,v})$, still denoted $\phi$.}
Write $\pi_v\circ\phi$ as a 
tensor product $\pi_{1,v}\otimes \pi_{2,v}$ 
of two generic irreducible representations of $\GL_n(k_{\so,v})$. 
Then 
\[\e_\As^\RS(s,\pi_v,\psi_{\so,v})
=\e^{\RS}(s,\pi_{1,v},\pi_{2,v},\psi_{\so,v})
=\e^{\RS}(s,\pi_{2,v},\pi_{1,v},\psi_{\so,v})\] 
where $\e^{\RS}(s,\pi_{1,v},\pi_{2,v},\psi_{\so,v})$ is the epsilon factor denoted 
$\e(s,\pi_{1,v},\pi_{2,v},\psi_{\so,v})$ 
in {\rm \cite[Theorem 2.7]{JPSS}}~if~$v$ is finite, 
and is the one canonically associated to the gamma factor of 
{\rm \cite[Theorem 2.1]{JacquetRSarchi}} if $v$ is archi\-medean.
\end{lemma}

\begin{proof}
Since $\e_\As^\RS(s,\pi_v,\psi_{\so,v})$ does not depend on $\d_v$,
we can choose $\d_v=\phi(1,-1)$. 
Then $\psi_v\circ \phi$ can be written 
$\psi_{\so,v}^{\phantom{-1}}\otimes \psi_{\so,v}^{-1}$ 
and we have
\[\Ww(\pi_v\circ \phi,\psi_v\circ \phi)
=\Ww(\pi_{1,v},\psi^{\phantom{-1}}_{\so,v})\otimes \Ww(\pi_{2,v},\psi_{\so,v}^{-1}).\] 
Moreover, $\om_{\pi_v}(\d_v)^{n-1}=\om_{\pi_{2,v}}(-1)^{n-1}$
hence \[\e_\As^\RS(s,\pi_v,\psi_{\so,v})=\e^{\RS}(s,\pi_{1,v},\pi_{2,v},\psi_{\so,v}).\] 
Now replace $\phi$ by the other isomorphism $\phi'$ of $k_{\so,v}$-algebras 
such that $\phi'\circ\phi^{-1}:(x,y)\mapsto(y,x)$
and replace $\d_v$ by $-\d_v=\phi'(1,-1)$. 
We then get
\[\e_\As^\RS(s,\pi_v,\psi_{\so,v})=\e^{\RS}(s,\pi_{2,v},\pi_{1,v},\psi_{\so,v})\]
which proves the expected result.
\end{proof}

We give another reason for the lemma above to be true for the possibly 
surprised reader. 

\begin{remark}
\label{remark no surprise}
It is in fact well known as a part of the local Langlands correspondence for 
$\GL_n(k_{\so,v})$ that 
\begin{equation}
\label{fouche}
\e^{\RS}(s,\pi_{1,v},\pi_{2,v},\psi_{\so,v})=\e^{\RS}(s,\pi_{2,v},\pi_{1,v},\psi_{\so,v})
\end{equation}
as it is equal to the Langlands--Deligne constant 
\[\e(s,\rho(\pi_{1,v})\otimes \rho(\pi_{2,v}),\psi_{\so,v})
=\e(s,\rho(\pi_{2,v})\otimes \rho(\pi_{1,v}),\psi_{\so,v}).\] 
Equality \eqref{fouche} can also be checked as follows. 
Using the notation of \cite[Theorem 2.7]{JPSS}, one has 
\[\frac{\Psi(1-s,\widetilde{\W}_{1,v},\widetilde{\W}_{2,v},\widehat{\Phi})}
{\L^\RS(1-s,\pi_{1,v}^\vee,\pi_{2,v}^\vee)}
=\om_{\pi_{2,v}}(-1)^{n-1}\cdot
\e^{\RS}(s,\pi_{1,v},\pi_{2,v},\psi_{\so,v})\cdot
\frac{\Psi(1-s,\W_{1,v},\W_{2,v},\Phi)}{\L^\RS(s,\pi_{1,v},\pi_{2,v})}\] 
for $\W_{1,v}\in \Ww(\pi_{1,v},\psi^{\phantom{-1}}_{\so,v})$, 
$\W_{2,v}\in \Ww(\pi_{2,v},\psi_{\so,v}^{-1})$ and 
$\Phi \in \Cc_c^\infty(k_{\so,v}^n)$. 
Similarly, one has 
\[\frac{\Psi(1-s,\widetilde{\W}_{2,v},\widetilde{\W}_{1,v},
\widehat{\Phi}^{\psi_{\so,v}^{-1}})}
{\L^\RS(1-s,\pi_{2,v}^\vee,\pi_{1,v}^\vee)}
=\om_{\pi_{1,v}}(-1)^{n-1}\cdot
\e^{\RS}(s,\pi_{2,v},\pi_{1,v},\psi_{\so,v}^{-1})\cdot
\frac{\Psi(1-s,\W_{2,v},\W_{1,v},\Phi)}{\L^\RS(s,\pi_{2,v},\pi_{1,v})}\] 
for 
$\W_{1,v}\in \Ww(\pi_1,\psi^{\phantom{-1}}_{\so,v})$, 
$\W_{2,v}\in\Ww(\pi_2,\psi_{\so,v}^{-1})$ 
and $\Phi \in \Cc_c^\infty(k_{\so,v}^n)$. 
The $\L$-factors do not depend on the ordering of the representations, 
and a simple change of variable using the relation 
\eqref{changeFourier}
gives 
\begin{equation*}
\Psi\left(1-s,\widetilde{\W}_{2,v},\widetilde{\W}_{1,v},
\widehat{\Phi}^{\psi_{\so,v}^{-1}}\right)
=\om_{\pi_1}(-1)\cdot\om_{\pi_2}(-1)\cdot\Psi(1-s,\widetilde{\W}_{2,v},
\widetilde{\W}_{1,v},\widehat{\Phi}) 
\end{equation*}
whereas 
\[\e^{\RS}(s,\pi_{2,v},\pi_{1,v},\psi_{\so,v}^{-1})
=\om_{\pi_1}(-1)^n\cdot\om_{\pi_2}(-1)^n\cdot\e^{\RS}(s,\pi_{2,v},\pi_{1,v},\psi_{\so,v})\]
by p. 7 of \cite{JacquetRSarchi}. 
Putting the different pieces together yields 
the equality we were looking for. 
\end{remark}

\begin{remark}
At p. 811 of \cite{Kable}, the author notices a sign ambiguity in 
the identification of the Asai epsilon factor  
with $\e^{\RS}(s,\pi_{1,v},\pi_{2,v},\psi_{\so,v})$ due to the ordering of 
$\pi_{1,v}$ and $\pi_{2,v}$. 
Lemma \ref{lemma split asai vs rankin selberg epsilon factor} 
or Remark \ref{remark no surprise} show that there is in fact no such ambiguity.
\end{remark}

\begin{remark}
\label{fabiola}
With the same assumptions as in Lemma 
\ref{lemma split asai vs rankin selberg epsilon factor}, 
we also have the equalities
\[\L_\As(s,\pi_v)
=\L^{\RS}(s,\pi_{1,v},\pi_{2,v})
=\L^{\RS}(s,\pi_{2,v},\pi_{1,v})\] 
between local $\L$-factors.
Note that 
\begin{eqnarray*}
\L^\RS(s,\pi_{1,v},\pi_{2,v})
&=&\L^{\LS}(s,\pi_{1,v},\pi_{2,v}), \\
\e^\RS(s,\pi_{1,v},\pi_{2,v},\psi_{\so,v}) &=& \e^{\LS}(s,\pi_{1,v},\pi_{2,v},\psi_{\so,v})
\end{eqnarray*}
where the factors on the right hand side are the Langlands--Shahidi factors of 
\cite{Sh-certain}. 
It is known by \cite{Sh-Fourier} in the non-archimedean case, 
and by \cite{Sh-certain} in the archimedean case.
\end{remark}

\subsection{Global factors}
\label{section global}

As in the  previous paragraph, $k/k_{\so}$ is a quadratic  extension of global
fields of characteristic different from $2$.
We denote~by $\AA$ the ring of adeles of $k$ and by $\AA_{\so}$ that of 
$k_{\so}$. 
We suppose that all places of $k_\so$ dividing $2$, 
as well as all archimedean places in the number field case,
are split in $k$.

We fix once and for all a non-trivial character $\psi_\so$ of 
$\AA_{\so}/k_\so$ 
and a non-zero element $\d\in k$ such that $\tr_{k/k_0}(\d)=0$. 
Thus 
\begin{equation}
\label{PSIKD}
\psi_{\so}^{k,\d} : x \mapsto \psi_\so(\tr_{k/k_0}(\d x))
\end{equation} 
is a non-trivial character of $\AA$ trivial on $k+\AA_{\so}$. 
Given a place $v$ of $k_\so$,
we denote by $\psi_{\so,v}$ the local component of $\psi_\so$ at $v$.

Let $\Pi$ be a cuspidal automorphic
representation of $\GL_n(\AA)$ as in \cite{BorelJacquet}. 
It decomposes as a restricted tensor product
\begin{equation*}
\Pi=\bigotimes\limits_{v}{}' \ \Pi_v
\end{equation*} 
where $v$ ranges over the set of all places of $k_\so$.
When $v$ is inert in $k$, 
then $\Pi_v$ is the local component of $\Pi$ at the place 
of $k$ above $v$.
When $v$ is split in $k$,
and given an isomorphism $\phi_v$ of $k_{\so,v}$-algebras between $k_v$ 
and $k_{\so,v}\oplus k_{\so,v}$,
the representation $\Pi_v\circ\phi_v$ decomposes as 
$\Pi_{1,v}\otimes \Pi_{2,v}$.

Note that we have
\begin{equation*}
\L_\As(s,\Pi_v)=\L_\As^\LS(s,\Pi_v)=\L_\As^{\Gal}(s,\Pi_v)
\end{equation*}
for any place $v$ of $k_\so$.
See Theorem \ref{manciniA1} when $v$ is inert, 
and Remark \ref{fabiola} when $v$ is split.

The factors 
$\L_\As(s,\Pi_v)$ and $\e_\As^\RS(s,\Pi_v,\psi_{\so,v})$ 
have now been defined at all places of $k_\so$.
We set 
\begin{eqnarray*}
\L_\As(s,\Pi)
&=&\prod_{v} \L_\As(s,\Pi_v), \\
\e_\As^\RS(s,\Pi)
&=&
\prod_{v} \e_\As^\RS(s,\Pi_v,\psi_{\so,v}), \\
\e_\As^\LS(s,\Pi)
&=&
\prod_{v} \e_\As^\LS(s,\Pi_v,\psi_{\so,v})
\end{eqnarray*}
where the products are taken over all places $v$ of $k_\so$.

Note that by \cite{Sh-Plancherel90} and \cite{LomeliLfunctions}, 
the factor $\e_\As^\LS(s,\Pi)$ is indeed independent of 
the character \eqref{PSIKD}
and one has the func\-tional equation 
\begin{equation}
\label{global equation LS} 
\L_\As^{}(s,\Pi)=\e_\As^\LS(s,\Pi) \cdot \L_\As^{}(1-s,\Pi^\vee).
\end{equation}

In fact, we claim that with our normalization
(see Definitions \ref{paulveyne} and \ref{paulveynesplit}), 
we have 
\begin{equation}\label{global equation RS} 
\L_\As(s,\Pi)=\e_\As^\RS(s,\Pi)\cdot \L_\As(1-s,\Pi^\vee).
\end{equation} 
Let us prove this claim. 
Whatever the place $v$ of $k_\so$ is, 
the local functional equation is of the form 
\begin{equation}
\label{piscine}
\begin{split}
\frac{\I_\As(1-s,\widehat{\Phi}_v,\widetilde{\W}_v)}{\L_\As(1-s,\Pi_v^\vee)}
= \om_{\Pi_v}(\d_v)^{n-1}&\cdot|\d|_v^{n(n-1)(s-1/2)/2} \\
& 
\cdot \l(k_v/k_{\so,v},\psi_{\so,v})^{n(n-1)/2}\cdot
\e_\As^\RS(s,\Pi_v,\psi_{\so,v})\cdot
\frac{\I_\As(s,\Phi_v,\W_v)}{\L_\As(s,\Pi_v)}.
\end{split}
\end{equation}
By \cite[Proposition 5]{Kable} and \cite[Section 2, Proposition]{Flicker}, 
together with local multiplicity $1$ for Whitta\-ker functionals, 
if we take a decomposable global Schwartz function $\Phi$ and a 
decomposable global Whit\-ta\-ker function 
$\W$ in the global Whittaker model of $\Pi$,
we have:
\begin{equation*}
\label{paletot}
\prod_{v}\I_\As(1-s,\widehat{\Phi}_v,\widetilde{\W}_v)=
\prod_{v}\I_\As(s,\Phi_v,\W_v).
\end{equation*} 
To be more precise, the left hand side term makes sense when the real part of 
$-s$ is large enough, 
whereas the right hand side term makes sense when the real part of 
$s$ is large enough, and both terms admit meromorphic continuations 
to $\CC$.
It is these meromorphic continuation that are equal. 

Now let $\T$ be a finite set of places of $k_\so$, 
containing the set of archimedean places, 
such that for all $v\notin\T$ 
one has
\begin{eqnarray*}
\I_\As(s,\Phi_v,\W_v) &=& \L_{\As}(s,\Phi_v,\W_v), \\
\om_{\Pi_v}(\d_v) &=& 1, \\
|\d|_v &=& 1, \\
\l(k_v/k_{\so,v},\psi_\so) &=& 1,
\end{eqnarray*}
hence $\e_\As^\RS(s,\Pi_v,\psi_{\so,v})=1$.
Taking the product of the equalities \eqref{piscine} for all $v$, 
we get
\begin{equation*}
\begin{split}
\L_{\As}(s,\Pi) = \Bigg(\prod_{v\in\T} 
\om_{\Pi_v}(\d_v)^{n-1}&\cdot|\d|_v^{n(n-1)(s-1/2)/2} \\
& 
\cdot \l(k_v/k_{\so,v},\psi_\so)^{n(n-1)/2} \cdot
  \e_\As^\RS(s,\Pi_v,\psi_{\so,v})\Bigg)\cdot
\L_{\As}(1-s,\Pi^\vee).
\end{split}
\end{equation*}
Since $\d\in k^\times$, we have
\begin{equation*}
\prod\limits_v \om_{\Pi_v}(\d_v) = \om_\Pi(\d) = 1
\quad\text{and}\quad
\prod\limits_v |\d|_v = 1.
\end{equation*}
On the other hand, we have (see Remark \ref{paulveynesplitrmk})
\begin{equation*}
\prod\limits_v \l(k_v/k_{\so,v},\psi_{\so,v}) = 
\prod\limits_v \e(1/2,\om_{k_v/k_{\so,v}},\psi_{\so,v}) = 
\e(1/2,\om_{k/k_\so}).
\end{equation*}
However, the global root number $\e(1/2,\om_{k/k_\so})$ is equal to $1$
by the dimension $1$ case of the main result
of \cite{FrohlichQueyrut}.
By the assumption on $\T$ we get
\[\prod_{v\in\T} \e_\As^\RS(s,\Pi_v,\psi_{\so,v})= \e_{\As}^{\RS}(s,\Pi)\] 
and (\ref{global equation RS}) follows.
In particular (\ref{global equation LS}) and (\ref{global equation RS}) imply:

\begin{theorem}\label{global equality} 
Let $\Pi$ be a cuspidal automorphic representation of $\GL_n(\AA)$.
Then 
\[\e_\As^\RS(s,\Pi)=\e_\As^\LS(s,\Pi).\]
\end{theorem}

Note that 
the functional equation of \cite[Theorem 5]{Kable} has a different epsilon factor 
and moreover is up to a sign. 
The presence of this sign is due to the fact that at an inert place $v$ of $k_\so$, 
Kable takes the lo\-cal factor 
$\e_\As^\FK(s,\Pi_v,\psi_{\so,v},\d_v)$ whereas we take
$\e_\As^\RS(s,\Pi_v,\psi_{\so,v})$. 

\subsection{Cuspidal representations}

Let $\F/\F_\so$ be our usual quadratic extension of non-archimedean local 
fields of residual characteristic different from $2$.
Fix a global field $k_\so$ such that 
$k_{\so,w}\simeq \F_\so$ at some place $w$. 
Write 
$\F\simeq\F_\so[\X]/(\P_{w})$ for $\P_{w}$ a polynomial of degree $2$ with
coefficients in $\F_\so$. 
We also fix, whenever $v$ is a place of $k_{\so}$ 
in the set $\SS$ made of all archimedean places 
and all finite places dividing $2$,
a polynomial $\P_v\in k_{\so,v}[\X]$ of degree~$2$ 
with simple roots. 
By the weak approximation~lem\-ma, there is a $\P\in k_\so[\X]$ of degree~$2$, 
as close as we want from $\P_v$ in $k_{\so,v}[\X]$ for each $v\in\SS\cup\{w\}$.
Thanks to Krasner's lemma, we take $\P$
close enough such that the extension spanned by its roots in 
the {separable} closure
$\overline{\F}$ of $\F$ is equal to $\F$,
and such that 
$k_{\so,v}[\X]/(\P)$ is split for $v\in\SS$.
Setting $k=k_\so[\X]/(\P)$, we have: 
\begin{enumerate}
\item 
$k$ is split at all archimedean places (when $k$ is a number field) 
and at all places dividing $2$;
\item 
one has $k_{\so,w}\simeq\F_\so$ and $k_{w}\simeq\F$.
\end{enumerate}

We explain below how to realize a cuspidal representation $\pi$ of 
$\G=\GL_n(\F)$ as the~local com\-ponent at $w$ 
of some suitable cuspidal automorphic representation of $\GL_n(\AA)$.
First, we realize its central character $\om_\pi$ as a local component 
of some character $\Omega$ of $\AA^\times/k^\times$.

\begin{lemma}
\label{YvonneDeGalais}
Let $\om$ be a unitary character of $\F^\times$,
and $u$ be a finite place of $k_{\so}$ different from $w$.
Then there exists a unitary automorphic character 
$\Omega:\AA^\times/k^\times\to\CC^\times$ 
such that:
\begin{enumerate}
\item 
the local component of $\Omega$ at $w$ is $\omega$,
\item
for all $v\neq u,w$, the local component $\Omega_v$ is unramified.
\end{enumerate}
\end{lemma}

\begin{proof}
The subgroup 
\[\U = \prod_{v\neq u,w} k_v^{\times 0}\] 
where $v$ ranges over all places of $k_\so$ different from $u,w$
and where $k_v^{\times 0}$ is the maximal compact~sub\-group of $k_v^\times$, 
is compact in $\AA^\times$, 
thus $k^\times\U$ is closed in $\AA^\times$.
The intersection $k^\times\U\cap k_w^{\times}$ is trivial,
thus $k_w^{\times}$ identifies with a locally compact subgroup of 
$\AA^\times/k^\times\U$. 
By Pontryagin duality, $\om$ extends 
to a unitary character $\Omega$ of $\AA^\times/k^\times\U$, 
which satisfies the required con\-ditions.
\end{proof}

\begin{lemma}
Let $\pi$ be a unitary cuspidal representation of $\G$,
and $u$ be a finite place of $k_{\so}$ which is split in $k$. 
Then there is a cuspidal automorphic representation $\Pi$ of $\GL_n(\AA)$ 
such that:
\begin{enumerate}
\item
the local component of $\Pi$ at $w$ is isomorphic to $\pi$,
\item
for all $v\neq u,w$, the local component $\Pi_v$ is unramified.
\end{enumerate}
\end{lemma}

\begin{proof}
The proof follows that of \cite[Appendice 1]{Henn-thesis}, 
adapted to our context, 
the reductive group~of in\-terest here being the restriction of $\GL_n$ from 
$k$ to $k_\so$. 
Let $\Omega$ be a unitary character as in Lemma \ref{YvonneDeGalais} extending 
the central character $\om_\pi$.

For each finite place $v\neq u,w$ of $k_{\so}$, we let $f_v$ denote the
complex function on $\GL_n(k_v)$ supported on $k_v^\times\GL_n(\o_{k_v})$ 
such that $f_v(zk)=\Omega_v(z)$ for all $z\in k_v^\times$ and 
$k\in\GL_n(\o_{k_v})$. 

If $k_{\so}$ is a number field and $v$ is archimedean, 
we choose a smooth complex function $f_v$ on $\GL_n(k_v)$,
compactly supported mod the centre $k_v^\times$, 
such that $f_v(1)=1$ and $f_v(zg)=\Omega_v(z)f_v(g)$ for all~ele\-ments
$z\in k_v^\times$ and $g\in\GL_n(k_v)$. 

We let $f_w$ be a coefficient of $\pi$ such that $f_{w}(1)=1$.

Finally we choose a smooth complex function $f_u$ on $\GL_n(k_u)$,
compactly supported mod the centre, 
such that $f_u(1)=1$ and $f_u(zg)=\Omega_u(z)f_u(g)$ for all 
$z\in k_u^\times$ and $g\in\GL_n(k_u)$,
and of support small enough such that 
\begin{equation*}
f(g^{-1})f(\gamma g)=0,
\quad
\text{for all $g,\gamma\in\GL_n(k)$ such that
$\g\notin k^\times$},
\end{equation*} 
where $f$ is the product of all the $f_v$, 
as in \cite[Appendice 1]{Henn-thesis}, top of p. 148.

We may also assume that $f_v(g^{-1})=\overline{f_v(g)}$ for all $v$ and all 
$g\in\GL_n(k_v)$.

Then 
there is a cuspidal automorphic
representation  $\Pi$ of  $\GL_n(\AA)$
such that $f_v$ acts non-trivially on $\Pi_v$ for each place $v$ of $k_\so$. 
In particular $\Pi_{w}\simeq \pi$ and $\Pi_v$ is unramified at every place 
different from $w$ and $u$. 
\end{proof}

Now let us consider a cuspidal representation $\pi$ of $\G\simeq\GL_n(k_w)$.
The character $\om_\pi|\cdot|_w^{-s}$ is~unitary for some $s\in\CC$,
thus $\pi_{1}=\pi|\det|_w^{-s}$ is unitary.
Lemma \ref{YvonneDeGalais} gives us a
cuspidal automorphic~re\-presentation $\Pi_{1}$ of $\GL_n(\AA)$.
Denoting by $|\cdot|$ the idelic norm on $\AA^\times/k^\times$, 
the cuspidal automorphic representation $\Pi=\Pi_{1}|\det|^s$
has a local component at $w$ isomorphic to $\pi$,
and all its local compotents at $v\neq w,u$ are unramified. 

We recalled in Remark \ref{fabiola} that 
\[\e_\As^\RS(s,\Pi_v,\psi_{\so,v})=\e_\As^\LS(s,\Pi_v,\psi_{\so,v})\] 
when $v$ is split (in particular when $v=u$), 
hence from Theorem \ref{thm unramified epsilon equality} and Theorem 
\ref{global equality}, we get
\[\e_\As^\RS(s,\Pi_{w},\psi_{\so,w})=\e_\As^\LS(s,\Pi_{w},\psi_{\so,w}).\]

Thus we have proved:

\begin{theorem}
\label{thm cuspidal epsilon equality}
Let $\pi$ be a cuspidal representation of $\G=\GL_n(\F)$ and $\psi_\so$ be a 
non-trivial character of $\F_\so$. 
Then \[\e_\As^\RS(s,\pi,\psi_\so)=\e_\As^\LS(s,\pi,\psi_\so).\]
\end{theorem}

When $\pi$ is cuspidal, 
Theorem \ref{thm cuspidal epsilon equality} tells us that
\begin{equation}
\label{laurierrose}
\e_{\As}\left(\frac 1 2,\pi,\psi_\so,\d\right) = 
\omega_\pi(\d)^{n-1}\cdot\lambda(\F/\F_\so,\psi_\so)^{-n(n-1)/2}
\cdot \e_{\As}^{\LS}\left(\frac 1 2,\pi,\psi_\so\right).
\end{equation}
If in addition $\pi$ is distinguished,
then combining this equality
with Theorem \ref{distinguished cuspidal Asairootnumber}
and since we~have 
$\om_\pi(\d^{-1})=\om_\pi(\d)$ as $\pi$ is distinguished and 
$\d^2\in\F_\so^\times$,
we recover \cite[Theorem 1.1]{AnandRoot} for dis\-tin\-guished 
cuspidal representations.

\begin{remark}
\label{coherent} 
When $\pi$ is cuspidal and $\om_{\F/\F_0}$-distinguished, 
we may go in the opposite direction: 
applying \cite[Theorem 1.1]{AnandRoot} together with 
\eqref{laurierrose} gives us 
the value of $\e_\As^\FK(1/2,\pi,\psi_\so,\d)$ 
when $\pi$ is~cus\-pidal and $\om_{\F/\F_0}$-distinguished.
\end{remark}

\begin{remark}
It is shown in \cite[Theorem 1.2]{AnandRoot} that the global Asai root number 
of a $\sigma$-self-dual~cus\-pidal automorphic representation is $1$.
Hence, by 
Theorem \ref{global equality}, the same holds for the Asai factor defined via 
the Rankin--Selberg method. 
Globalizing local distinguished cuspidal representations~as 
local components of distinguished cuspidal automorphic representations 
as in 
\cite{Prasad-SP} or \cite{Gan-Lomeli} and follow\-ing the methods of 
\cite{AnandRoot}, it is possible to prove that 
$\e_\As^\FK(1/2,\pi,\psi_\so,\d)=1$ 
by global methods as well. However our 
proof in this paper has the advantage that it is purely local. 
\end{remark}

\begin{remark}
  In his recent preprint \cite{B-P18}, Beuzart-Plessis extends
  Theorem \ref{thm cuspidal epsilon equality} to all 
generic~re\-pre\-sentations, using a global method as well.
\end{remark}

\appendix

\begin{center}
{\bf Appendix}
\end{center}

\section{Some remarks in positive characteristic}
\label{section positive char}

We use the notation of Section \ref{Notation} and Paragraph \ref{not91}.
In particular, $\G$ denotes the group $\GL_n(\F)$ for some $n\>1$,
and we have defined Asai local $\L$-factors
$\L_\As^{}(s,\pi)$, $\L_\As^{\LS}(s,\pi)$ and $\L_\As^{\Gal}(s,\pi)$
for all generic irreducible representations of $\G$.
We will first prove that these factors are all equal.

\begin{theorem}
\label{manciniA1}
For any generic irreducible complex representation $\pi$ of $\G$,
we have
\begin{equation*}
\L_\As^\FK(s,\pi)=\L_\As^{\Gal}(s,\pi)=\L_\As^\LS(s,\pi).
\end{equation*}
\end{theorem}

\begin{proof}
When $\F$ has characteristic $0$, this follows from 
\cite{MatringeConjectures,MatringeManuscripta,MatringeGeneric} 
and \cite{AnandRajan}, \cite{Henn-L}.

Now we notice that the local results in \cite[Section 3]{Kable} 
hold in positive characteristic,
and the global results of \cite[Section 4]{Kable} also
hold in positive characteristic though written in characteristic $0$~only.
Indeed they refer to \cite{Flicker} which~is for any global field. 
The main point is that \cite[Theorem 5]{Kable} is true for function fields, and its 
proof slightly simplifies because of the absence of archimedean places. 
This implies that,
when $\F$ has characteristic $p\neq2$, the equality 
\begin{equation*}
\L_\As^\FK(s,\pi)=\L_\As^\LS(s,\pi)
\end{equation*} 
holds for any discrete series representation:
the ingredients which make the proof of 
\cite[Theorem 1.6]{AnandRajan} work are then all available. 
Once again, 
notice 
that its proof simplifies in the positive characteristic case as there are no 
archimedean places to worry about. 

Now notice that \cite[Theorem 3.1]{MatringeManuscripta}
holds when $\F$ has characteristic $p$. 
Indeed its proof relies on \cite[Theorem 3.1.2]{Ok} which is for any 
non-archimedean local field of odd residual characteristic. 
Then the classification of generic distinguished representations in 
\cite{MatringeGeneric} relies only on the geometric lemma of 
Bernstein--Zelevinsky, the Bernstein--Zelevinsky explicit description of 
discrete series representations and their Jacquet 
modules, and the fact that a distinguished irreducible representation of $\G$ 
is $\sigma$-self-dual. All the aforementioned results are true in positive 
characteristic (different from $2$ for the latter) hence the classification of 
\cite{MatringeGeneric} still holds when $\F$ has characteristic $p$. 

Finally, the Cogdell--Piatetski-Shapiro method 
of derivatives to analyze the exceptional poles used in 
\cite{MatringeConjectures} still works in positive characteristic as well (for 
example the original paper \cite{CPS} is written~in arbitrary characteristic) 
hence 
the inductivity relation of $\L_\As^\FK(s,\pi)$ for any {generic} irreducible 
repre\-sen\-tation (see \cite[Proposition 4.22]{MatringeConjectures}) follows. 
All in all, when $\F$ has characteristic $p$, we have
\begin{equation*}
\L_\As^\FK(s,\pi)=\L_\As^{\Gal}(s,\pi)
\end{equation*} 
for any {generic} irreducible representation.
\end{proof}

We now prove that the dichotomy theorem of \cite{Kable} and \cite{MR2063106} 
holds when $\F$ has characteristic $p$.

\begin{theorem}
\label{global proof of dichotomy}
Let $\pi$ be a $\s$-self-dual discrete series representation of $\G$.
Then $\pi$ is either distinguished or $\om_{\F/\F_\so}$-distinguished, 
but not both. 
\end{theorem}

\begin{proof}
When $\F$ has characteristic $0$, this is 
\cite[Theorem 4]{Kable} and \cite[Corollary 1.6]{MR2063106}.

Assume that $\F$ has characteristic $p$.
If $\pi$ is a discrete series representation of $\G$ and $\omega$ is 
a cha\-rac\-ter of $\mult \F$ extending $\omega_{\F/\F_0}$, 
the equality 
\[\L(s,\pi,\pi^\sigma)=\L_\As(s,\pi)\cdot\L_\As(s,\omega\otimes \pi)\] 
becomes a 
consequence of the relation 
\[\Ind'_{\F/\F_\so}\left(\rho(\pi)\otimes \rho(\pi)^\s\right)
=\M'_{\F/\F_\so}(\rho(\pi))\oplus 
\om_{\F/\F_\so} \M'_{\F/\F_\so}(\rho(\pi)).\] 
Then, if $\pi$ is $\s$-self-dual, 
the Rankin--Selberg local $\L$-factor 
$\L^\RS(s,\pi,\pi^\sigma)$ has a simple pole at $s=0$ accor\-ding to 
Proposition $8.1$ and Theorem $8.2$ of \cite{JPSS}, hence 
either $\L_\As(s,\pi)$ or $\L_\As(s,\omega\otimes \pi)$ has a pole at $s=0$ 
but not both. Finally 
one concludes appealing to \cite[Proposition 3.4]{MatringeManuscripta} (the 
paper \cite{MatringeManuscripta} is 
valid for $\F$ of characteristic $p$ 
as it only relies on the paper \cite{Ok} which is true in 
this setting). 
\end{proof}

Notice that \cite{VS} gives a purely local proof of Theorem 
\ref{global proof of dichotomy}
when $\pi$ is cuspidal. 

\section{Modular versions of results by Bruhat, Kable and Ok}
\label{AppB}

In this appendix, which culminates in \ref{AB3},
we generalize three results which were known for complex representations 
only.

In \ref{parB1}, we generalize a result of Bruhat on equivariant distributions 
to the case of smooth represen\-tations of a locally profinite group with
coefficients in an (almost) arbitrary commutative ring. 
For complex representations, a formal proof can be found in an 
unpublished~ver\-sion of Rodier's paper \cite{RodierWhittaker} on 
Whittaker models. 
The result is also stated in \cite{RodierWhittaker} as Theorem 4
and refers to Bruhat's~thesis as a reference.

\subsection{A modular version of a result of Bruhat on equivariant distributions}
\label{parB1}

In this subsection,
$\GGG$ is a locally profinite group,
$\HHH$ is a closed subgroup of $\GGG$ and 
$\RRR$ is~a commutative ring with unit.
We assume that 
there is a right invariant $\RRR$-valued measure $dh$ on $\HHH$
giving measure $1$ to some compact open subgroup of~$\HHH$.
According to \cite[I.2.4]{Vig96}, this is equivalent~to assuming 
that $\HHH$ has a compact open subgroup whose pro-order is 
invertible in $\RRR$.

Let $\ttt$ be a smooth representation of $\HHH$ on an $\RRR$-module $\V$. 
Write $\Cc^\infty_c(\GGG,\V)$ 
for the space of locally constant, compactly supported 
functions on $\GGG$ with values in $\V$, 
which canonically identifies with $\Cc^\infty_c(\GGG,\RRR)\otimes\V$,
and write $\ind_\HHH^\GGG(\ttt)$ for the compact induction of $\ttt$ to $\GGG$.
Both are equipped with an action of $\GGG$ by right translations, 
denoted $g \cdot f : x \mapsto f(xg)$ for all $g,x\in\GGG$.

We denote by $\d=\d_{\HHH}$ the character of $\HHH$ such that 
$d(xh)=\d(x)dh$ for all $x\in\HHH$, that is
\begin{equation*}
\int_\HHH f(xh)\ dh = \d(x)^{-1}\cdot\int_\HHH f(h)\ dh
\end{equation*}
for all $f\in\Cc_{{\rm c}}^\infty(\HHH,\RRR)$ and $x\in\HHH$.
We will use the fact that
\begin{equation*}
\int_\HHH f(h^{-1})\ dh = \int_\HHH \d(h)^{-1}f(h)\ dh 
\end{equation*}
as well as the fact that 
the restriction of $\d$ to any~com\-pact open subgroup of $\HHH$ is trivial. 

We start with the following lemma, 
proved by Rodier in \cite[Proposition 1]{RodierWhittaker}. 
Unlike Rodier, 
we use unnormalized induction.

\begin{lemma}
\label{projection}
\label{HautJura}
\begin{enumerate}
\item
For all $f\in\Cc^\infty_c(\GGG,\V)$, the function
\begin{equation*}
\pp(f) : g \mapsto \int\limits_\HHH \tau(h^{-1}) f(hg)\ dh
\end{equation*}
is in $\ind^\GGG_\HHH(\tau)$.
\item
The map $\pp:\Cc^\infty_c(\GGG,\V)\to\ind^\GGG_\HHH(\tau)$
defined in {\rm (i)} is surjective.
\item
The map $\pp$ is $\GGG$-equivariant and,
for all $f\in\Cc^\infty_c(\GGG,\V)$ and $x\in\HHH$,
one has
\begin{equation*}
\pp(f_x) = \pp(\d(x)^{-1}\tau(x)f)
\end{equation*}
where $f_x$ is the function $g\mapsto f(xg)$.
\end{enumerate}
\end{lemma}

\begin{proof}
First, we prove (i).
For all $g\in\GGG$, the integral 
\begin{equation*}
\int\limits_\HHH \tau(h^{-1}) f(hg)\ dh
\end{equation*}
is a finite sum since $f$ is smooth and compactly supported.
For any $x\in\HHH$ and $g\in\GGG$, one has 
\begin{equation*}
\pp(f)(xg) 
= \int\limits_\HHH \tau(h^{-1}) f(hxg)\ dh
= \int\limits_\HHH \tau(xh^{-1}) f(hg)\ dh
= \tau(x)\left(\pp(f)(g)\right)
\end{equation*}
since $dh$ is right invariant.
It follows that $\pp(f)$ is in $\ind^\GGG_\HHH(\tau)$.

Let us prove (ii).
Given $v\in\V$ and $g\in\GGG$,
there is an open subgroup $\J$ of $\GGG$
such that $\HHH\cap g\J g^{-1}$~lea\-ves $v$ in\-va\-riant 
and its measure is invertible in $\RRR$.
Let $\phi:\GGG\to\V$ be the function supported in $\HHH g\J$ 
and defined by $\phi(hgj)=\tau(h) v$ for all $h\in\HHH$ and $j\in\J$.
It belongs to $\ind^\GGG_\HHH(\tau)$,
and the~li\-near span of all such maps is the full induced representation, 
hence it suffices to show that such a $\phi$ is in the image of 
$\pp$ to prove that $\pp$ is surjective.
Let $f:\GGG\to\V$ be the function supported in $g\J$ and defined
for all $x\in g\J$ by 
\begin{equation*}
f(x)=\frac{1}{dh(\HHH\cap g\J g^{-1})} \cdot v.
\end{equation*}
One checks immediately that $f\in\Cc^\infty_c(\GGG,\V)$ and $\pp(f)=\phi$.

Finally, let us prove (iii).
One has
\begin{equation*}
\pp(f_x) 
= \int\limits_\HHH \tau(h^{-1}) f(xhg)\ dh
= \int\limits_\HHH \tau(h^{-1}x) f(hg)\d(x)^{-1}\ dh
= \int\limits_\HHH \tau(h^{-1})\left(\d(x)^{-1}\tau(x) f(hg)\right)\ dh
\end{equation*}
which is indeed equal to $\pp(\d(x)^{-1}\tau(x)f)$.
\end{proof}

\begin{remark}
One can reformulate Lemma \ref{HautJura}(iii) as follows.
The space $\Cc^\infty_c(\GGG,\V)$ has an action of $\GGG$ by 
right translations,
as well as an action of $\HHH$ defined by
\begin{equation}
\label{actionH}
x \odot f : g \mapsto \d(x)^{-1}\tau(x) f(x^{-1}g).
\end{equation}
for all $x\in\HHH$ and $f\in\Cc^\infty_c(\GGG,\V)$. 
Then the map $\pp$ is $\GGG$-equivariant and $\HHH$-invariant.
\end{remark}

Recall that,
given $f\in\Cc^\infty_c(\GGG,\V)$ and $x\in\HHH$,
we write $f_x$ for the function $g\mapsto f(xg)$.

\begin{lemma}
\label{bruhat}
\begin{enumerate}
\item
Let $\Ll$ be a linear form on $\Cc^\infty_c(\GGG,\V)$ such that 
\begin{equation}
\label{lapindujura}
\Ll(f_x) = \Ll(\d(x)^{-1}\tau(x)f)
\end{equation}
for all $f\in\Cc^\infty_c(\GGG,\V)$ and $x\in\HHH$.
Then there is a unique linear form $\Ll'$ on $\ind^\GGG_\HHH(\tau)$ such that 
$\Ll=\Ll'\circ \pp$.
\item
The map $\Ll'\mapsto\Ll'\circ \pp$ is an isomorphism of $\RRR$-modules:
\begin{equation*}
  \Hom_{\RRR}(\ind^\GGG_\HHH(\tau),\RRR) \simeq
  \Hom_{\HHH}(\Cc^\infty_c(\GGG,\V),\RRR) 
\end{equation*}
where the right hand side is made of all linear forms on
$\Cc^\infty_c(\GGG,\V)$ satisfying \eqref{lapindujura}.
\end{enumerate}
\end{lemma}

\begin{proof}
First, we notice that, if $\Ll'$ is a linear form on $\ind^\GGG_\HHH(\tau)$,
then $\Ll=\Ll'\circ\pp$ is a linear form on $\Cc^\infty_c(\GGG,\V)$
satisfying \eqref{lapindujura}, by Lemma \ref{projection}. 

Now let $\Ll$ be a linear form on $\Cc^\infty_c(\GGG,\V)$
satisfying \eqref{lapindujura}.
We denote by $\pp'$ the natural projection from 
$\Cc_c^{\infty}(\GGG,\RRR)$ onto $\Cc_c^{\infty}(\HHH\backslash\GGG,\RRR)$
defined by 
\begin{equation*}
\pp'(\phi)(g)= \int_\HHH \phi(hg)\ dh.
\end{equation*}
Note that $\Cc_c^{\infty}(\HHH\backslash\GGG,\RRR)$ is the compact induction 
$\ind^\GGG_\HHH(\one)$ of the trivial $\RRR$-character of $\HHH$,
thus $\pp'$ is a particular case of the projection given by Lemma \ref{projection}
when one chooses for $\ttt$ the trivial character.
Let us fix $f\in\Cc_c^\infty(\GGG,\V)$ and $\phi\in \Cc_c^\infty(\GGG,\RRR)$, 
and notice that both $\phi \pp(f)$ and 
$\pp'(\phi)f$ considered as functions on $\GGG$ are in $\Cc_c^\infty(\GGG,\V)$.
We are going to prove that
\begin{equation*}
\Ll(\phi \pp(f))=\Ll(\pp'(\phi) f).
\end{equation*}
Let $\K$ be a compact open subgroup of $\HHH$ leaving $f$ and $\phi$ fixed 
under left translations, and acting~tri\-vially on all vectors in the image of 
$f$ (which is possible for the linear span of the image of $f$~in $\V$ is finite 
dimensional). 
One defines the compact subset: 
\[\C=\K\left[\left(\supp(f)\supp(\phi)^{-1}\cup 
\supp(\phi)\supp(f)^{-1}\right)\cap\HHH\right]\K\] 
of $\HHH$.
It is stable under $x\mapsto x^{-1}$ and $\K$-bi-invariant.
We claim that there is a compact open~sub\-group $\U$ of $\K$ such that: 
\begin{enumerate}
\item
one has $x\U x^{-1}\subseteq\K$ for all $x\in\C$, and
\item
the pro-order of $\U$ is invertible in $\RRR$.
\end{enumerate}
Indeed, consider the continuous function $\mu:\HHH\times\HHH\to\HHH$ 
defined by $(x,y)\mapsto xyx^{-1}$.
The preimage $\mu^{-1}(\K)$ is an open subset of $\HHH\times\HHH$ containing 
$\C\times\{1\}$.
For all $x\in\C$, there are an open neighbour\-hood $\B_x$ of $x$ 
and a compact open subgroup $\U_x$ of $\HHH$ such that $\B_x\times\U_x$ 
is contained in $\mu^{-1}(\K)$.~As 
$\C$ is compact, $\C\times\{1\}$ is contained in the union of finitely many 
$\B_x\times\U_x$.
The intersection $\U$ of these finitely many $\U_x$ satisfies (i).
To get (ii), one chooses a small enough open subgroup of $\U$.

We are now in a position to prove the sought equality. 
First notice that, for all $g\in\GGG$, 
the function $h\mapsto\phi(g) \tau(h)^{-1} f(hg)$
is constant on any $\U$-double coset in $\C$.
Let $\A$ be a set of representatives of $\U\backslash\C/\U$.
Writing $k(a)=dh(\U a\U)$ for all $a\in\A$, 
and noticing that 
$\phi(g) \tau(h)^{-1} f(hg)$ vanishes when $h\notin\supp(f)\supp(\phi)^{-1}$,
we get
\begin{equation}
\label{TonyBuddenbrook}
\phi \pp(f)(g) 
= \int\limits_{\C} \phi(g) \tau(h)^{-1} f(hg)\ dh 
= \sum\limits_{ a} k(a) \phi(g) \tau( a^{-1}) f_{ a}(g)
\end{equation}
for all $g\in\GGG$, where $a$ ranges over $\A$.

Similarly, using the formula 
\begin{equation*}
\pp'(\phi)(g)
= \int_\HHH \phi(hg)\ dh
= \int_\HHH \d(h)^{-1}\phi(h^{-1}g)\ dh
\end{equation*}
for all $g\in\GGG$, one has
\begin{equation*}
\pp'(\phi) f 
= \sum\limits_{ a} k(a) \d( a)^{-1}\phi_{ a^{-1}} f.
\end{equation*}
Hence 
\begin{eqnarray*}
\Ll(\pp'(\phi) f)
&=& \sum\limits_{ a} k(a) \Ll(\d( a)^{-1}\phi_{a^{-1}} f) \\
&=& \sum_{ a} k(a) \Ll(\d( a)^{-1}(\phi f_{ a})_{ a^{-1}}) \\
&=& \sum_{ a} k(a) \Ll(\tau( a)^{-1}\phi f_{ a})
\end{eqnarray*}
which is equal to $\Ll(\phi \pp(f))$ by \eqref{TonyBuddenbrook}.

Now, by Lemma \ref{projection}(ii), 
there is $\phi\in\Cc_c^\infty(\GGG,\RRR)$ such that $\pp'(\phi)$ 
is~equal to $1$ on $\supp(f)$. 
For such a $\phi$, one has $\Ll(f)=\Ll(\phi \pp(f))$. 
From this latter equality, we deduce that the kernel of $\pp$
is contained in that of $\Ll$, which proves Lemma \ref{bruhat}. 
\end{proof}

Now let $\HHH'$ be another closed subgroup of $\GGG$ and
$\chi$ be an $\RRR$-character of $\HHH'$.
Let $\Ll$ be a linear form as in Lemma \ref{bruhat}
and suppose that one has 
\begin{equation}
\label{equation right-equiv-0} 
\Ll(y\cdot f)=\chi(y)\Ll(f),
\quad
f\in\Cc_c^\infty(\GGG,\V),
\quad
y\in\HHH'.
\end{equation}
Then, by uniqueness of the linear form $\Ll'$ 
corresponding to $\Ll$, one has
\begin{equation}
\label{ThomasBuddenbrook}
\Ll'(y\cdot\phi)=\chi(y)\Ll'(\phi),
\quad
\phi\in\ind_\HHH^\GGG(\tau).
\end{equation}
We arrive to the following result which we 
shall use many times hereunder. 

\begin{corollary}
\label{corollary equiv linear forms on iduced reps}
The map $\Ll'\mapsto \Ll'\circ p$ is an isomorphism of $\RRR$-modules
between:
\begin{enumerate}
\item
linear forms $\Ll'$ on $\ind_\HHH^\GGG(\tau)$ satisfying \eqref{ThomasBuddenbrook}, 
and
\item
linear forms $\Ll$ on $\Cc_c^\infty(\GGG,\V)$ satisfying
\eqref{lapindujura} and \eqref{equation right-equiv-0}.
\end{enumerate}
\end{corollary}

\subsection{A modular version of a result of Kable}
\label{AB2}

In this subsection,
we generalize a result of Kable (\cite[Proposition 1]{Kable}) 
to the case of smooth~re\-pre\-sentations of $\GL_n(\F)$ with coefficients in a 
commutative ring with sufficiently many roots of unity of $p$-power order 
and in which $p$ is invertible.
In fact, we expand and simplify Kable's proof, 
appealing to Theorem \ref{bruhat} when he appeals to Warner~\cite{warner}.

We go back to the main notation of the paper:
$\G$ is the group $\GL_n(\F)$ where~$\F/\F_\so$ is a~quadra\-tic extension,
$\s$ is the Galois involution and
$\P$ is the mirabolic sub\-group of~$\G$.
We also write $\G'$ for the group $\GL_{n-1}(\F)$
considered as a subgroup of $\G$ in the
usual way, and $\P'$ for the mirabolic sub\-group of~$\G'$.
Denoting by $\U$ the unipotent radical of $\P$,
one has the semi-direct product decomposition $\P=\G' \U$.

We also assume that $\FC$ is a commutative ring with unit, such that $p$
is invertible in $\FC$ and there is a non-trivial 
$\FC$-character $\psi_\so$ of $\F_\so$.

Let $\psiu$ be the restriction to $\U$ 
of the standard $\s$-self-dual non-degenerate character $\psi$ of $\N$
defined by \eqref{Muller} for some non-zero $\d\in\F^\times$ of trace $0$.

Since $p$ is invertible in $\FC$ and $\G$ is locally pro-$p$,
there is a non-zero right invariant measure $dh$~on $\P'\U$ with values in $\FC$,
giving measure $1$ to some compact open subgroup.
Given any smooth~repre\-sentation $\tau$ of $\P'$ on an $\FC$-module $\V$,
we denote by $\tau\otimes\psiu$ the~re\-pre\-sentation of $\P'\U$ defined by 
\begin{equation*}
\tau\otimes\psiu : xu \mapsto \psiu(u)\tau(x)
\end{equation*} 
for $x\in\P'$ and $u\in\U$.
Following \cite{BZ76}, we set 
\[\Phi^+(\tau)= \ind_{\P'\U}^\P(\tau\otimes\psiu).\] 
This defines a functor from smooth $\FC$-representations of $\P'$ to smooth 
$\FC$-representations of $\P$.
Note that, since we use the unnormalized version of the functor $\Phi^+$ as in 
\cite{BZ76}, we do not have to worry about the existence of a square root of 
$q$ in $\FC$. 

We will write $\nu$ and $\nu_\so$ for the unramified characters
$g\mapsto|\det(g)|$ and $g\mapsto|\det(g)|_\so$,
respectively. 

\begin{proposition}
\label{kable's result}
For any smooth $\FC$-representation $\tau$ of $\P'$
and any character $\chi$ of $\P^\s$, one has an isomorphism:
\begin{equation*}
\Hom_{\P^\sigma}(\Phi^+(\tau),\chi)\simeq \Hom_{\P'^\s}(\tau ,\chi\nu_\so)
\end{equation*}
of $\FC$-modules.
\end{proposition}

\begin{proof}
First, we apply Corollary \ref{corollary equiv linear forms on iduced reps}
with $\GGG=\P$, $\HHH=\P'\U$ and $\ttt=\tau\otimes\psiu$.
Since the character $\d_{\P'\U}$ associated with $\P'\U$ is equal to 
$\nu^{2}$, we get an isomorphism of $\FC$-modules
from $\Hom_{\P^\s}(\Phi^+(\tau),\chi)$ to the space 
of all linear forms $\T$ on 
$\Cc_c^\infty(\P,\V)$ such that:
\begin{eqnarray}
\label{equation right-equiv} 
\T(g_\so\cdot f) &=& \chi(g'_\so)\cdot\T(f), \\ 
\label{equation right-equiv bis} 
\T(u_\so\cdot f) &=& \T(f), \\ 
\label{equation left-equiv} 
\T(f_{g}) &=& \T(\nu(g)^{-2}\tau(g)f), \\ 
\label{equation left-equiv bis} 
\T(f_u) &=& \psiu(u)\T(f), 
\end{eqnarray}
for all $g_\so\in\G'^\s$, $u_\so\in\U^\s$, 
$g\in\P'$, $u\in\U$ and $f\in\Cc_c^\infty(\P,\V)$.
We now consider the $\FC$-linear map
$\Aa$ from $\Cc_c^\infty(\P,\V)$ to $\Cc^\infty_c(\P'\G'^\s,\V)$
defined by
\begin{equation*}
\Aa (f) : x \mapsto \int\limits_\U \psiu^{-1}(u)f(ux)\ du
\end{equation*}
for all $f\in\Cc_c^\infty(\P,\V)$ and $x\in\P'\G'^\s$,
where $du$ is some right invariant measure on $\U$.
It is obtained by composing the map
\begin{equation*}
f \mapsto \left(x\mapsto\int\limits_\U \psiu^{-1}(u)f(ux)\ du\right),
\end{equation*}
with $x\in\G'$,
with the res\-triction map from $\Cc_c^\infty(\G',\V)$
to $\Cc_c^\infty(\P'\G'^\s,\V)$.
The former is surjective since $\Cc_c^\infty(\P,\V)$ canonically identifies with
$\Cc_c^\infty(\U,\FC)\otimes \Cc_c^\infty(\G',\V)$,
and so is the latter since $\P'\G'^\s$ is a closed subset of $\G'$ 
(for it is made of all matrices in $\G'$ the last row of which is fixed by $\s$).
Thus the adjoint map
\begin{equation}
\label{ChristianBuddenbrook}
\Aa^* : \Hom_{\FC}(\Cc_c^\infty(\P'\G'^\s,\V),\FC)
\to
\Hom_{\FC}(\Cc_c^\infty(\P,\V),\FC) 
\end{equation}
is injective.
We claim that its image is the space of all linear forms $\T$
satisfying \eqref{equation right-equiv bis} and \eqref{equation left-equiv bis}.

First, let us check that the image of $\Aa^*$ is contained in that space.
Indeed, given a linear form $\SS$ in the left hand side of
\eqref{ChristianBuddenbrook}
and $f\in \Cc_c^\infty(\P,\V)$, one has
$\Aa(f_u)=\psiu(u)\Aa(f)$
for all $u\in\U$ and
\begin{eqnarray*}
  \Aa(u_\so\cdot f)(x)
  &=& \int\limits_{\U} \psiu^{-1}(u) f(uxu_\so)\ du \\
  &=& \psiu^{-1}(xu_\so x^{-1}) \int\limits_{\U} \psiu^{-1}(u) f(ux)\ du 
\end{eqnarray*}
for all $x\in\P'\G'^\s$ and $u_\so\in\U^\s$.
Since $\psiu(xu_\so x^{-1})=1$ for all $x\in\P'\G'^\s$,
this is equal to $\Aa(f)(x)$ as expected.
(Note that we used the fact that $\psiu$ is trivial on $\U^\s$.)
To prove surjectivity, we follow the second paragraph of the proof of 
\cite[Proposition 1]{Kable} at p.~797.

Now consider a linear form $\SS$ on $\Cc_c^\infty(\P'\G'^\s,\V)$.
We check immediately that
$\Aa^*(\SS)=\SS\circ\Aa$ satisfies \eqref{equation right-equiv}
if and only if
\begin{equation}
\label{equation right-equiv 2} 
\SS(g'_\so \cdot f) = \chi(g'_\so)\cdot\SS(f)
\end{equation}
for all $f\in\Cc_c^\infty(\P'\G'^\s,\V)$ and $g'_\so\in \G'^\s$.
On the other hand,
$\SS\circ\Aa$ satisfies \eqref{equation left-equiv} if and only if
\begin{equation}
\label{equation left-equiv 2} 
\SS(f_{p'}) = \SS(\nu^{-1}(p') \tau(p')f)
\end{equation} 
for all $f\in\Cc_c^\infty(\P'\G'^\s,\V)$
and $p'\in \P'$.
Indeed, notice that
\begin{eqnarray*}
  \Aa(f_{p'})(x)
  &=& \int\limits_{\U} \psiu^{-1}(u) f(p'ux)\ du \\
  &=& \int\limits_{\U} \psiu^{-1}(p'up'^{-1}) f(up'x) \nu^{-1}(p')\  du \\
  &=& \int\limits_{\U} \psiu^{-1}(u) f(up'x) \nu^{-1}(p')\  du   
\end{eqnarray*}
for all $f\in\Cc_c^\infty(\P,\V)$, $x\in\P'\G'^\s$ and $p'\in\P'$,
where the second equality follows from the fact that the character
$\d_{\P'}$ associated with $\P'$ is $\nu$,
and the third one from the fact that $\P'$ normalizes $\psiu$.
It follows that $\Aa^*$ induces an isomorphism of $\FC$-modules
between:
\begin{enumerate}
\item
the space of linear forms $\T$ on $\Cc_c^\infty(\P,\V)$ satisfying 
\eqref{equation right-equiv}, \eqref{equation right-equiv bis}, 
\eqref{equation left-equiv} and \eqref{equation left-equiv bis}, 
and
\item
the space of linear forms $\SS$ on $\Cc_c^\infty(\P'\G'^\s,\V)$ satisfying
\eqref{equation right-equiv 2} and \eqref{equation left-equiv 2}.
\end{enumerate}
Now consider the map $(x,y)\mapsto x^{-1}y$
from $\P'\times\G'^\s$ onto $\P'\G'^\s$.
It identifies $\P'\G'^\s$ with the homo\-geneous space
$\P'^\s\backslash(\P'\times\G'^\s)$
where $\P'^\s=\P'\cap\G'^\s$ is diagonally embedded in
$\P'\times\G'^\s$.
This thus identifies the space $\Cc_c^\infty(\P'\G'^\s,\V)$
with the compact induction
$\ind^{\P'\cap\G'^\s}_{\P'^\s}(\one\otimes\V)$
where $\one\otimes\V$ denotes the trivial representation
of $\P'^\s$ on $\V$.
Namely, $f\in\Cc_c^\infty(\P'\G'^\s,\V)$
identifies with the function $\phi$ on
$\P'\cap\G'^\s$ defined by
$\phi(x,y)=f(x^{-1}y)$ for $(x,y)\in\P'\cap\G'^\s$.
This thus gives us an isomorphism of $\FC$-modules
between:
\begin{enumerate}
\item
the space of linear forms $\SS$ on $\Cc_c^\infty(\P'\G'^\s,\V)$ satisfying 
\eqref{equation right-equiv 2} and \eqref{equation left-equiv 2}, and
\item
the space of linear forms $\Q$ on $\ind^{\P'\cap\G'^\s}_{\P'^\s}(\one\otimes\V)$ such that 
\begin{equation}
\label{equation right-equiv 3} 
\Q((p',g'_\so)\cdot\phi)=\chi(g'_\so)\cdot\Q(\nu(p')\tau(p'^{-1})\phi) 
\end{equation}
for all $\phi\in\ind^{\P'\cap\G'^\s}_{\P'^\s}(\one\otimes\V)$
and $(p',g'_\so)\in \P' \times \G'^\s$. 
\end{enumerate}
We now apply Corollary \ref{corollary equiv linear forms on iduced reps} again,
with $\GGG=\P'\times\G'^\s$, $\HHH=\P'^\s$ and $\ttt=\one\otimes\V$.
Since the character $\d_{\P'^\s}$ associated with $\P'^\s$ is equal to 
$\nu_\so$, we get an isomorphism of $\FC$-modules between:
\begin{enumerate}
\item
the space of linear forms $\Q$ on $\ind^{\P'\cap\G'^\s}_{\P'^\s}(\one\otimes\V)$ 
satisfying \eqref{equation right-equiv 3}, and
\item
the space of linear forms $\L$ on $\Cc_c^\infty(\P'\times\G'^\s,\V)$
such that:
\begin{eqnarray}
\label{equation right-equiv 4} 
\L((p',g'_\so)\cdot\phi) &=& \chi(g'_\so)\cdot\L(\nu(p')\tau(p'^{-1})\phi), \\ 
\label{equation left-equiv 4} 
\L(\phi_{p'_\so}) &=& \nu_\so^{-1}(p'_\so)\cdot\L(\phi), 
\end{eqnarray}
for all $\phi\in\Cc_c^\infty(\P'\times\G'^\s,\V)$,
$p'_\so\in\P'^\s$ and $(p',g'_\so)\in\P'\times\G'^\s$.
\end{enumerate}
For $\phi$ and $\L$ as above,
we define $\phi^\vee:(x,y)\mapsto\phi(x^{-1},y^{-1})$
and $\M(\phi)=\L(\phi^\vee)$.
This defines an iso\-morphism of $\FC$-modules between:
\begin{enumerate}
\item
the space of linear forms $\L$ on $\Cc_c^\infty(\P'\times\G'^\s,\V)$ 
satisfying \eqref{equation right-equiv 4} and 
\eqref{equation left-equiv 4}, and
\item
the space of linear forms $\M$ on $\Cc_c^\infty(\P'\times\G'^\s,\V)$
such that:
\begin{eqnarray}
\label{equation right-equiv 5}
\M(p'_\so\cdot\phi) &=& \nu_\so(p'_\so)\cdot\M(\phi), \\ 
\label{equation left-equiv 5}
\M(\phi_{(p',g'_\so)}) &=& 
\M(\chi^{-1}(g'_\so )\nu^{-1}(p')\tau(p')\phi), 
\end{eqnarray}
for all $\phi\in\Cc_c^\infty(\P'\times\G'^\s,\V)$,
$p'_\so\in\P'^\s$ and $(p',g'_\so)\in\P'\times\G'^\s$.
\end{enumerate}
We now apply Corollary 
\ref{corollary equiv linear forms on iduced reps} again,
with $\GGG=\HHH=\P'\times\G'^\s$ and $\ttt=\tau\otimes\chi^{-1}$.
Since the character $\d_{\P'\times\G'^\s}$ associated with
$\P'\times\G'^\s$ is equal to
$\nu^{-1}\otimes 1$, we get an isomorphism of
$\FC$-modules between:
\begin{enumerate}
\item
the space of linear forms $\M$ on $\Cc_c^\infty(\P'\times\G'^\s,\V)$
satisfying \eqref{equation right-equiv 5} and 
\eqref{equation left-equiv 5}, and
\item
the space of linear forms $t$ on
$\ind^{\P'\times\G'^\s}_{\P'\times\G'^\s}(\tau\otimes\chi^{-1})$
such that
\begin{equation*}
t(p'_\so\cdot\varphi) = \nu_\so(p'_\so)\cdot\varphi
\end{equation*}
for all 
$\varphi\in\ind^{\P'\times\G'^\s}_{\P'\times\G'^\s}(\tau\otimes\chi^{-1})$ 
and $p'_\so\in\P'^\s$. 
\end{enumerate}
Finally, one verifies that the map $\varphi\mapsto\varphi(1,1)$
from $\ind^{\P'\times\G'^\s}_{\P'\times\G'^\s}(\tau\otimes\chi^{-1})$
to $\V$ induces an     isomorphism of $\FC$-modules between the space of
linear forms $t$ as above and $\Hom_{\P'^\s}(\tau,\chi\nu_\so)$,
which ends the proof of the proposition.
\end{proof}

\subsection{A modular version of a result of Ok for cuspidal representations}
\label{AB3}

In this subsection, 
we generalize a result of Ok 
(\cite[Theorem~3.1.2]{Ok})
on irreducible complex repre\-sentations of $\G=\GL_n(\F)$.
More precisely, 
using Proposition \ref{kable's result}, we prove it for any 
\textit{cuspidal}~re\-presentation of $\G$ with coefficients in an 
algebraically closed field of characteristic different from~$p$.

In this subsection, $\FC$ is an algebraically closed field of characteristic 
different from $p$.

\begin{proposition}
\label{Ok in general}
Let $\pi$ be a cuspidal representation of $\G$ with coefficients in $\FC$.
Then the space $\Hom_{\P^\s}(\pi,\one)$ has dimension $1$.
If in addition $\pi$ is $\H$-distinguished, 
then we have 
\[\Hom_{\P^\s}(\pi,\one)=\Hom_{\G^\s}(\pi,\one).\]
\end{proposition}

\begin{proof}
By \cite{BZ76} and \cite[III.1]{Vig96}, the restriction of $\pi$ to $\P$ is 
isomorphic to $\ind^\P_\N(\psi)$, 
where $\psi$ is the~stan\-dard $\s$-self-dual non-degenerate character 
of $\N$ which has been fixed at the beginning of \ref{AB2}.
This induced representation can be written $(\Phi^+)^{n-1}\Psi^+(\one)$,
where $\one$ denotes the trivial character of the trivial group, 
$\Psi^+(\one)$ is the trivial character of the (trivial) mirabolic subgroup 
$\P_1(\F)$ and $\Phi^+$ is the functor which has been defined in \ref{AB2}.
Applying $n-1$ times Proposition \ref{kable's result}, we get the expected result. 
\end{proof}


{\small\textsc{Department of Mathematics, Indian Institute of Technology Bombay, 
Mumbai 400076, India}, 
\textit{E-mail:} \verb!anand@math.iitb.ac.in!}

{\small\textsc{Department of Mathematics, Imperial College London, 
London SW7 2AZ, United Kingdom}, 
\textit{E-mail:} \verb!robkurinczuk@gmail.com!}

{\small\textsc{Université de Poitiers, 
Laboratoire de Mathématiques et Applications, 
86962, Futuroscope Chasseneuil Cedex, France},
\textit{E-mail:} \verb!nadir.matringe@math.univ-poitiers.fr!}

{\small\textsc{Laboratoire de Mathémati\-ques de Versailles, 
UVSQ, CNRS, Université Paris-Saclay, 78035, Versailles, France},
\textit{E-mail:} \verb!vincent.secherre@math.uvsq.fr!}

{\small\textsc{School of Mathematics, 
University of East Anglia, Norwich NR4 7TJ, United Kingdom},
\textit{E-mail:} \verb!Shaun.Stevens@uea.ac.uk!}

\end{document}